\documentclass{article}
\usepackage{authblk}
\usepackage{fancyhdr}
\usepackage{amscd}
\usepackage[dvipdfmx]{graphicx}
\usepackage{amsmath}
\usepackage{amssymb}
\usepackage{amsthm}
\usepackage{mathrsfs}
\newtheorem{dfn}{Definition}[section]
\newtheorem{thm}[dfn]{Theorem}
\newtheorem{prop}[dfn]{Proposition}
\newtheorem{lem}[dfn]{Lemma}

\newtheorem{rem}[dfn]{Remark}

\newtheorem{ass}[dfn]{Assumption}

\numberwithin{equation}{section}

\newcommand{\KMi}[1]{{\color{black} #1}}

\usepackage{color}
\begin{document}

\title{A simple description of blow-up solutions through dynamics at infinity in nonautonomous ODEs 
}

\author[1,2]{Kaname Matsue\thanks{(Corresponding author, {\tt kmatsue@imi.kyushu-u.ac.jp})}}

\affil[1]{
\normalsize{
Institute of Mathematics for Industry, Kyushu University, Fukuoka 819-0395, Japan
}
}
\affil[2]{
\normalsize{
International Institute for Carbon-Neutral Energy Research (WPI-I$^2$CNER), Kyushu University, Fukuoka 819-0395, Japan
}
}

\maketitle

\begin{abstract}
A simple criterion of the existence of (type-I) blow-up solutions for nonautonomous ODEs is provided.
In a previous study \cite{Mat2025_NHIM}, geometric criteria for characterizing blow-up solutions for nonautonomous ODEs are provided by means of dynamics at infinity.
The basic idea towards the present aim is to correspond such criteria to leading-term equations associated with blow-up ansatz characterizing multiple-order asymptotic expansions, which originated from the corresponding study developed in the framework of autonomous ODEs. 
Restricting our attention to constant coefficients of leading terms of blow-ups, results involving the simple criterion of blow-up characterizations in autonomous ODEs can be mimicked to nonautonomous ODEs.
\end{abstract}


\tableofcontents

\section{Introduction}
\label{section-intro}

Our main interest in the present study is a characterization of {\em blow-up solutions} of ODEs depending on the real-time variable $t$:
\begin{equation}
\label{ODE-original}
{\bf y}' \equiv \frac{d{\bf y}}{dt} = f(t,{\bf y}),\quad {\bf y}(t_0) = {\bf y}_0\in \tilde U,
\end{equation}
where $\tilde U\subset \mathbb{R}^n$ is an open set, $f: U\equiv \mathbb{R}\times \tilde U\to \mathbb{R}^n$ is $C^r$ with $r\geq 1$ and $t_0\in \mathbb{R}$ is given.
In particular, the {\em nonautonomous nature} is explicitly concerned here.
\par
A solution ${\bf y}(t)$ is said to \KMi{\em blow up} at $t_{\max} < \infty$ if its modulus (or norm) diverges as $t\to t_{\max}-0$.
The finite value $t_{\max}$ is known as the {\em blow-up time}.
Properties of blow-up solutions such as the existence and asymptotic behavior are of great importance in the field of (ordinary, partial, delay, etc.) differential equations, and are widely investigated in decades in many problems from many aspects, such as {\em ignition in combustion problem \cite{D1985}, chemotaxis (e.g., \cite{M2016, W2013_KS}), superconductivity (e.g., \cite{DNZ2022, MZ2008, PS2001}), self-focusing of optical laser \cite{SS1999}, wild oscillations in nonlinear beam problems (e.g., \cite{GP2011, GP2013}), finite-time collapse of crystalline curvature flows (e.g., \cite{A2002, IY2003})}.
Other aspects can be found in e.g., \cite{BFGK2011, BK1988, CHO2007, FM2002, GV2002}. 
Asymptotic behavior of blow-up solutions is often referred to as determination of {\em blow-up rates}, which is characterized as the following form:
\begin{equation*}
{\bf y}(t) \sim {\bf u}(\theta(t))\quad \text{ as }\quad t\to t_{\max},
\end{equation*}
for a function ${\bf u}$, where 
\begin{equation*}
\theta(t) = t_{\max}-t.
\end{equation*}
Note that $t_{\max}$ is assumed to be a finite value and to be known a priori for asymptotic analysis of ${\bf y}(t)$.
\par
The author and his collaborators have recently developed a framework to characterize blow-up solutions from the viewpoint of {\em dynamics at infinity} (e.g. \cite{Mat2018, Mat2019}), as well as machineries of computer-assisted proofs for the existence of blow-up solutions extracting  their both qualitative and quantitative features (e.g. \cite{LMT2023, MT2020_1, MT2020_2, TMSTMO2017}).
As in the present paper, finite-dimensional vector fields with scale invariance in an asymptotic sense, {\em asymptotic quasi-homogeneity} defined precisely in Definition \ref{dfn-AQH}, are mainly concerned.
In this framework, dynamics at infinity is appropriately described, and blow-up solutions are shown to be characterized by dynamical properties of invariant sets, such as equilibria, \lq\lq at infinity".
More precisely, the original vector field is transformed by appropriate change of coordinates with respect to the asymptotic quasi-homogeneity, and resulting vector field is called the {\em desingularized vector field}, and \lq\lq dynamics at infinity" is appropriately described by dynamics on the {\em horizon}, the image of infinity under appropriate transformations referred to as {\em compactifications} in many studies (e.g., \cite{Mat2018}) or {\em embeddings} (e.g., \cite{Mat2025_NHIM}), for the desingularized vector fields.
In particular, {\em hyperbolicity} of such invariant sets induces blow-up rates of the form $a_0(t_{\max} - t)^{-\rho}$ uniquely determined by the asymptotic quasi-homogeneity of the original vector field.
By means of the terminology in the field of (partial) differential equations, such blow-ups are said to be {\em type-I}.
In the latest study by the author, the similar description of blow-up solutions for {\em nonautonomous} systems of ODEs such as (\ref{ODE-original}) is provided as an application of blow-up description by means of shadowing to \lq\lq trajectories at infinity" with \lq\lq hyperbolic" properties, such as those on normally hyperbolic invariant manifolds (NHIMs), as well as over/inflowing invariant manifolds; normally hyperbolic structure over manifolds with possibly non-trivial boundary (\cite{W2013_NHIM}), or invariant manifolds admitting asymptotic phase \cite{Mat2025_NHIM}.
An appropriate form for studying (\ref{ODE-original}) in this framework is the {\em extended autonomous system}:
\begin{equation*}
\frac{d{\bf y}}{d\eta} = f(t,{\bf y}),\quad \frac{dt}{d\eta} = 1,\quad (t(0), {\bf y}(\eta = 0)) = (t_0, {\bf y}_0).
\end{equation*}
An equivalent expression is provided as follows:
\begin{equation}
\label{nonaut-extend}
\frac{d}{d\eta}\begin{pmatrix}
t \\ {\bf y}
\end{pmatrix} = \begin{pmatrix}
1 \\ f(t, {\bf y})
\end{pmatrix}\equiv f^{\rm ext}(t,{\bf y}).
\end{equation}
While theoretical and numerical studies of blow-ups for nonautonomous systems are realized through the methodology mentioned above, quite hard (sometimes tedious) calculations are required for meaningful observations indeed.
It is therefore worth investigating simpler blow-up criteria for general nonautonomous systems compared with investigations of blow-ups through the above methodology, in particular {\em global} embeddings discussed in Section \ref{sec:global}.
For example, when we apply it to the first Painlev'{e} equation
\begin{equation*}
u'' = 6u^2 + t,
\end{equation*}
a $3$-dimensional polynomial vector field of {\em order $25$ (!!)} has to be investigated for understanding the full dynamics including blow-ups (\cite{Mat2025_NHIM}).
Even for calculations of \lq\lq equilibria at infinity" (whose precise meaning is mentioned later) describing blow-ups, systems of polynomials of order $7$ have to be solved, which must be costly for the original equation above\footnote{
Tedious calculations mentioned here would be avoided when we apply \lq\lq localized" embeddings referred to as {\em directional} ones \cite{Mat2025_NHIM}, as far as an interest is restricted to blow-up behavior in practical directions, like blow-ups with positive values provided that the positivity of divergence is unraveled in advance.
}.
Note that typical preceding studies before development of the above methodology rely on special structures of systems so that they can be transformed into autonomous systems (e.g., \cite{FV2003, H2016, S2004, WWL2012}).
\par
On the other hand, {\em multi-order asymptotic expansions as well as their correspondence to several objects in dynamics at infinity} are recently provided by the author and his collaborators \cite{asym1, asym2}.
While the corresponding subject is originally motivated to derive a systematic methodology to calculate multi-order asymptotic expansions of blow-up solutions, quantities characterizing asymptotic expansions of blow-ups has been turned out to have one-to-one correspondence of (linear) dynamical structure of \lq\lq equilibria at infinity".
More precisely,
\begin{itemize}
\item Roots of {\em the balance law} (in asymptotic expansions) characterizing the coefficients of the leading term of blow-up solutions and equilibria on the horizon (for desingularized vector fields).
\item Eigenstructure between matrices associated with asymptotic expansions and the Jacobian matrices at the above equilibria on the horizon (for desingularized vector fields).
\end{itemize}
This correspondence provides not only a simple criterion to verify the existence of (type-I) blow-up solutions from the viewpoint of asymptotic expansions, but also fundamental characterization to their analytic and dynamical (or geometric) nature. 
Because this characterization discussed in \cite{asym2} was for blow-ups in {\em autonomous} systems, it would be natural to question the similar nature for blow-ups in {\em nonautonomous} systems as a generalization towards a universal nature of blow-up phenomena, which is our main aim in the present paper.
That is, we shall derive a correspondence of blow-up solutions among two different viewpoints, dynamics at infinity and asymptotic expansions; {\em Theorem \ref{thm-blow-up-estr}}, by means of quantities mentioned above.
As a consequence, a simple criterion of the existence of blow-up solutions in nonautonomous systems is derived; {\em Theorem \ref{thm-existence-blow-up}}.
\par
\bigskip
The rest of the present paper is organized as follows.
In Section \ref{section-preliminary}, we briefly review treatments of nonautonomous systems of ODEs for investigating blow-up solutions, as well as a criterion of their existence.
There the associated {\em desingularized vector field}  is introduced so that \lq\lq dynamics at infinity" can be considered. 
Blow-up solutions for nonautonomous systems of ODEs are then characterized, in the simplest case, by $1$-parameter families of hyperbolic equilibria forming NHIMs on the horizon.
Our main arguments are provided in Section \ref{section-correspondence}, where a problem to study asymptotic expansions of (type-I) blow-ups is formulated first, based on the similar arguments to \cite{asym1}.
Moreover, the correspondence of quantities characterizing blow-up solutions from the viewpoint reviewed in Section \ref{section-preliminary} and their asymptotic expansions is provided there.
The main idea is based on arguments in \cite{asym2} with several modifications due to the presence of variable $t$ in (\ref{nonaut-extend}).
Such technical difficulties rely on linear algebra, which will be overcome carefully applying the geometric information of \lq\lq infinity" through the machinery in Section \ref{section-preliminary}.
Finally, several examples are shown in Section \ref{section-examples}, where we shall see that the correspondence of quantities characterizing the nature of blow-ups a simple and reasonable way to investigate their existence.
Some examples also show the complexity to verify the criterion of blow-ups based on methodologies in Section \ref{section-preliminary} themselves and the significance of the present correspondence towards further applications.
Supplemental arguments in linear algebra are collected in Appendix \ref{section-algebra}, and numerical investigations of the correspondence in an example in Section \ref{section-examples} are provided in Appendix \ref{section-numerics}.

\section{Blow-up description for nonautonomous systems through dynamics at infinity: Short review}
\label{section-preliminary}

Here we briefly review a characterization of blow-up solutions for autonomous, finite-dimensional ODEs from the viewpoint of dynamical systems.
Details of the present methodology are already provided in \cite{Mat2025_NHIM}.
%
%
\subsection{Asymptotically quasi-homogeneous vector fields}
\label{sec:QH}

\begin{dfn}[Homogeneity index and admissible domain. cf. \cite{Mat2018, Mat2025_NHIM}]\rm
\label{dfn-index}
Let $\alpha = (\alpha_1,\cdots, \alpha_n)$ be a set of nonnegative integers.
We say the index set $I_\alpha=\{i\in \{1,\cdots, n\}\mid \alpha_i > 0\}$ the set of {\em homogeneity indices associated with $\alpha = (\alpha_1,\cdots, \alpha_n)$}.
Let $U\subset \mathbb{R}^n$.
We say the domain $U\subset \mathbb{R}^n$ {\em admissible with respect to the sequence $\alpha$}
if
\begin{equation*}
U = \left\{x=(x_1,\cdots, x_n)\in \mathbb{R}^n \mid x_i\in \mathbb{R}\text{ if }i\in I_\alpha,\ (x_{j_1},\cdots, x_{j_{n-l}}) \in \tilde U\right\},
\end{equation*} 
where $\{j_1, \cdots, j_{n-l}\} = \{1,\cdots, n\}\setminus I_\alpha$ and $\tilde U$ is an open set in $\mathbb{R}^{n-l}$ spanning $(x_{j_1},\cdots, x_{j_{n-l}})$ with $\{j_1, \cdots, j_{n-l}\} = \{1,\cdots, n\}\setminus I_\alpha$.
\end{dfn}

\begin{dfn}[Asymptotically quasi-homogeneous vector fields, cf. \cite{D1993, Mat2018}]\rm
\label{dfn-AQH}
Let $U\subset \mathbb{R}^n$ be an admissible set with respect to $\alpha$. 
Also, let $f_0:U \to \mathbb{R}$ be a function.
Let $\alpha_1,\ldots, \alpha_n$ be nonnegative integers with $(\alpha_1,\ldots, \alpha_n) \not = (0,\ldots, 0)$ and $k > 0$.
We say that $f_0$ is a {\em quasi-homogeneous function\footnote{
In preceding studies, all $\alpha_i$'s and $k$ are typically assumed to be natural numbers.
In the present study, on the other hand, the above generalization is valid.
} of type $\alpha = (\alpha_1,\ldots, \alpha_n)$ and order $k$} if
\begin{equation*}
f_0( s^{\Lambda_\alpha}{\bf x} ) = s^k f_0( {\bf x} )\quad \text{ for all } {\bf x} = (x_1,\ldots, x_n)^T \in U \text{ and } s>0,
\end{equation*}
where\footnote{
Throughout the rest of this paper, the power of real positive numbers or functions to matrices is described in the similar manner.
}
\begin{equation*}
\Lambda_\alpha =  {\rm diag}\left(\alpha_1,\ldots, \alpha_n\right),\quad s^{\Lambda_\alpha}{\bf x} = (s^{\alpha_1}x_1,\ldots, s^{\alpha_n}x_n)^T.
\end{equation*}
Next, let $X = \sum_{i=1}^n f_i({\bf x})\frac{\partial }{\partial x_i}$ be a continuous vector field defined on $U$.
We say that $X$, or simply $f = (f_1,\ldots, f_n)^T$ is a {\em quasi-homogeneous vector field of type $\alpha = (\alpha_1,\ldots, \alpha_n)$ and order $k+1$} if each component $f_i$ is a quasi-homogeneous function of type $\alpha$ and order $k + \alpha_i$.
\par
Finally, we say that $X = \sum_{i=1}^n f_i({\bf x})\frac{\partial }{\partial x_i}$, or simply $f: U\to \mathbb{R}^n$ is an {\em asymptotically quasi-homogeneous vector field of type $\alpha = (\alpha_1,\ldots, \alpha_n)$ and order $k+1$ (at infinity)} if there is a quasi-homogeneous vector field  $f_{\alpha,k} = (f_{i; \alpha,k})_{i=1}^n$ of type $\alpha$ and order $k+1$ such that
\begin{equation}
\label{residual}
f_i( s^{\Lambda_\alpha}{\bf x} ) - s^{k+\alpha_i} f_{i;\alpha,k}( {\bf x} ) = o(s^{k+\alpha_i}),\quad i\in \{1,\ldots, n\}
 \end{equation}
as $s\to +\infty$ uniformly on $\left\{{\bf x}\in U \mid \sum_{i\in I_\alpha} x_i^2 = 1, (x_{j_1},\cdots, x_{j_{n-l}}) \in \tilde K\right\}$ for any compact subset $\tilde K\subset \tilde U$.
\end{dfn}

\begin{rem}
In the above definition, non-polynomial-like functions such as $\sin x$ are not included to characterize quasi-homogeneity.
Indeed, such functions are allowed to exist only in the residual terms characterized by the asymptotic quasi-homogeneity (\ref{residual}). 
On the other hand, (\ref{residual}) is required for all $i\in \{1,\ldots, n\}$.
\end{rem}

A fundamental property of quasi-homogeneous functions and vector fields is reviewed in e.g. \cite{asym1}.
Throughout the rest of this section, consider an (autonomous) $C^r$ vector field (\ref{ODE-original}) with $r\geq 1$, where $f: U \to \mathbb{R}^n$ is asymptotically quasi-homogeneous of type $\alpha = (\alpha_1,\ldots, \alpha_n)$ and order $k+1$ at infinity defined on an admissible set $U\subset \mathbb{R}^n$ with respect to $\alpha$.

Some fundamental property of quasi-homogeneous functions and vector fields are reviewed here.
\begin{lem}[\cite{asym2}]
\label{temporary-label2}
A quasi-homogenous function $f_0$ of type $(\alpha_1,\ldots,\alpha_n)$ and order $k$ satisfies the following differential equation:
\begin{equation}\label{temporary-label1}
\sum_{l=1}^n \alpha_l y_l \frac{\partial f_0}{\partial y_l}({\bf y}) = k f_0({\bf y}),
\end{equation}
equivalently
\begin{equation*}
(\nabla_{\bf y} f_0({\bf y}))^T \Lambda_{\alpha} {\bf y} = k f_0({\bf y}).
\end{equation*}
\end{lem}

See \cite{asym2} for detailed arguments about asymptotic behavior of their derivatives.

\begin{lem}[\cite{asym2}]
\label{lem-identity-QHvf}
A quasi-homogeneous vector field $f=(f_1,\ldots,f_n)$ of 
type $\alpha = (\alpha_1,\ldots,\alpha_n)$ and order $k+1$ satisfies the following differential equation:
\begin{equation}
\label{temporary-label3}
\sum_{l=1}^n \alpha_l y_l \frac{\partial f_i}{\partial y_l}({\bf y}) = (k+\alpha_i) f_i({\bf y}) \qquad (i=1,\ldots,n).
\end{equation}
This equation can be rephrased as
\begin{equation}\label{temporay-label4}
(D f)({\bf y}) \Lambda_\alpha \mathbf{y} = \left( k I+ \Lambda_\alpha \right) f({\bf y}).
\end{equation}
\end{lem}

\begin{rem}
Note that these characterizations still hold when the type $\alpha$ contains a component with $\alpha_i = 0$.
In particular, the framework involving quasi-homogeneity discussed here can be applied to dynamics at infinity for nonautonomous systems.
\end{rem}

Throughout successive sections, consider an (autonomous) $C^r$ vector field\footnote{
In \cite{asym1} $r\geq 2$ was assumed, which was for justification of asymptotic expansions, while such an extra smoothness is not necessary for the present purpose.
} (\ref{nonaut-extend}) with $r\geq 1$, 
where $f: U\equiv \mathbb{R}\times \tilde U \to \mathbb{R}^n$ is asymptotically quasi-homogeneous of type $\alpha = (0, \alpha_1,\ldots, \alpha_n)$ and order $k+1$ at infinity, where $\tilde U\subset \mathbb{R}^n$ is an admissible set with respect to $\alpha$.

\subsection{Quasi-parabolic embeddings}
\label{sec:global}

Here we review an example of embeddings which embed the original (locally compact) phase space into a compact manifold to characterize \lq\lq infinity" as a bounded object.
While there are several choices of embeddings, the following embedding to characterize {\em dynamics at infinity} is applied here.

\begin{dfn}[Quasi-parabolic embedding, cf. \cite{MT2020_1}]\rm
\label{dfn-quasi-para}
Let $\{\beta_i\}_{i\in I_\alpha}$ be the collection of natural numbers so that 
\begin{equation}
\label{LCM}
\alpha_i \beta_i \equiv c \in \mathbb{N},\quad i\in I_\alpha
\end{equation}
is the least common multiplier.
In particular, $\{\beta_i\}_{i\in I_\alpha}$ is chosen to be the smallest among possible collections.
Let $p_\alpha({\bf y})$ be a functional given by
\begin{equation}
\label{func-p}
p_\alpha({\bf y}) \equiv \left( \sum_{i \in I_\alpha} y_i^{2\beta_i} \right)^{1/2c}.
\end{equation}
Define the mapping $T_{{\rm para};\alpha}: \mathbb{R}^n \to \mathbb{R}^n$ as the inverse of
\begin{equation}
\label{parabolic-cpt}
S_{{\rm para};\alpha}({\bf x}) = {\bf y},\quad y_j = \kappa_\alpha^{\alpha_j} x_j,\quad j=1,\ldots, n,
\end{equation}
where 
\begin{equation*}
\kappa_\alpha = \kappa_\alpha({\bf x}) = (1- p_\alpha({\bf x})^{2c})^{-1} \equiv \left( 1 - \sum_{j\in I_\alpha} x_j^{2\beta_j}\right)^{-1}.
\end{equation*}
We say the mapping $T_{{\rm para};\alpha}$ the {\em quasi-parabolic embedding (with type $\alpha$)}.
\end{dfn}

\begin{rem}
\label{rem-kappa}
The functional $\kappa_\alpha = \tilde \kappa_\alpha({\bf y})$ as a functional determined by ${\bf y}$ is implicitly determined by $p_\alpha({\bf y})$.
Details of such a characterization of $\kappa_\alpha$ in terms of ${\bf y}$, and the bijectivity and smoothness of $T_{{\rm para};\alpha}$ are shown in \cite{MT2020_1} with a general class of embeddings, where the embedding is referred to as {\em compactifications}.
\end{rem}

As proved in \cite{MT2020_1}, $T_{{\rm para};\alpha}$ maps $U$ one-to-one onto the set
$\mathcal{D} \equiv \{{\bf x}\in U \mid p_\alpha({\bf x}) < 1\}$.
Infinity in the original coordinate then corresponds to a point on the level set of $p_\alpha$:
\begin{equation*}
\mathcal{E} = \{{\bf x} \in U \mid p_\alpha({\bf x}) = 1\}.
\end{equation*}
\begin{dfn}\rm
We call the set $\mathcal{E}$ the {\em horizon}.
\end{dfn}

\subsection{Dynamics at infinity in nonautonomous systems}
\label{label-dynamics-infinity}
Once we fix an embedding associated with the type $\alpha = (0, \alpha_1, \ldots, \alpha_n)$ of the vector field $f$ with order $k+1$, we can derive the vector field which makes sense including the horizon.
Then the {\em dynamics at infinity} makes sense through the appropriately transformed vector field called the {\em desingularized vector field}, denoted by $g$.
The common approach is twofold.
Firstly, we rewrite the vector field (\ref{ODE-original}), or (\ref{nonaut-extend}) with respect to the new variable used in embeddings.
Secondly, we introduce the time-scale transformation of the form $d\tau = q({\bf x})\kappa_\alpha({\bf x}(t))^k dt$ for some function $q({\bf x})$ which is bounded including the horizon. 
We then obtain the vector field with respect to the new time variable $\tau$, which is continuous, including the horizon.

\begin{rem}
\label{rem-choice-cpt}
Continuity of the desingularized vector field $g$ including the horizon is guaranteed by the smoothness of $f$ and asymptotic quasi-homogeneity (\cite{MT2020_1}).
In the case of parabolic-type embeddings, $g$ inherits the smoothness of $f$ including the horizon, which is not always the case of other embeddings in general. 
Details are discussed in \cite{Mat2018}.
\end{rem}

\begin{dfn}[Time-scale desingularization]\rm 
Define the new time variable $\tau$ by
\begin{equation}
\label{time-desing-nonaut}
d\tau = (1-p_\alpha({\bf x})^{2c})^{-k}\left\{1-\frac{2c-1}{2c}(1-p_\alpha({\bf x})^{2c}) \right\}^{-1}d\eta,
\end{equation}
equivalently
\begin{equation*}
\eta - \eta_0 = \int_{\tau_0}^\tau \left\{1-\frac{2c-1}{2c}(1-p_\alpha({\bf x}(\tau))^{2c}) \right\}(1-p_\alpha({\bf x}(\tau))^{2c})^k d\tau,
\end{equation*}
where $\tau_0$ and $\eta_0$ denote the correspondence of initial times, ${\bf x}(\tau) = T({\bf y}(\tau))$ and ${\bf y}(\tau)$ is a solution ${\bf y}(\eta)$ under the parameter $\tau$.
We shall call (\ref{time-desing-nonaut}) {\em the time-scale desingularization of order $k+1$}.
\end{dfn}

We then obtain the corresponding desingularized vector field $g^{\rm ext}$ defined below with $g = (g_1, \ldots, g_n)^T$;
\begin{align}
\notag
\frac{d}{d\tau}\begin{pmatrix}
t \\ {\bf x}
\end{pmatrix} &= g^{\rm ext}(t,{\bf x}) \equiv \begin{pmatrix}
g_0(t,{\bf x}) \\
g(t, {\bf x})
\end{pmatrix} \\
\label{desing-para-nonaut}
	&= \left(1-\frac{2c-1}{2c}(1-p_\alpha({\bf x})^{2c}) \right)\tilde f^{\rm ext}(t, {\bf x}) - G(t, {\bf x})\Lambda^{\rm ext}_\alpha \begin{pmatrix}
t \\ {\bf x}
\end{pmatrix}
\end{align}
with the following notations, which are consistent with the general derivation in autonomous systems \cite{Mat2025_NHIM}:
\begin{align}
\notag
\tilde f^{\rm ext}(t,{\bf x}) &= \begin{pmatrix}
\tilde f_0, \tilde f_1, \ldots, \tilde f_n
\end{pmatrix}^T,\quad f_0(t,{\bf y}) \equiv 1,\\
\label{f-tilde}
\tilde f_j(t, x_1,\ldots, x_n) &:= \kappa_\alpha^{-(k+\alpha_j)} f_j(t, \kappa_\alpha^{\alpha_1}x_1, \ldots, \kappa_\alpha^{\alpha_n}x_n),\quad j=0,1,\ldots, n,\\
\label{Gx}
G(t, {\bf x}) &\equiv \sum_{j\in I_\alpha} \frac{x_j^{2\beta_j-1}}{\alpha_j}\tilde f_j(t, {\bf x}),\quad 
\Lambda^{\rm ext}_\alpha = {\rm diag}(0, \alpha_1, \ldots, \alpha_n).
\end{align}
Note that, in the above notation, $\tilde f_0(t,{\bf x}) = (1-p_\alpha({\bf x})^{2c})^k$ via (\ref{f-tilde}).
The above identifications are also consistent with the evolution of $t$ followed by the time-scale desingularization (\ref{time-desing-nonaut}).

Smoothness of $f$ and the asymptotic quasi-homogeneity guarantee the smoothness of the right-hand side $g$ of (\ref{desing-para-nonaut}) including the horizon $\mathcal{E}\equiv \{p_\alpha({\bf x}) = 1\}$, {\em provided\footnote{
This requirement is essential in the nonautonomous case, while it is not the case of the autonomous case.
Indeed, the function $g_0$ is explicitly included in the desingularized vector field $g^{\rm ext}$, while it is included only in the formula of $t_{\max}$ in (\ref{tmax}) for autonomous systems.
In any case, if $k\in (0,1)$, then $g_0$ is not smooth on $\mathcal{E}$.
}
$k=0$ or $k\geq 1$}.
In particular, {\em dynamics at infinity}, such as divergence of solution trajectories to specific directions, is characterized through dynamics generated by (\ref{desing-para-nonaut}) around the horizon. 
See \cite{Mat2025_NHIM} for details.

\begin{rem}[Invariant structure]
\label{rem-invariance}
The horizon $\mathcal{E}$ is a codimension one invariant submanifold of $\widetilde{D}\equiv \mathcal{D}\cup \mathcal{E}$. 
Indeed, direct calculations yield that
\begin{equation*}
\left. \frac{d}{d \tau}p_\alpha({\bf x}(\tau))^{2c}\right|_{\tau=0} = 0\quad \text{ whenever }\quad (t, {\bf x}(0))\in \mathcal{E}.
\end{equation*}
\end{rem}

\subsection{Type-I nonautonomous blow-up}
Through the embedding we have introduced, 
dynamics around the horizon characterize dynamics at infinity, including blow-up behavior.
For a point ${\bf p}_\ast = (t_\ast, {\bf x}_\ast)$, let
\begin{equation*}
W_{\rm loc}^s({\bf p}_\ast) = W_{\rm loc}^s({\bf p}_\ast; g^{\rm ext}) := \{ (t, {\bf x})\in U \mid |\varphi_{g^{\rm ext}}(t,{\bf x}) - \varphi_{g^{\rm ext}}(t,{\bf p}_\ast) | \to 0\, \text{ as }\,t\to +\infty\}
\end{equation*}
be the {\em (local) stable set} of ${\bf p}_\ast$ for the dynamical system generated by $g^{\rm ext}$, where $U$ is a neighborhood of ${\bf p}_\ast$ in $\mathbb{R}^{n+1}$ or an appropriate phase space, and $\varphi_{g^{\rm ext}}$ is the flow generated by $g^{\rm ext}$.
In a special case where 
${\bf p}_\ast$ is a point on a (boundaryless, compact, connected) {\em normally hyperbolic invariant manifold} $M$ (NHIM for short), the stable set $W_{\rm loc}^s({\bf p}_\ast)$ admits a smooth manifold structure in a small neighborhood of ${\bf p}_\ast$ through a stable foliation $\mathcal{F}^s$ of $W^s_{\rm loc}(M) = W^s_{\rm loc}(M; g^{\rm ext})$, the {\em (local) stable manifold} of $M$.
Following \cite{Mat2025_NHIM}, where a description of blow-ups by means of NHIMs on the horizon is provided, we shall use the notation below.
For any set $M\subset \widetilde{D}\equiv \mathcal{D}\cup \mathcal{E}$, $\bar t\in \mathbb{R}$ and an interval $I\subset \mathbb{R}$, let
\begin{equation}
\label{M_slice}
M_{\bar t} \equiv M\cap \{t = \bar t\},\quad M_I \equiv M\cap \{t\in I\}
\end{equation}
be the slice and the tube of $M$ on $I$, respectively.

\begin{ass}
\label{ass-nonaut-inv}
Fix an initial time $t_0\in \mathbb{R}$.
There is a compact interval $I'$ with $t_0 \in {\rm int}I' (\not = \emptyset)$ such that the system (\ref{desing-para-nonaut}) admits an invariant manifold $M\subset \mathcal{E}$ satisfying
\begin{align}
\notag
&M_{I'} = \{ (t, {\bf x}_\ast(t)) \mid t\in I',\, {\bf x}_\ast(t) \in \mathcal{E} \text{ is an equilibrium for $\varphi_g$ satisfying (\ref{NH-spec-nonaut-stationary}) below}\}\\
\label{NH-spec-nonaut-stationary}
&\sharp \{{\rm Spec}(Dg^{\rm ext}(t, {\bf x}_\ast(t))) \cap i\mathbb{R} \} = 1\quad \text{ for all }\quad t\in I'
\end{align}
where ${\rm Spec}(A)$ denotes the set of eigenvalues of a squared matrix $A$, 
that is, $M_{I'}$ is a curve of hyperbolic equilibria parameterized by $t$.
\end{ass}

\begin{thm}[Nonautonomous blow-up in a special case, cf. \cite{Mat2025_NHIM}]
\label{thm:blowup}
Suppose that $g$ admits an invariant manifold $M\subset \mathcal{E}$ satisfying all requirements in Assumption \ref{ass-nonaut-inv}. 
Let $I \subset I'$ be any compact interval satisfying $t_0 \in I \subset {\rm int}I'$.
Finally, suppose that a solution ${\bf y}(t)$ of (\ref{nonaut-extend}) with a bounded initial point $(t_0, {\bf y}_0 (= {\bf y}(t_0))) \in I\times \mathbb{R}^n$ whose image $(t(\tau), {\bf x}(\tau)) = T_{{\rm para};\alpha}((t(\tau), {\bf y}(\tau)))$ for $T_{{\rm para};\alpha}$ is on $W_{\rm loc}^s(M_I; g^{\rm ext})$.
Then ${\bf y}(t)$ is a blow-up solution of (\ref{nonaut-extend}), equivalently of (\ref{ODE-original}).
In particular, $t_{\max}\in I$ and ${\bf y}(t)$ diverges as $t\to t_{\max}-0$.
\par
Moreover, for ${\rm Spec}(Dg^{\rm ext}(t_{\max}, {\bf x}_\ast(t_{\max})))\setminus i\mathbb{R} \equiv \{\lambda_j\}_{j=1}^n$, if the non-resonance condition\footnote{
In \cite{Mat2025_NHIM}, this condition is referred to as the {\em Sternberg-Sell condition of order $1$}. 
In the present case, it is nothing but the non-resonance condition because the \lq\lq spectrum" of the linearized matrix(-valued function) consists of discrete eigenvalues.
} is satisfied with $r\geq 4$, namely
\begin{equation*}
a_1 \lambda_1 + \cdots + a_n \lambda_n - \lambda_j \not = 0
\end{equation*}
for any $j\in \{1,\ldots, n\}$ and any $(a_1, \cdots, a_n) \in \mathbb{Z}_{>0}^n$ with $\sum_{j=1}a_j \geq 2$, then we have
\begin{align*}
p_\alpha({\bf y}(t)) \sim C_0 (t_{\max}-t)^{-1/k}\quad \text{ as }\quad t \to t_{\max}-0
\end{align*}
for some constant $C_0 > 0$ as well as
\begin{equation*}
y_i(t) \sim C_i (t_{\max}-t)^{-\alpha_i /k} \quad \text{ as }\quad t \to t_{\max}-0
\end{equation*} 
for some constants $C_i$, as far as $x_{\ast, i}\not = 0$ with $(t_{\max}, {\bf x}_\ast)\in M_I$.
\end{thm}

Note that $M_I$ is a manifold with boundary admitting {\em normally hyperbolic} structure (e.g., \cite{W2013_NHIM}).
The key point of the theorem is that blow-up solutions for (\ref{ODE-original}) are characterized as trajectories on local stable manifolds of equilibria
on the horizon $\mathcal{E}$ for the desingularized vector field.
Investigations of blow-up structure are therefore reduced to those of stable manifolds of equilibria (or general invariant sets) on the horizon for the associated vector field.
Moreover, the theorem also claims that, under mild assumptions, (normally) hyperbolic equilibria on the horizon induce {\em type-I blow-up. 
That is, the leading term of the blow-up behavior is determined by the type $\alpha$ and the order $k+1$ of $f$.}
An explicit formula of blow-up time $t_{\max}$ is provided through (\ref{time-desing-nonaut}) as follows;
\begin{equation}
\label{tmax}
t_{\max} = t_0 + \frac{1}{2c}\int_{\tau_0}^\infty \left(1 + (2c-1) p_\alpha({\bf x}(\tau))^{2c}\right) (1-p_\alpha({\bf x}(\tau))^{2c})^k d\tau,
\end{equation}
where $t_0 = t(\tau_0)$.
The numerical integrations of $t_{\max}$ approximated by step functions are applied to computing $t_{\max}$ in Section \ref{section-ex-Hamiltonian}.
See e.g., \cite{LMT2023, MT2020_1} for rigorous enclosures of $t_{\max}$.


\par
\bigskip

We end this section by providing several properties to describe features of dynamics on the horizon. 
First we consider a function $G$ in (\ref{Gx}), which is essential to characterize equilibria on the horizon mentioned in Section \ref{section-review-asym}.
The gradient\footnote{
Notations are referred to in the beginning of Section \ref{section-correspondence}.
}
 of the horizon $\mathcal{E} = \{p_\alpha(t, {\bf x}) = 1\}$ at $(t, {\bf x})\in \mathcal{E}$ is given by
\begin{equation*}
D p_\alpha({\bf x}) = \frac{p_\alpha({\bf x})^{1-2c}}{c} \left(0, \beta_1x_1^{2\beta_1-1}, \ldots, \beta_nx_n^{2\beta_n-1} \right)^T = \frac{1}{c} \left(0, \beta_1x_1^{2\beta_1-1}, \ldots, \beta_nx_n^{2\beta_n-1} \right)^T,
\end{equation*}
where $\beta_l$ is set as $0$ whenever $\alpha_l = 0$.
In particular, 
\begin{equation}
\label{grad-xast}
D p_\alpha({\bf x}_\ast) = \frac{1}{c}\left(0, \beta_1x_{\ast,1}^{2\beta_1-1}, \ldots, \beta_n x_{\ast,n}^{2\beta_n-1} \right)^T
\end{equation}
holds at an equilibrium $(t_\ast, {\bf x}_\ast) \in \mathcal{E}$.
For the frequent use in the following discussions, write
\begin{equation*}
D p_\alpha({\bf x}_\ast) \equiv (0, D_{\bf x} p_\alpha({\bf x}_\ast)^T)^T,\quad D_{\bf x} p_\alpha({\bf x}_\ast)^T = \frac{1}{c}\left( \beta_1x_{\ast,1}^{2\beta_1-1}, \ldots, \beta_n x_{\ast,n}^{2\beta_n-1} \right).
\end{equation*}
Similarly, we observe that
\begin{equation}
\label{grad-2c-xast}
D (p_\alpha( {\bf x})^{2c})^T = 2\left(0, \beta_1x_1^{2\beta_1-1}, \ldots, \beta_nx_n^{2\beta_n-1} \right)
\end{equation}
for any $(t,{\bf x})\in \tilde{\mathcal{D}}$.
Using the gradient, the function $G(t,{\bf x})$ in (\ref{Gx}) is also written by
\begin{equation*}
G(t,{\bf x}) = \sum_{j\in I_\alpha} \frac{\beta_j}{c}x_j^{2\beta_j-1}\tilde f_j(t, {\bf x}) = \frac{1}{2c}D (p_\alpha({\bf x})^{2c})^T \tilde f^{\rm ext}(t, {\bf x}). 
\end{equation*}
From (\ref{grad-xast}) and (\ref{grad-2c-xast}), we have $D (p_\alpha( {\bf x}_\ast)^{2c})^T = 2c D p_\alpha( {\bf x}_\ast)^T$ for any $(t_\ast, {\bf x}_\ast) \in \mathcal{E}$, in particular
\begin{equation*}
G(t_\ast, {\bf x}_\ast) = \frac{1}{2c}D (p_\alpha( {\bf x}_\ast)^{2c})^T \tilde f^{\rm ext}(t_\ast, {\bf x}_\ast) = Dp_\alpha( {\bf x}_\ast)^T \tilde f^{\rm ext}(t_\ast, {\bf x}_\ast). 
\end{equation*}

\begin{lem}
\label{lem-G}
\begin{equation*}
\kappa_\alpha^{-1} \frac{d\kappa_\alpha}{d\tau} = G(t(\tau), {\bf x}(\tau)),
\end{equation*}
where $G(t, {\bf x})$ is given in (\ref{Gx}).
\end{lem}
Note that $G$ does depend on $t$, while the functional $\kappa_\alpha$ is independent of $t$ from the choice of type $\alpha$.
Nevertheless, the essence of the proof is the same as the autonomous cases (\cite{asym2}), paying attention to the treatment of $t$ and its scaling.

\begin{proof}
Direct calculations with (\ref{LCM}) yield that
\begin{align*}
\kappa_\alpha^{-2} \frac{d\kappa_\alpha}{d\tau} &\equiv -\frac{d(\kappa_\alpha)^{-1}}{d\tau} \\
	&= D (p_\alpha({\bf x})^{2c})^T \begin{pmatrix}
\frac{dt}{d\tau} & \frac{d{\bf x}}{d\tau}
\end{pmatrix}^T\\
	&= D (p_\alpha({\bf x})^{2c})^T\left( \frac{1}{2c}\left(1 + (2c-1)p_\alpha({\bf x})^{2c} \right)\tilde f^{\rm ext}(t, {\bf x}) - G(t, {\bf x})\Lambda^{\rm ext}_\alpha \begin{pmatrix}
t \\ {\bf x}
\end{pmatrix}\right) \\
	&= \left(1 + (2c-1)p_\alpha({\bf x})^{2c} \right)G(t, {\bf x}) - 2cp_\alpha({\bf x})^{2c}G(t, {\bf x})  \\
	&= (1-p_\alpha({\bf x})^{2c})G(t, {\bf x})\\
	&= \kappa_\alpha^{-1}G(t, {\bf x}).
\end{align*}
\end{proof}

\section{Eigenstructure of dynamics at infinity and a simplified blow-up criterion}
\label{section-correspondence}

This section addresses several structural correspondences between dynamics around equilibria on the horizon for desingularized vector fields and another nature of blow-ups; {\em asymptotic expansions} derived in \cite{asym1}.
As a result, we see that a criterion for their existence is simplified through the latter nature.
Unless otherwise mentioned, let $g^{\rm ext}$ be the desingularized vector field (\ref{desing-para-nonaut}) associated with $f$.
\par
Because we shall use specific objects frequently, several abbreviations below are introduced.
\begin{dfn}[Summary of notations]\rm
Several notations we frequently use in the present section are collected here.
\begin{itemize}
\item $I = I_{n+1}$: the identity matrix in $\mathbb{R}^{n+1}$, while $I_m$ denotes the identity matrix in $\mathbb{R}^m$ for $m\in \mathbb{Z}_{>0}$.
\item $(t_\ast, {\bf x}_\ast)$: an equilibrium on the horizon $\mathcal{E}$ for the desingularized vector field.
\item $Dg^{\rm ext}_\ast:= Dg^{\rm ext}(t_\ast, {\bf x}_\ast)$; the derivative of $g^{\rm ext}$ with respect to $(t,{\bf x})$ at $(t_\ast, {\bf x}_\ast)$.
\item $(D_{\bf x}g)_\ast$: $D_{\bf x}g(t_\ast, {\bf x}_\ast)$, the ${\bf x}$-derivative of $g$; the ${\bf x}$-component of $g^{\rm ext}$.
\item $(D_t g)_\ast:= D_t g(t_\ast, {\bf x}_\ast)$, the $t$-derivative of $g$.
\item $D\tilde f^{\rm ext}_\ast := D\tilde f^{\rm ext}(t_\ast, {\bf x}_\ast)$: the derivative of $\tilde f^{\rm ext}$ with respect to $(t,{\bf x})$ at $(t_\ast, {\bf x}_\ast)$.
\item $(D\tilde f_{\alpha, k}^{\rm ext})_\ast := D\tilde f_{\alpha, k}^{\rm ext}(t_\ast, {\bf x}_\ast)$: the derivative of $\tilde f_{\alpha, k}^{\rm ext}$ with respect to $(t,{\bf x})$ at $(t_\ast, {\bf x}_\ast)$.
\item $(D_{\bf x}\tilde f)_\ast := D_{\bf x}\tilde f(t_\ast, {\bf x}_\ast)$, the ${\bf x}$-derivative of $\tilde f$; the ${\bf x}$-component of $\tilde f^{\rm ext}$.
\item $(D_t \tilde f)_\ast := D_t \tilde f(t_\ast, {\bf x}_\ast)$, the $t$-derivative of $\tilde f$.
\item $\tilde f_\ast := \tilde f(t_\ast, {\bf x}_\ast)$ and $\tilde f^{\rm ext}_\ast := \tilde f^{\rm ext}(t_\ast, {\bf x}_\ast)$; values of the corresponding functions evaluated at $(t_\ast, {\bf x}_\ast)$.
\item $\tilde f_{\ast,i} := \tilde f_i(t_\ast, {\bf x}_\ast)$; the $i$-th component of $\tilde f_i(t_\ast, {\bf x}_\ast)$.
\item $(D_{\bf x}p_\alpha)_\ast := D_{\bf x}p_\alpha ({\bf x}_\ast) \in \mathbb{R}^n$, the ${\bf x}$-gradient of $p_\alpha$ at $(t_\ast, {\bf x}_\ast)$.
\item $(Dp_\alpha)_\ast := D_{(t,{\bf x})} p_\alpha ({\bf x}_\ast) \in \mathbb{R}^{n+1}$, the $(t,{\bf x})$-gradient of $p_\alpha$ at $(t_\ast, {\bf x}_\ast)$.
\item $DG_\ast := DG(t_\ast, {\bf x}_\ast)$: the $(t,{\bf x})$-gradient of $G$ at $(t_\ast, {\bf x}_\ast)$.
\item $T_{(t_\ast, {\bf x}_\ast)}\mathcal{E}$: the tangent space of $\mathcal{E}$ at $(t_\ast, {\bf x}_\ast)$.
\item $T_{{\bf x}_\ast}\mathcal{E}_{\bf x}$: the tangent space of the projection $\mathcal{E}_{\bf x} \subset \mathbb{R}^n$ of $\mathcal{E}$ onto the ${\bf x}$-component at ${\bf x}_\ast \in \mathcal{E}_{\bf x}$.
\end{itemize}

\end{dfn}

\subsection{A system associated with multi-order asymptotic expansion of type-I blow-up solutions}
\label{section-review-asym}
We quickly introduce the methodology of multi-order asymptotic expansions of type-I blow-up solutions proposed in \cite{asym1}.
The method begins with the following ansatz, which can be easily verified through the desingularized vector field and our blow-up description; Theorem \ref{thm:blowup}.

\begin{ass}
\label{ass-fundamental}
The asymptotically quasi-homogeneous system (\ref{ODE-original}) of type $\alpha$ and the order $k+1$ admits a solution 
\begin{equation*}
{\bf y}(t) = (y_1(t), \ldots, y_n(t))^T
\end{equation*}
 which blows up at $t = t_{\max} < \infty$ with the type-I blow-up behavior, namely\footnote{
For two scalar functions $h_1$ and $h_2$, $h_1 \sim h_2$ as $t\to t_{\max}-0$ iff $(h_1(t)/h_2(t))\to 1$ as $t\to t_{\max}-0$. 
 }
\begin{equation}
\label{blow-up-behavior}
y_i(t) \sim c_i\theta(t)^{-\alpha_i / k}, 
\quad t\to t_{\max}-0,\quad i=1,\ldots, n
\end{equation}
for some constants $c_i\in \mathbb{R}$, where $\theta(t) = t_{\max} - t$.
\end{ass}
An aim in this subject is, under the above assumption, to write ${\bf y}(t)$ as
\begin{equation}
\label{blow-up-sol}
{\bf y}(t) = \theta(t)^{-\frac{1}{k}\Lambda_\alpha} {\bf Y}(t),\quad {\bf Y}(t) = (Y_1(t), \ldots, Y_n(t))^T
\end{equation}
with the asymptotic expansion by means of {\em general asymptotic series}\footnote{
For two scalar functions $h_1$ and $h_2$, $h_1 \ll h_2$ as $t\to t_{\max}-0$ iff $(h_1(t)/h_2(t))\to 0$ as $t\to t_{\max}-0$. 
For two vector-valued functions ${\bf h}_1$ and ${\bf h}_2$ with ${\bf h}_i = (h_{1;i}, \ldots, h_{n;i})$, ${\bf h}_1 \ll {\bf h}_2$ iff $h_{l;1} \ll h_{l;2}$ for each $l = 1,\ldots, n$.
}
\begin{align}
\label{Y-asym}
{\bf Y}(t) &= {\bf Y}_0 + \tilde {\bf Y}(t),\quad \tilde {\bf Y}(t) \ll {\bf Y}_{0}(t)\quad (t\to t_{\max}-0)
\end{align}
and determine the concrete form of the factor ${\bf Y}(t)$.

Decompose the vector field $f^{\rm ext}$ into two terms as follows:
\begin{equation*}
f^{\rm ext}(t, {\bf y}) = f^{\rm ext}_{\alpha, k}(t, {\bf y}) + f^{\rm ext}_{\rm res}(t, {\bf y}),
\end{equation*}
where $f^{\rm ext}_{\alpha, k}$ is the quasi-homogeneous component of $f$ and $f^{\rm ext}_{\rm res}$ is the residual (i.e., lower-order) terms.
The componentwise expressions are 
\begin{equation*}
f^{\rm ext}_{\alpha, k}(t, {\bf y}) = (f_{0;\alpha, k}(t, {\bf y}), \ldots, f_{n;\alpha, k}(t, {\bf y}))^T,\quad f^{\rm ext}_{\rm res}(t, {\bf y}) = (f_{0;{\rm res}}(t, {\bf y}), \ldots, f_{n;{\rm res}}(t, {\bf y}))^T,
\end{equation*}
respectively.
The similar expressions are also used to other vector fields such as $f$.
\begin{rem}
\label{rem-QH-0th}
In nonautonomous systems, the $0$-th component $f_0(t,{\bf y}) = 1$ is regarded as the residual term, which is compatible with asymptotic quasi-homogeneity.
In other words, $f_{0;\alpha, k}(t, {\bf y}) \equiv 0$.
\end{rem}
Substituting (\ref{blow-up-sol}) into (\ref{ODE-original}), we derive the system of $(t, {\bf Y}(t))$, which is 
\begin{equation}
\label{blow-up-basic}
\frac{d}{ds}\begin{pmatrix}
t \\ {\bf Y}
\end{pmatrix} = \theta(t)^{-1}\left\{ - \frac{1}{k}\Lambda^{\rm ext}_\alpha \begin{pmatrix}
t \\ {\bf Y}
\end{pmatrix} + f^{\rm ext}_{\alpha, k}(t, {\bf Y}) \right\} + \theta(t)^{ \frac{1}{k}\Lambda^{\rm ext}_\alpha } f^{\rm ext}_{{\rm res}}\left( \theta(t)^{-\frac{1}{k}\Lambda^{\rm ext}_\alpha } \begin{pmatrix}
t \\ {\bf Y}
\end{pmatrix} \right).
\end{equation}
From the asymptotic quasi-homogeneity of $f^{\rm ext}$, the most singular part of the above system yields an identity which the leading term ${\bf Y}_0$ of ${\bf Y}(t)$ must satisfy.

\begin{dfn}[Balance law]\rm
\label{dfn-balance}
We call the identity
\begin{equation}
\label{0-balance}
-\frac{1}{k}\Lambda^{\rm ext}_\alpha \begin{pmatrix}
t \\ {\bf Y}_0
\end{pmatrix} + f^{\rm ext}_{\alpha, k}(t, {\bf Y}_0) = 0
\end{equation}
{\em a balance law} for the blow-up solution ${\bf y}(t)$ for (\ref{ODE-original}).
Note that the identity in the $0$-th (namely, $t$-)component is the trivial one.
\end{dfn}
The next step is to derive the collection of systems for $\tilde {\bf Y}_j(t)$ by means of {\em inhomogeneous linear systems}, but we only mention a key concept towards our aim here.

\begin{dfn}[Blow-up power eigenvalues]\rm
\label{dfn-blow-up-power-ev}
Suppose that a nonzero root ${\bf Y}_0$ of the balance law (\ref{0-balance}) is given.
We call the constant matrix
\begin{equation}
\label{blow-up-power-determining-matrix}
A^{\rm ext} = - \frac{1}{k}\Lambda^{\rm ext}_\alpha + D f^{\rm ext}_{\alpha, k}(t, {\bf Y}_0)
\end{equation}
the {\em blow-up power-determining matrix} for the blow-up solution ${\bf y}(t)$, and call the eigenvalues $\{\lambda_i\}_{i=0}^n \equiv {\rm Spec}(A^{\rm ext})$ the {\em blow-up power eigenvalues}, where the derivative $D = (D_t, D_{\bf y})^T$ is with respect to $(t,{\bf y})$, and eigenvalues with nontrivial multiplicity are distinguished in this expression, except specifically noted.
Finally, we shall use the following expression of $A^{\rm ext}$: 
\begin{equation}
\label{Aext-sub}
A^{\rm ext} = \begin{pmatrix}
0 & {\bf 0}_n^T\\
D_t f_{\alpha, k}(t, {\bf Y}_0) & A
\end{pmatrix},\quad A\in M_n(\mathbb{R}).
\end{equation}
\end{dfn}
The balance law (\ref{0-balance}) and the matrix $A^{\rm ext}$ can provide multi-order asymptotic expansions of blow-up solutions of the form (\ref{blow-up-sol}) for nonautonomous systems like (\ref{ODE-original}) in the similar way to autonomous systems, as derived in \cite{asym1}.
Procedures of asymptotic expansions of blow-ups for nonautonomous systems will be omitted in the present paper due to the similarity of our setting for discussing expansions. 
Instead, we concentrate on the correspondence of dynamics at infinity to structures derived from the balance law (\ref{0-balance}) and the matrix $A^{\rm ext}$.

\subsection{Balance law and equilibria on the horizon}

The first issue for the correspondence of dynamical structures is \lq\lq equilibria" among two systems.
Recall that equilibria for the desingularized vector field (\ref{desing-para-nonaut}) associated with quasi-parabolic embeddings satisfy
\begin{equation}
\label{balance-para}
 \left(1-\frac{2c-1}{2c}(1-p_\alpha({\bf x})^{2c}) \right)\tilde f^{\rm ext}(t, {\bf x}) = G(t, {\bf x}) \Lambda^{\rm ext}_\alpha \begin{pmatrix}
t \\ {\bf x}
\end{pmatrix},
\end{equation}
where $G(t, {\bf x})$ is given in (\ref{Gx}).
Equilibria $(t_\ast,{\bf x}_\ast)$ with ${\bf x}_\ast = (x_{\ast, 1}, \ldots, x_{\ast, n})^T$ on the horizon $\mathcal{E}$ satisfy $p_\alpha({\bf x}_\ast) \equiv 1$ and hence the following identity holds:
\begin{equation}
\label{const-horizon-0}
\tilde f_{\ast, i} = \alpha_i x_{\ast, i} G(t_\ast, {\bf x}_\ast),\quad i = 1,\ldots, n,
\end{equation}
equivalently
\begin{equation}
\label{const-horizon}
\frac{\tilde f_{\ast, i}}{\alpha_i x_{\ast, i} } = G(t_\ast, {\bf x}_\ast) \equiv C_\ast 
\end{equation}
provided $x_{\ast, i} \not = 0$.
Note that the identity in $t$-term, namely the $0$-th component, automatically holds.
Because at least one $x_i$ is not $0$ on the horizon, the constant $C_\ast$ is determined independently among different $i$'s.
On the other hand, only the quasi-homogeneous part $f_{\alpha, k}$ of $f$ involves equilibria on the horizon.
In general, we have
\begin{align}
\notag
\tilde f_{i;\alpha, k}(t, {\bf x}) &= \kappa_\alpha^{-(k+\alpha_i)} f_{i;\alpha, k}(t, \kappa_\alpha^{\Lambda_\alpha}{\bf x} )\\ 	
\notag
	&= \kappa_\alpha^{-(k+\alpha_i)}\kappa_\alpha^{k+\alpha_i} f_{i;\alpha, k}(t, {\bf x})\\
\label{identity-f-horizon}
	&= f_{i;\alpha, k}(t, {\bf x})
\end{align}
for $(t, {\bf x})=(t, x_1,\ldots, x_n)\in \mathcal{E}$.
The identity (\ref{balance-para}) is then rewritten as follows:
\begin{equation}
\label{const-horizon-0-Cast}
 f^{\rm ext}_{\alpha,k}(t_\ast, {\bf x}_\ast) = G(t_\ast, {\bf x}_\ast) \Lambda^{\rm ext}_\alpha \begin{pmatrix}
t_\ast \\ {\bf x}_\ast
\end{pmatrix} = C_\ast \Lambda^{\rm ext}_\alpha \begin{pmatrix}
t_\ast \\ {\bf x}_\ast
\end{pmatrix}.
\end{equation}
Introducing a scaling parameter $r_{{\bf x}_\ast} (> 0)$, we have
\begin{equation*}
f^{\rm ext}_{\alpha,k}(t_\ast, {\bf x}_\ast) = r_{{\bf x}_\ast}^{-(kI + \Lambda^{\rm ext}_\alpha)} f^{\rm ext}_{\alpha,k}(t_\ast, r_{{\bf x}_\ast}^{\Lambda_\alpha} {\bf x}_\ast )
\end{equation*}
from the quasi-homogeneity.
Note that this identity makes sense including the $0$-th component (cf. Remark \ref{rem-QH-0th}).
Substituting this identity into (\ref{balance-para}), we have
\begin{equation*}
r_{{\bf x}_\ast}^{-(k+\alpha_i)} f_{i;\alpha,k}(t_\ast, r_{{\bf x}_\ast}^{\alpha_1}x_{\ast,1}, \ldots, r_{{\bf x}_\ast}^{\alpha_n}x_{\ast,n}) = \alpha_i x_{\ast,i} C_\ast,\quad i=1,\ldots, n.
\end{equation*}
Here we assume that $r_{{\bf x}_\ast}$ satisfies the following equation:
\begin{equation}
\label{balance-C01}
r_{{\bf x}_\ast}^k C_\ast = \frac{1}{k},
\end{equation}
which implies that $r_{{\bf x}_\ast}$ is uniquely determined once $C_\ast$ is given, {\em provided} $C_\ast > 0$.
The positivity of $C_\ast$ is nontrivial in general, while we have the following result.

\begin{lem}
\label{lem-Cast-nonneg}
Let ${\bf p}_\ast = (t_\ast, {\bf x}_\ast) \in \mathcal{E}$ be an equilibrium for $g^{\rm ext}$ such that it is located on a compact connected NHIM $M\subset \mathcal{E}$ and that the local stable manifold $W^s_{\rm loc}({\bf p}_\ast; g^{\rm ext})$ satisfies $W^s_{\rm loc}({\bf p}_\ast; g^{\rm ext})\cap \mathcal{D}\not = \emptyset$.
Then $C_\ast \equiv G({\bf p}_\ast) \geq 0$.
\end{lem}
\begin{proof}
Assume that the statement is not true, namely $C_\ast < 0$.
We can choose a solution $(t(\tau), {\bf x}(\tau))$ asymptotic to ${\bf p}_{\ast}$ whose initial point $(t_0, {\bf x}(t_0))$ satisfies $\kappa_\alpha({\bf x}(t_0)) < \infty$ by assumption.
Along such a solution, we integrate $G({\bf p}(\tau))$ with ${\bf p}(\tau) = (t(\tau), {\bf x}(\tau))$.
Lemma \ref{lem-G} indicates that
\begin{equation*}
\int_{\tau_0}^\tau G({\bf p}(\eta))d\eta = \ln \kappa_\alpha({\bf x}(\tau)) - \ln \kappa_\alpha({\bf x}(\tau_0)).
\end{equation*}
Note that $\kappa_\alpha$ is independent of $t$ in an explicit manner, as mentioned after Lemma \ref{lem-G}.
By the continuity of $G$, $G({\bf p}(\tau))$ is always negative along ${\bf p}(\tau)$ in a small neighborhood of ${\bf p}_\ast$ in $W^s_{\rm loc}({\bf p}_\ast; g^{\rm ext})$.
On the other hand, ${\bf p}(\tau)\to {\bf p}_\ast\in \mathcal{E}$ holds as $\tau\to +\infty$, implying $\kappa_\alpha = \kappa_\alpha({\bf x}(\tau))\to +\infty$. 
The real-valued function $\ln r$ is monotonously increasing in $r$, and hence $\ln \kappa_\alpha({\bf x}(\tau))$ diverges to $+\infty$ as $\tau\to \infty$, which contradicts the fact that the integral of $G({\bf p}(\tau))$ is negative.
\end{proof}
At this moment, we cannot exclude the possibility that $C_\ast = 0$, similar to the autonomous case (\cite{asym2}).
Now we {\em assume} $C_\ast \not = 0$.
Then $C_\ast > 0$ holds and $r_{{\bf x}_\ast}$ in (\ref{balance-C01}) is well-defined.
Finally the equation (\ref{balance-para}) is written by 
\begin{align*}
\frac{ \alpha_i}{k} r_{{\bf x}_\ast}^{\alpha_i} x_{\ast, i}  = f_{i;\alpha,k}(t_\ast, r_{{\bf x}_\ast}^{\alpha_1} x_{\ast, 1}, \ldots, r_{{\bf x}_\ast}^{\alpha_n} x_{\ast, n}),\quad i=1,\ldots, n,
\end{align*}
which is nothing but the balance law (\ref{0-balance}).
As a summary, we have the one-to-one correspondence among roots of the balance law and equilibria on the horizon for the desingularized vector field (\ref{desing-para-nonaut}).

\begin{thm}[One-to-one correspondence of the balance]
\label{thm-balance-1to1}
Let $(t_\ast, {\bf x}_\ast)$ with ${\bf x}_\ast = (x_{\ast,1}, \ldots, x_{\ast,n})^T$ be an equilibrium on the horizon for the desingularized vector field (\ref{desing-para-nonaut}).
Assume that $C_\ast$ in (\ref{const-horizon}) is positive so that $r_{{\bf x}_\ast} = (kC_\ast)^{-1/k} >0$ is well-defined.
Then the vector $(t_\ast, {\bf Y}_0)$ with
\begin{equation}
\label{x-to-C}
{\bf Y}_0 = (Y_{0,1},\ldots, Y_{0,n})^T = (r_{{\bf x}_\ast}^{\alpha_1}x_{\ast,1},\ldots, r_{{\bf x}_\ast}^{\alpha_n}x_{\ast,n})^T \equiv r_{{\bf x}_\ast}^{\Lambda_\alpha}{\bf x}_\ast 
\end{equation}
is a root of the balance law (\ref{0-balance}).
\par
Conversely, let $(t_\ast, {\bf Y}_0)$ with ${\bf Y}_0\not \equiv 0$ be a root of the balance law (\ref{0-balance}). 
Then the vector $(t_\ast, {\bf x}_\ast)^T$ with
\begin{equation}
\label{C-to-x}
(x_{\ast,1},\ldots, x_{\ast,n})^T = (r_{{\bf Y}_0}^{-\alpha_1}Y_{0,1},\ldots, r_{{\bf Y}_0}^{-\alpha_n}Y_{0,n})^T \equiv r_{{\bf Y}_0}^{-\Lambda_\alpha}{\bf Y}_0
\end{equation}
is an equilibrium on the horizon for (\ref{desing-para-nonaut}), where $r_{{\bf Y}_0} = p_\alpha({\bf Y}_0) > 0$.
Finally, the quantity $C_\ast = G(t_\ast, {\bf x}_\ast)$ constructed by $(t_\ast, {\bf x}_\ast)$ through (\ref{C-to-x}) is positive.
\end{thm}

\begin{proof}
We have already seen how the first statement is derived, and hence we shall prove the second statement here.
First let
\begin{equation*}
r_{{\bf Y}_0} \equiv p_\alpha({\bf Y}_0) > 0,\quad \bar Y_{0,i} := \frac{Y_{0,i}}{r_{{\bf Y}_0}^{\alpha_i}}.
\end{equation*}
By definition $p_\alpha(\bar {\bf Y}_0) = 1$, where $\bar {\bf Y}_0 = (\bar Y_{0,1},\ldots, \bar Y_{0,n})^T$.
Substituting $\bar {\bf Y}_0$ into the right-hand side of (\ref{balance-para}) with $t = t_\ast$, we have
\begin{align*}
\alpha_i \bar Y_{0,i} \sum_{j\in I_\alpha} \frac{\bar Y_{0,j}^{2\beta_j-1}}{\alpha_j}\tilde f_j(t_\ast, \bar {\bf Y}_0) = \alpha_i \bar Y_{0,i} \sum_{j\in I_\alpha} \frac{\bar Y_{0,j}^{2\beta_j-1}}{\alpha_j} \tilde f_{j; \alpha, k}(t_\ast, \bar {\bf Y}_0),
\end{align*}
where we have used the identity $p_\alpha(\bar {\bf Y}_0) = 1$ and (\ref{identity-f-horizon}).
From quasi-homogeneity of $f_{\alpha, k}$ and the balance law (\ref{0-balance}), we further have
\begin{align*}
 \alpha_i \bar Y_{0,i} \sum_{j\in I_\alpha} \frac{\bar Y_{0,j}^{2\beta_j-1}}{\alpha_j} \tilde f_{j; \alpha, k}(t_\ast, \bar {\bf Y}_0) 
&= \alpha_i \bar Y_{0,i} \sum_{j\in I_\alpha} \frac{\bar Y_{0,j}^{2\beta_j-1}}{\alpha_j} r_{{\bf Y}_0}^{-(k+\alpha_j)}f_{j; \alpha, k}(t_\ast, {\bf Y}_0) \\
&= \alpha_i \bar Y_{0,i} \sum_{j\in I_\alpha} \frac{\bar Y_{0,j}^{2\beta_j-1}}{\alpha_j} r_{{\bf Y}_0}^{-(k+\alpha_j)} \frac{\alpha_j}{k}Y_{0,j} 
 = \alpha_i \bar Y_{0,i} \frac{r_{{\bf Y}_0}^{-k}}{k}\sum_{j\in I_\alpha} \bar Y_{0,j}^{2\beta_j}\\
 &= \alpha_i \bar Y_{0,i} \frac{r_{{\bf Y}_0}^{-k}}{k} = r_{{\bf Y}_0}^{-(k+\alpha_i)} \frac{\alpha_i}{k} Y_{0,i}\\
&= r_{{\bf Y}_0}^{-(k+\alpha_i)} f_{i; \alpha, k}(t_\ast, {\bf Y}_0) = f_{i; \alpha, k}(t_\ast, \bar {\bf Y}_0) = \tilde f_{i; \alpha, k}(t_\ast, \bar {\bf Y}_0),
\end{align*}
implying that $(t_\ast, \bar {\bf Y}_0)$ is a root of (\ref{balance-para}).
\par
For the last statement, we directly calculate $G(t_\ast, {\bf x}_\ast)$:
\begin{align*}
G(t_\ast, {\bf x}_\ast) &= \sum_{j\in I_\alpha} \frac{x_{\ast, j}^{2\beta_j-1}}{\alpha_j}\tilde f_{j,\ast}\\
	&= \sum_{j\in I_\alpha} \frac{x_{\ast, j}^{2\beta_j-1}}{\alpha_j}\tilde f_{j; \alpha, k}(t_\ast, {\bf x}_\ast)\quad \text{(from characterization of equilibria on the horizon)}\\
	&= \sum_{j\in I_\alpha} \frac{r_{{\bf Y}_0}^{-\alpha_j(2\beta_j-1)}Y_{0,j}^{2\beta_j-1}}{\alpha_j}\tilde f_{j; \alpha, k}(t_\ast, r_{{\bf Y}_0}^{-\Lambda_\alpha}{\bf Y}_0)\\
	&= r_{{\bf Y}_0}^{-2c} \sum_{j\in I_\alpha} \frac{r_{{\bf Y}_0}^{\alpha_j}Y_{0,j}^{2\beta_j-1}}{\alpha_j} r_{{\bf Y}_0}^{-(k+\alpha_j)} \tilde f_{j; \alpha, k}(t_\ast, {\bf Y}_0) \quad \text{(from quasi-homogeneity)}\\
	&= r_{{\bf Y}_0}^{-k-2c} \sum_{j\in I_\alpha} \frac{Y_{0,j}^{2\beta_j-1}}{\alpha_j} \tilde f_{j; \alpha, k}(t_\ast, {\bf Y}_0)\\
	&= r_{{\bf Y}_0}^{-k-2c} \sum_{j\in I_\alpha} \frac{Y_{0,j}^{2\beta_j-1}}{\alpha_j} \cdot \frac{\alpha_j}{k} Y_{0,j}\quad \text{(from (\ref{0-balance}))} \\
	&= \frac{1}{k} r_{{\bf Y}_0}^{-k-2c} \sum_{j\in I_\alpha} Y_{0,j}^{2\beta_j}\\
	&= \frac{1}{k} r_{{\bf Y}_0}^{-k} > 0
\end{align*}
because ${\bf Y}_0 \not = {\bf 0}_n$.
\end{proof}

\subsection{Structure of $Dg^{\rm ext}_\ast$ and technical assumptions}
\label{section-Dg}
Before discussing the correspondence of eigenstructures, we shall summarize structures of $Dg^{\rm ext}$.
The argument in \cite{asym2} indicates that $Dg^{\rm ext}$ includes the objects related to the matrix $A^{\rm ext}$.
While a slight modification will be necessary in the nonautonomous case, we obtain the similar decomposition of $Dg^{\rm ext}$ so that the eigenstructures can be extracted through $A^{\rm ext}$.
\par
Here we calculate and investigate details of $Dg^{\rm ext}$, in particular at an equilibrium on the horizon.
Let $(t_\ast, {\bf x}_\ast) \in \mathcal{E}$ be an equilibrium on the horizon $\mathcal{E}$ for $g^{\rm ext}$, and ${\bf v}^{\rm ext}_{\ast,\alpha}\in \mathbb{R}^{n+1}$ be a vector given by
\begin{equation}
\label{v-ast}
{\bf v}^{\rm ext}_{\ast,\alpha} := \Lambda_\alpha^{\rm ext} \begin{pmatrix}
t_\ast \\ {\bf x}_\ast
\end{pmatrix} 
= \begin{pmatrix}
0 \\ \Lambda_\alpha {\bf x}_\ast 
\end{pmatrix} \equiv  \begin{pmatrix}
0 \\ {\bf v}_{\ast, \alpha} 
\end{pmatrix}.
\end{equation}
It follows from (\ref{desing-para-nonaut}) that
\begin{align}
\notag
Dg^{\rm ext}(t, {\bf x}) &= (2c-1)p_\alpha(t, {\bf x})^{2c-1}\tilde f^{\rm ext}(t, {\bf x})D p_\alpha(t, {\bf x})^T +  \left( 1-\frac{2c-1}{2c}(1-p_\alpha(t, {\bf x})^{2c})\right)  D\tilde f^{\rm ext}(t, {\bf x})\\
\label{Dg-general}
	&\quad  - (0, \alpha_1 x_1, \ldots, \alpha_n x_n)^T DG(t, {\bf x})^T - G(t, {\bf x})\Lambda^{\rm ext}_\alpha, \\
\notag
Dg^{\rm ext}_\ast &= (2c-1)\tilde f^{\rm ext}_\ast (D p_\alpha)_\ast^T + D\tilde f^{\rm ext}_\ast - {\bf v}^{\rm ext}_{\ast, \alpha} (D G_\ast)^T - C_\ast \Lambda^{\rm ext}_\alpha \\
\notag
 	&\quad \text{(from the definition of ${\bf v}^{\rm ext}_{\ast, \alpha}$ and (\ref{const-horizon}))}\\
\notag
	&= \left\{ - C_\ast \Lambda^{\rm ext}_\alpha + D\tilde f^{\rm ext}_\ast \right\} + {\bf v}^{\rm ext}_{\ast, \alpha} \left( (2c-1) C_\ast (D p_\alpha)_\ast^T - DG_\ast \right)^T. \quad \text{(from (\ref{const-horizon-0}))}
\end{align}
Next, using (\ref{Gx}) and (\ref{const-horizon}), we have\footnote{
The expression of $DG_\ast$ is formally written including $j\not \in I_\alpha$ for simplicity.
The actual form of the $j$-th component of $DG_\ast$ with $j\not \in I_\alpha$ is $0$.
}
\begin{align*}
DG_\ast &= {\rm diag}\left(0, \frac{2\beta_1-1}{\alpha_1}x_{\ast, 1}^{2\beta_1-2},\ldots, \frac{2\beta_n-1}{\alpha_n}x_{\ast, n}^{2\beta_n-2}\right)C_\ast {\bf v}^{\rm ext}_{\ast, \alpha} \\
	&\quad + (A_g^{\rm ext} + C_\ast  \Lambda^{\rm ext}_\alpha)^T \left(0, \frac{x_{\ast, 1}^{2\beta_1-1}}{\alpha_1},\ldots, \frac{x_{\ast, n}^{2\beta_n-1}}{\alpha_n}\right)^T \\
	&= C_\ast \left(0, 2\beta_1 x_{\ast, 1}^{2\beta_1-1},\ldots, 2\beta_n x_{\ast, n}^{2\beta_n-1}\right)^T
	+ (A_g^{\rm ext})^T  (D p_\alpha)_\ast \quad \text{(using (\ref{grad-xast}))} \\
	&= 2cC_\ast (D p_\alpha)_\ast + (A_g^{\rm ext})^T (D p_\alpha)_\ast,
\end{align*}
where
\begin{align}
\label{Ag-1}
A_g^{\rm ext} &:= -C_\ast \Lambda^{\rm ext}_\alpha + \begin{pmatrix}
0 & {\bf 0}_n^T\\
(D_t \tilde f)_\ast & (D_{\bf x}\tilde f)_\ast
\end{pmatrix}
	\equiv \begin{pmatrix}
0 & {\bf 0}_n^T\\
(D_t \tilde f)_\ast & A_g
\end{pmatrix}.
\end{align}
The Jacobian matrix $Dg^{\rm ext}_\ast$ then has a decomposition 
\begin{equation*}
Dg^{\rm ext}_\ast = A_g^{\rm ext} + B_g^{\rm ext} + \delta_{1,k} A_g^{\rm res},
\end{equation*}
where
\begin{align}
\label{Bg-1}
B_g^{\rm ext} &= - {\bf v}^{\rm ext}_{\ast, \alpha} (D p_\alpha)_\ast^T (A_g^{\rm ext} + C_\ast I_{n+1})
	\equiv \begin{pmatrix}
0 & {\bf 0}_n^T\\
\ast & B_g
\end{pmatrix},\\
\label{Agres-1}
A_g^{\rm res}  &= \begin{pmatrix}
0 & -2c (D_{\bf x} p_\alpha)_\ast^T \\
{\bf 0}_n & O_{n, n}
\end{pmatrix}.
\end{align}
Note that $A_g^{\rm res}$ stems from $\tilde f_0 = O(\kappa^{-k})$ as $\kappa\to \infty$, which has no contribution when $k>1$.
Here we have the following proposition by the same arguments as that discussed in \cite{asym2}.
\begin{prop}[\cite{asym2}]
\label{prop-proj-B}
Let ${\bf v}^{\rm ext}_{\ast,\alpha}$, be the vector given in (\ref{v-ast}), 
and
\begin{equation*}
P^{\rm ext}_\ast := {\bf v}^{\rm ext}_{\ast,\alpha}(D p_\alpha)_\ast^T.
\end{equation*} 
Then $P^{\rm ext}_\ast$ as the linear mapping on $\mathbb{R}^{n+1}$ is the (nonorthogonal) projection\footnote{
In the \lq\lq homogeneous" case $\alpha = (0, 1,\ldots, 1)$ (except the scaling of $t$), this is orthogonal.
} onto ${\rm span}\{{\bf v}^{\rm ext}_{\ast,\alpha}\}$.
Similarly, the map $I-P^{\rm ext}_\ast$ is the (nonorthogonal) projection onto the tangent space $T_{(t_\ast, {\bf x}_\ast)}\mathcal{E}$ along ${\rm span}\{{\bf v}^{\rm ext}_{\ast,\alpha}\}$. 
Moreover, in the matrix form, $P^{\rm ext}_\ast$ is written as follows:
\begin{equation}
\label{proj}
P^{\rm ext}_\ast = \begin{pmatrix}
0 & {\bf 0}_n^T\\
{\bf 0}_n & P_\ast 
\end{pmatrix}.
\end{equation}
\end{prop}

Before the proof, it should be noted that
the inner product of the gradient $(D p_\alpha)_\ast$ at an equilibrium $(t_\ast, {\bf x}_\ast) \in \mathcal{E}$ given in (\ref{grad-xast}) and the vector ${\bf v}^{\rm ext}_{\ast,\alpha}$ is unity:
\begin{equation}
\label{inner-gradp-v}
(D p_\alpha)_\ast^T {\bf v}^{\rm ext}_{\ast, \alpha} = \frac{1}{c}\sum_{j\in I_\alpha} \beta_j x_{\ast, j}^{2\beta_j-1} \alpha_j x_{\ast, j} = \frac{c}{c} \sum_{j\in I_\alpha} x_{\ast, j}^{2\beta_j} = 1.
\end{equation}

\begin{proof}
The first two properties follow from the identity $(D p_\alpha)_\ast^T {\bf v}^{\rm ext}_{\ast, \alpha} = 1$ and the fact that ${\bf v}\in T_{(t_\ast, {\bf x}_\ast)}\mathcal{E}$ satisfies $(Dp_\alpha)_\ast^T {\bf v} = 0$.
We shall provide the proof of the last statement, which follows from the direct calculation:
\begin{align*}
{\bf v}^{\rm ext}_{\ast,\alpha} (D p_\alpha)_\ast^T = \frac{1}{c}\begin{pmatrix}
0 \\ \alpha_1 x_{\ast, 1} \\ \vdots \\ \alpha_n x_{\ast, n} 
\end{pmatrix}
\begin{pmatrix}
0, \beta_1x_1^{2\beta_1-1}, \ldots, \beta_nx_n^{2\beta_n-1}
\end{pmatrix} &= \begin{pmatrix}
0 & {\bf 0}_n^T \\
{\bf 0}_n & {\bf v}_{\ast, \alpha} (D_{\bf x} p_\alpha)_\ast^T
\end{pmatrix},
\end{align*}
with $P_\ast \equiv {\bf v}_{\ast, \alpha} (D_{\bf x} p_\alpha)_\ast^T$.
\end{proof}


Using the projection $P^{\rm ext}_\ast$, we conclude that the Jacobian matrix $Dg^{\rm ext}_\ast$ is decomposed as follows:
\begin{align}
\notag
Dg^{\rm ext}_\ast &= A^{\rm ext}_g + B^{\rm ext}_g + \delta_{1,k} A_g^{\rm res}\\
\notag
	&= (I-P_\ast^{\rm ext})A^{\rm ext}_g - C_\ast P_\ast^{\rm ext} +  \delta_{1,k} A_g^{\rm res},\\
\label{dec-Dg}
A^{\rm ext}_g &= \begin{pmatrix}
0 & {\bf 0}_n^T\\
(D_t\tilde f)_\ast & A_g
\end{pmatrix},\quad B^{\rm ext}_g = - P^{\rm ext}_\ast (A^{\rm ext}_g + C_\ast I),\\
\notag
A_g^{\rm res} &= \begin{pmatrix}
0 & -2c (D_{\bf x} p_\alpha)_\ast^T \\
{\bf 0}_n & O_{n,n}
\end{pmatrix}.
\end{align}

The balance law determines the coefficients of type-I blow-ups (cf. \cite{asym1}), which turns out to correspond to equilibria on the horizon for the desingularized vector field.
This correspondence provides a relationship among two different vector fields involving blow-ups.
Arguments here will indicate the correspondence of associated eigenstructures {\em under technical assumptions}.
To see this, we make the following assumption to $f$, which is essential to the following arguments coming from the technical restriction due to the form of parabolic embeddings, while it can be relaxed for general systems.
\begin{ass}[Order restriction, growth in residual terms]
\label{ass-f-1}
$k\geq 1$ is assumed\footnote{
Our description of blow-ups by means of dynamics at infinity relies on local (center-)stable manifolds of {\em compact} invariant sets on the horizon $\mathcal{E}$. 
When $k=0$, the global-in-time trajectories approaching $\mathcal{E}$ must provide $t_{\max} = \infty$ from (\ref{time-desing-nonaut}) or the formula (\ref{tmax}) below. 
Unlike the autonomous cases, $\mathcal{E}$ explicitly depends on $t$ in the nonautonomous setting. 
In particular, the case $k = 0$ in the nonautonomous setting will cause the divergence of $(t(\tau), {\bf x}(\tau))$ in $t$-direction, and the present framework, in particular Theorem \ref{thm:blowup}, cannot be applied.
}. 
For each $i = 1,\ldots, n$,
\begin{equation}
\label{fi-order}
\tilde f_{i;{\rm res}}(t, {\bf x}) = O\left(\kappa_\alpha({\bf x})^{-(1+\epsilon)} \right),\quad \frac{\partial \tilde f_{i;{\rm res}}}{\partial x_l}(t, {\bf x}) = o\left(\kappa_\alpha({\bf x})^{(1+\epsilon)}\right),\quad l=0,\ldots, n
\end{equation}
hold for some $\epsilon > 0$ as $(t, {\bf x})$ approaches to $\mathcal{E}$, where the alias $x_0 = t$ is used.
\end{ass}

\begin{rem}
\label{rem-order-typical}
The estimates (\ref{fi-order}) are typically satisfied if the terms of order $k+\alpha_i - 1$ of $f_i$ are identically $0$.
\end{rem}

Under this assumption, we have the following identity in derivatives.
\begin{lem}
\label{lem-ass-f-1}
Let $(t_\ast, {\bf x}_\ast)\in \mathcal{E}$ be an equilibrium for $g^{\rm ext}$.
Under Assumption \ref{ass-f-1}, we have $(D \tilde f)_\ast = (D \tilde f_{\alpha, k})_\ast$. 
In particular, we have
\begin{equation}
\label{Ag-QH}
A^{\rm ext}_g = -C_\ast \Lambda^{\rm ext}_\alpha + (D\tilde f_{\alpha, k}^{\rm ext})_\ast.
\end{equation}
\end{lem}
Remark that the above statement discusses about $f$, not the extended vector field $f^{\rm ext}$.

\begin{proof}
Now $\tilde f_{i;{\rm res}}$ is expressed as
\begin{align*}
\tilde f_{i;{\rm res}}(t, {\bf x}) &\equiv \kappa_\alpha^{-(k+\alpha_i)} f_{i,{\rm res}}(t, \kappa_\alpha^{\Lambda_\alpha}{\bf x})\quad \text{(by (\ref{f-tilde}) and $\alpha_0 = 0$)} \\
	&\equiv \kappa_\alpha^{-(1+\epsilon)} \tilde f_{i;{\rm res}}^{(1)} (t, {\bf x})
\end{align*}
with
\begin{equation*}
\tilde f_{i;{\rm res}}^{(1)}(t, {\bf x}) = O(1),\quad \frac{\partial \tilde f_{i;{\rm res}}^{(1)}}{\partial x_l}(t, {\bf x}) = o(\kappa_\alpha({\bf x})^{1+\epsilon}), \quad l=0,\ldots, n
\end{equation*}
as $(t, {\bf x})$ approaches to $\mathcal{E}$.
The partial derivative of the component $\tilde f_i$ with respect to $x_l$ at $(t_\ast, {\bf x}_\ast)$ is
\begin{equation*}
\frac{\partial \tilde f_i}{\partial x_l}(t_\ast, {\bf x}_\ast) = \frac{\partial \tilde f_{i;\alpha, k}}{\partial x_l}(t_\ast, {\bf x}_\ast) + 
(1+\epsilon)\kappa_\alpha^{-\epsilon}\frac{\partial \kappa_\alpha^{-1}}{\partial x_l}\tilde f_{i;{\rm res}}^{(1)}(t_\ast, {\bf x}_\ast)  + \kappa_\alpha^{-(1+\epsilon)} \frac{\partial \tilde f_{i;{\rm res}}^{(1)}}{\partial x_l}(t_\ast, {\bf x}_\ast ).
\end{equation*}
Using the fact that $\kappa_\alpha^{-1} = 0$ on the horizon, 
our present assumption implies that the \lq\lq gap" terms
\begin{equation*}
(1+\epsilon)\kappa_\alpha^{-\epsilon}\frac{\partial \kappa_\alpha^{-1}}{\partial x_l}\tilde f_{i;{\rm res}}^{(1)}(t_\ast, {\bf x}_\ast)  + \kappa_\alpha^{-(1+\epsilon)} \tilde f_{i;{\rm res}}^{(1)}(t_\ast, {\bf x}_\ast )
\end{equation*}
are identically $0$ on the horizon and hence the Jacobian matrix $D_{\bf x}\tilde f(t_\ast, {\bf x}_\ast)$ with respect to ${\bf x}$ coincides with $D_{\bf x}\tilde f_{\alpha, k}(t_\ast, {\bf x}_\ast)$.
$D_{\bf t}\tilde f(t_\ast, {\bf x}_\ast)$ is simpler because $\kappa_\alpha$ is independent of $t$ and
\begin{equation*}
\frac{\partial \tilde f_i}{\partial t}(t_\ast, {\bf x}_\ast) = \frac{\partial \tilde f_{i;\alpha, k}}{\partial t}(t_\ast, {\bf x}_\ast) + \kappa_\alpha^{-(1+\epsilon)} \frac{\partial \tilde f_{i;{\rm res}}^{(1)}}{\partial t}(t_\ast, {\bf x}_\ast ).
\end{equation*}
\end{proof}

Through this result, $\tilde f^{\rm ext}$ can be identified with $\tilde f_{\alpha, k}^{\rm ext}$ within our present interests.
Also, as mentioned in Remark \ref{rem-invariance}, the horizon $\mathcal{E}$ is a codimension one invariant manifold for $g^{\rm ext}$ and hence the whole eigenstructure is decomposed into the following two types:
\begin{itemize}
\item $n$-independent (generalized) eigenvectors\footnote{
In nonautonomous systems, the $t$-contribution increases the number of remaining eigenvectors. 
} of $Dg^{\rm ext}_\ast$ spanning the tangent space $T_{(t_\ast, {\bf x}_\ast)}\mathcal{E}$.
\item An eigenvector transversal to $T_{(t_\ast, {\bf x}_\ast)}\mathcal{E}$.
\end{itemize}

\subsection{Eigenstructure in \lq\lq transversal" direction; case 1: $k>1$}
\label{section-special-k>1}

We first investigate the eigenstructure of $Dg^{\rm ext}_\ast$ transversal to $T_{(t_\ast, {\bf x}_\ast)}\mathcal{E}$, which shall be called the {\em transversal} eigenpair\footnote{
In \cite{asym2}, this eigenpair was called a \lq\lq {\em common}" eigenpair.
}.
We know that this eigenstructure can be extracted for {\em any} given systems, but the structure is different among cases $k>1$ and $k=1$, because $\tilde f_0(t,{\bf x}) = \kappa_\alpha^{-k}$ in any nonautonomous system (\ref{ODE-original}) from the form of the extended autonomous system (\ref{nonaut-extend}), and its gradient in the ${\bf x}$-direction does not vanish on the horizon when $k=1$.
\par
Here we pay our attention to the simpler case $k>1$, where the structure is easily extracted based on arguments in autonomous cases \cite{asym2}. 
We begin with the special eigenstructure of $A^{\rm ext}$, which is described regardless of $k$.

\begin{prop}[Eigenvalue $1$. cf. \cite{asym2}]
\label{prop-ev1}
Consider an asymptotically quasi-homogeneous vector field $f^{\rm ext}$ of type $\alpha = (0, \alpha_1,\ldots, \alpha_n)$ and order $k+1$.
Suppose that a nontrivial root $(t, {\bf Y}_0)$ of the balance law (\ref{0-balance}) is given.
Then the corresponding blow-up power-determining matrix $A^{\rm ext}$ has an eigenvalue $1$ with the associating eigenvector
\begin{equation}
\label{vector-ev1}
{\bf v}^{\rm ext}_{0,\alpha} = \Lambda^{\rm ext}_\alpha \begin{pmatrix}
t \\ {\bf Y}_0
\end{pmatrix} = \Lambda^{\rm ext}_\alpha \begin{pmatrix}
0 \\ {\bf Y}_0
\end{pmatrix}.
\end{equation} 
\end{prop}

Note that the matrix $A^{\rm ext}$ only involves the quasi-homogeneous part $f^{\rm ext}_{\alpha, k}$ of $f^{\rm ext}$.

\begin{proof}
Consider \eqref{temporay-label4} at $\mathbf{y} = \mathbf{Y}_0$
with the help of (\ref{0-balance}):
\begin{equation}
D f^{\rm ext}_{\alpha,k} (t, \mathbf{Y}_0) \Lambda^{\rm ext}_\alpha\begin{pmatrix}
t \\ \mathbf{Y}_0 
\end{pmatrix} 
= \left(  k I+ \Lambda^{\rm ext}_\alpha  \right) f^{\rm ext}_{\alpha,k}(t, \mathbf{Y}_0)
= \left( I + \frac{1}{k}\Lambda^{\rm ext}_\alpha \right) \Lambda^{\rm ext}_\alpha \begin{pmatrix}
t \\ \mathbf{Y}_0
\end{pmatrix}.
\end{equation}
Then, using the definition of $A^{\rm ext}$, we have
\begin{align*}
A^{\rm ext} \Lambda^{\rm ext}_\alpha \begin{pmatrix}
t \\ \mathbf{Y}_0
\end{pmatrix}
&= \left( -\frac{1}{k} \Lambda^{\rm ext}_\alpha + Df^{\rm ext}_{\alpha,k}(t, \mathbf{Y}_0) \right) \Lambda^{\rm ext}_\alpha \begin{pmatrix}
t \\ \mathbf{Y}_0
\end{pmatrix}\\
&= -\frac{1}{k} (\Lambda^{\rm ext}_\alpha)^2 \begin{pmatrix}
t \\ \mathbf{Y}_0
\end{pmatrix} + \left(I+ \frac{1}{k} \Lambda^{\rm ext}_\alpha  \right) \Lambda^{\rm ext}_\alpha \begin{pmatrix}
t \\ \mathbf{Y}_0
\end{pmatrix}\\
&= \Lambda^{\rm ext}_\alpha \begin{pmatrix}
t \\ \mathbf{Y}_0
\end{pmatrix},
\end{align*}
which shows the desired statement.
\end{proof}



The \lq\lq transversal" eigenstructure of $Dg^{\rm ext}_\ast$ with $k>1$ is characterized as follows.
\begin{prop}[Transversal eigenpair, $k>1$]
\label{prop-ev-special-1}
Suppose that Assumption \ref{ass-f-1} holds and $k>1$.
Also suppose that $(t_\ast, {\bf x}_\ast) \in \mathcal{E}$ is an equilibrium on the horizon for the associated desingularized vector field $g^{\rm ext}$ in (\ref{desing-para-nonaut}).
Then the Jacobian matrix $Dg^{\rm ext}_\ast$ always possesses the eigenpair $\{-C_\ast, {\bf v}^{\rm ext}_{\ast,\alpha}\}$, where ${\bf v}^{\rm ext}_{\ast,\alpha}$ is given in (\ref{v-ast}).
\end{prop}

\begin{proof}
Using (\ref{Ag-QH}), the same idea as the proof of Proposition \ref{prop-ev1} can be applied to obtaining
\begin{align*}
A_g^{\rm ext} \Lambda^{\rm ext}_\alpha  \begin{pmatrix}
t_\ast \\ {\bf x}_\ast
\end{pmatrix}
&= \left\{ - C_\ast \Lambda^{\rm ext}_\alpha + (D\tilde f^{\rm ext}_{\alpha,k})_\ast \right\} \Lambda^{\rm ext}_\alpha \begin{pmatrix}
t_\ast \\ {\bf x}_\ast
\end{pmatrix} \quad \text{(from (\ref{Ag-QH}))} \\
&= - C_\ast  (\Lambda^{\rm ext}_\alpha)^2 \begin{pmatrix}
t_\ast \\ {\bf x}_\ast
\end{pmatrix}+ (k I+ \Lambda^{\rm ext}_\alpha ) \tilde f^{\rm ext}_{\alpha, k}(t_\ast, {\bf x}_\ast) \quad \text{(from (\ref{temporay-label4}))} \\ 
&= - C_\ast (\Lambda^{\rm ext}_\alpha)^2 \begin{pmatrix}
t_\ast \\ {\bf x}_\ast
\end{pmatrix} + C_\ast(k I+ \Lambda^{\rm ext}_\alpha ) \Lambda^{\rm ext}_\alpha \begin{pmatrix}
t_\ast \\ {\bf x}_\ast
\end{pmatrix} \quad \text{(from (\ref{const-horizon-0-Cast}))} \\ 
&= k C_\ast \Lambda^{\rm ext}_\alpha \begin{pmatrix}
t_\ast \\ {\bf x}_\ast
\end{pmatrix},
\end{align*}
which shows that the matrix $A_g^{\rm ext}$ admits an eigenvector ${\bf v}^{\rm ext}_{\ast,\alpha}$ with associated eigenvalue $kC_\ast$.
In other words,
\begin{equation}
\label{vector-evkC}
A_g^{\rm ext} {\bf v}^{\rm ext}_{\ast,\alpha} = kC_\ast {\bf v}^{\rm ext}_{\ast,\alpha}.
\end{equation}
From (\ref{Bg-1}) and (\ref{inner-gradp-v}), we have
\begin{align*}
B_g^{\rm ext} {\bf v}^{\rm ext}_{\ast,\alpha}  &= - P_\ast^{\rm ext} (A_g^{\rm ext} + C_\ast I){\bf v}^{\rm ext}_{\ast,\alpha}\\
	&= - (k+1)C_\ast P_\ast^{\rm ext} {\bf v}^{\rm ext}_{\ast,\alpha}\quad \text{(from (\ref{vector-evkC}))}\\
	 &= - (k+1)C_\ast {\bf v}^{\rm ext}_{\ast, \alpha}. \quad \text{(from (\ref{inner-gradp-v}))} 
\end{align*}
Therefore we have
\begin{align*}
Dg^{\rm ext}_\ast {\bf v}^{\rm ext}_{\ast,\alpha} = (A_g^{\rm ext} + B_g^{\rm ext}) {\bf v}^{\rm ext}_{\ast,\alpha} = \{ kC_\ast - (k+1)C_\ast\} {\bf v}^{\rm ext}_{\ast,\alpha} &= -C_\ast {\bf v}^{\rm ext}_{\ast,\alpha}
\end{align*}
and, as a consequence, the vector ${\bf v}^{\rm ext}_{\ast,\alpha}$ is an eigenvector of $Dg^{\rm ext}_\ast$ associated with $-C_\ast$ and the proof is completed.
\end{proof}

\begin{rem}
In the present case, all information involving the eigenpair $(-C_\ast, {\bf v}^{\rm ext}_{\ast, \alpha})$ are derived from the submatrix $(D_{\bf x} g)_\ast$.
Assuming $t_\ast$ as a parameter, the corresponding result follows from arguments in \cite{asym2} and Lemma \ref{lem-eigen-extend-1}. 
\end{rem}

This theorem and (\ref{inner-gradp-v}) imply that the eigenvector ${\bf v}^{\rm ext}_{\ast,\alpha}$ is transversal to the tangent space $T_{(t_\ast, {\bf x}_\ast)}\mathcal{E}$, in which sense the corresponding eigenpair is called {\em transversal eigenpair}.
Combined with the $Dg^{\rm ext}$-invariance of the tangent bundle $T\mathcal{E}$ (cf. Remark \ref{rem-invariance}), we conclude that the eigenvector ${\bf v}^{\rm ext}_{\ast,\alpha}$ provides the blow-up direction in the linear sense.
Comparing Propositions \ref{prop-ev1} and \ref{prop-ev-special-1}, the eigenpair $\{1, {\bf v}^{\rm ext}_{0,\alpha}\}$ of the blow-up power-determining matrix $A^{\rm ext}$ provides a characteristic information of blow-up solutions.
Similarly, from the eigenpair $\{ -C_\ast, {\bf v}^{\rm ext}_{\ast,\alpha}\}$, a direction of trajectories $(t(\tau), {\bf x}(\tau))$ for (\ref{desing-para-nonaut}) converging to $(t_\ast, {\bf x}_\ast)$ is uniquely determined.
From Theorem \ref{thm-balance-1to1}, the constant $C_\ast$ becomes positive in the present case.



\subsection{Eigenstructure in \lq\lq tangent spaces"}
Here we investigate eigenvectors of $Dg^{\rm ext}_\ast $ in the tangent space $T_{(t_\ast, {\bf x}_\ast)}\mathcal{E}$, which will be referred to as {\em tangent eigenvectors}, and the correspondence among those for $Dg^{\rm ext}_\ast $, $A_g^{\rm ext}$ and $A^{\rm ext}$.
This process is opposite to that in \cite{asym2}, because arguments in the former case are essentially the same as those in autonomous cases \cite{asym2}, whereas the remaining one requires qualitatively different treatments, shown in the next subsection.
\par
\bigskip
In the following argument, we assume Assumption \ref{ass-f-1}.
First note that Lemma \ref{lem-eigen-extend-1} indicates the following correspondence.
\begin{lem}
\label{lem-corr-aut-nonaut}
Under Assumption \ref{ass-f-1}, a pair $(\lambda, {\bf u}^{\rm ext})\in \mathbb{R}\times \mathbb{C}^{n+1}$ with $\lambda \not = 0$ is an eigenpair of $A^{\rm ext}$ if and only if the following statements hold:
\begin{itemize}
\item ${\bf u}^{\rm ext} = (0, {\bf u})\in \mathbb{C}^{1+n}$.
\item $(\lambda, {\bf u})\in \mathbb{R}\times \mathbb{C}^{n}$ is an eigenpair of $A$.
\end{itemize}

\end{lem}
Using this expression, we can apply the same arguments in \cite{asym2} to characterizing the correspondence of eigenstructures among different matrices.
In particular, we obtain the following proposition.

\begin{prop}
\label{prop-corr-Dg-Ag}
For any $\lambda \in \mathbb{C}$ and $N\in \mathbb{N}$, we have
\begin{align}
\label{Dg-I-P-2-1}
(Dg^{\rm ext}_\ast - \lambda I)^N (I -P^{\rm ext}_\ast) 
	&= (I -P^{\rm ext}_\ast) (A^{\rm ext}_g - \lambda I)^N
\end{align}
and
\begin{align}
\label{Dg-I-P-2-2}
(A_g^{\rm ext} - k C_\ast I) ( Dg^{\rm ext}_\ast - \lambda I)^N (I-P^{\rm ext}_\ast) 
	&= (A_g^{\rm ext} - \lambda I)^N (A_g^{\rm ext} - k C_\ast I).
\end{align}
In particular, $A_g^{\rm res}$ provides no contribution to tangential eigenpairs.
\end{prop}

\begin{proof}
We must pay attention to the contribution of $A_g^{\rm res}$ to the influence on $Dg^{\rm ext}_\ast$: 
\begin{align*}
Dg^{\rm ext}_\ast (I-P^{\rm ext}_\ast) &= (I-P^{\rm ext}_\ast) A^{\rm ext}_g  (I-P^{\rm ext}_\ast) + (I-P^{\rm ext}_\ast)A_g^{\rm res} (I-P^{\rm ext}_\ast) -C_\ast P^{\rm ext}_\ast (I-P^{\rm ext}_\ast) \\
	&= (I-P^{\rm ext}_\ast) A^{\rm ext}_g  (I-P^{\rm ext}_\ast) + (I-P^{\rm ext}_\ast)A_g^{\rm res} (I-P^{\rm ext}_\ast).
\end{align*}
Now we have 
\begin{equation*}
(I-P^{\rm ext}_\ast) A^{\rm ext}_g  P_\ast^{\rm ext} = 0
\end{equation*}
from (\ref{vector-evkC}); $A^{\rm ext}_g {\bf v}^{\rm ext}_{\ast, \alpha} = kC_\ast {\bf v}^{\rm ext}_{\ast,\alpha}$, 
namely
\begin{align*}
Dg^{\rm ext}_\ast (I-P^{\rm ext}_\ast) 
	&= (I-P^{\rm ext}_\ast) A^{\rm ext}_g + (I-P^{\rm ext}_\ast)A_g^{\rm res}(I-P^{\rm ext}_\ast).
\end{align*}
On the other hand, the remaining matrices have the following forms:
\begin{align*}
(I-P^{\rm ext}_\ast)A_g^{\rm res} &= \begin{pmatrix}
1 & {\bf 0}_n^T \\
{\bf 0}_n & I_n -P_\ast
\end{pmatrix}\begin{pmatrix}
0 & -2c (D_{\bf x} p_\alpha)_\ast^T \\
{\bf 0}_n & O
\end{pmatrix} = \begin{pmatrix}
0 & -2c (D_{\bf x} p_\alpha)_\ast^T \\
{\bf 0}_n & O
\end{pmatrix},\\
(I-P^{\rm ext}_\ast)A_g^{\rm res}P^{\rm ext}_\ast &= \begin{pmatrix}
1 & {\bf 0}_n^T \\
{\bf 0}_n & I_n -P_\ast
\end{pmatrix}\begin{pmatrix}
0 & -2c (D_{\bf x} p_\alpha)_\ast^T \\
{\bf 0}_n & O
\end{pmatrix} \begin{pmatrix}
0 & {\bf 0}_n^T \\
{\bf 0}_n & P_\ast
\end{pmatrix} = \begin{pmatrix}
0 & -2c (D_{\bf x} p_\alpha)_\ast^T P_\ast  \\
{\bf 0}_n & O
\end{pmatrix},\quad 
\end{align*}
and hence
\begin{equation*}
(I-P^{\rm ext}_\ast)A_g^{\rm res}(I-P^{\rm ext}_\ast) = \begin{pmatrix}
0 & -2c (D_{\bf x} p_\alpha)_\ast^T (I_n - P_\ast ) \\
{\bf 0}_n & O
\end{pmatrix}.
\end{equation*}
Moreover, we have
\begin{align*}
P_\ast &= {\bf v}_{\ast, \alpha} \frac{1}{c}\left( \beta_1 x_{\ast, 1}^{2\beta_1-1}, \ldots, \beta_n x_{\ast, n}^{2\beta_n-n}\right),\\
2c (D_{\bf x} p_\alpha)_\ast^T (I_n - P_\ast) &= 2c (D_{\bf x} p_\alpha)_\ast^T  - 2c (D_{\bf x} p_\alpha)_\ast^T P_\ast \\
	&= 2c (D_{\bf x} p_\alpha)_\ast^T - 2\left( \beta_1 x_{\ast, 1}^{2\beta_1-1}, \ldots, \beta_n x_{\ast, n}^{2\beta_n-n}\right) = {\bf 0}_n^T.
\end{align*}
As a summary, 
\begin{align*}
(Dg^{\rm ext}_\ast - \lambda I) (I-P^{\rm ext}_\ast) 
	&= (I-P^{\rm ext}_\ast) (A^{\rm ext}_g - \lambda I) + (I-P^{\rm ext}_\ast)A_g^{\rm res}(I-P^{\rm ext}_\ast)\\
	&= (I-P^{\rm ext}_\ast) (A^{\rm ext}_g - \lambda I).
\end{align*}
Repeating this procedure, we obtain (\ref{Dg-I-P-2-1}).
(\ref{Dg-I-P-2-2}) is also obtained in a similar way by noting that
\begin{equation*}
(A^{\rm ext}_g - kC_\ast I) P_\ast^{\rm ext} = 0,
\end{equation*}
as observed in the proof of Proposition \ref{prop-ev-special-1}.
\end{proof}
This proposition indicates that {\em there is no difference to determine eigenstructure on the tangent space $T_{(t_\ast, {\bf x}_\ast)}\mathcal{E}$ between $k>1$ and $k=1$}, in other words, regardless of the presence of $A_g^{\rm res}$.
One consequence is the following characterization of \lq\lq tangential" eigenvectors stated below.
The proof is skipped because it is the natural extension to the autonomous case (\cite{asym2}, Theorem 3.17).

\begin{prop}[cf. \cite{asym2}, Theorem 3.17]
\label{prop-evec-Dg}
Let $(t_\ast, {\bf x}_\ast)\in \mathcal{E}$ be an equilibrium on the horizon for $g^{\rm ext}$ and suppose that Assumption \ref{ass-f-1} holds.
\begin{enumerate}
\item Assume that $\lambda \in {\rm Spec}(A_g^{\rm ext})$ and let ${\bf w}\in \mathbb{C}^n$ be such that
${\bf w}\in \ker((A_g^{\rm ext} - \lambda I)^{m_\lambda})\setminus \ker((A_g^{\rm ext} - \lambda I)^{m_\lambda -1})$ with $(I-P_\ast^{\rm ext}){\bf w}\not = 0$ for some $m_\lambda \in \mathbb{N}$. 
\begin{itemize}
\item If $\lambda \not = k C_\ast$, then $(I -P_\ast^{\rm ext}){\bf w} \in \ker((Dg^{\rm ext}_\ast - \lambda I)^{m_\lambda})\setminus \ker((Dg^{\rm ext}_\ast - \lambda I)^{m_\lambda -1})$.
\item If $\lambda = k C_\ast$, then either of the following holds:
\begin{itemize}
\item $(I-P_\ast^{\rm ext}){\bf w}\in \ker((Dg^{\rm ext}_\ast - k C_\ast I)^{m_\lambda})\setminus \ker((Dg^{\rm ext}_\ast - k C_\ast I)^{m_\lambda -1})$,
\item $(I - P_\ast^{\rm ext}){\bf w}\in \ker((Dg^{\rm ext}_\ast - k C_\ast I)^{m_\lambda-1})\setminus \ker((Dg^{\rm ext}_\ast - k C_\ast I)^{m_\lambda -2})$.
\end{itemize}
\end{itemize}
\item 
Conversely, assume that $\lambda_g \in {\rm Spec}(Dg^{\rm ext}_\ast )$ and let ${\bf w}_g\in \mathbb{C}^n$ be such that
$(I - P_\ast^{\rm ext}){\bf w}_g\in \ker((Dg^{\rm ext}_\ast - \lambda_g I)^{m_{\lambda_g}})\setminus \ker((Dg^{\rm ext}_\ast - \lambda_g I)^{m_{\lambda_g} -1})$ with $(I-P_\ast^{\rm ext}){\bf w}_g\not = 0$ for some $m_{\lambda_g} \in \mathbb{N}$. 
\begin{itemize}
\item If $\lambda_g \not = kC_\ast$, then $(A_g^{\rm ext} - kC_\ast I){\bf w}_g \in \ker((A_g^{\rm ext} - \lambda_g I)^{m_{\lambda_g}})\setminus \ker((A_g^{\rm ext} - \lambda_g I)^{m_{\lambda_g} -1})$.
\item If $\lambda_g = kC_\ast$, then either of the following holds:
\begin{itemize}
\item $(A_g^{\rm ext} - kC_\ast I){\bf w}_g\in \ker((A_g^{\rm ext} - kC_\ast I)^{m_{\lambda_g}})\setminus \ker((A_g^{\rm ext} - kC_\ast I)^{m_{\lambda_g} -1})$,
\item $(A_g^{\rm ext} - kC_\ast I){\bf w}_g\in \ker((A_g^{\rm ext} - kC_\ast I)^{m_{\lambda_g}+1})\setminus \ker((A_g^{\rm ext} - kC_\ast I)^{m_{\lambda_g}})$.
\end{itemize}
\end{itemize}
\end{enumerate}
\end{prop}

Next, consider the correspondence of eigenstructures between $A^{\rm ext}_g$ and $A^{\rm ext}$.
In particular, we have the following correspondence.

\begin{prop}[cf. \cite{asym2}, Proposition 3.19]
\label{prop-correspondence-ev-Ag-A-2}
Let $(t_\ast, {\bf x}_\ast)^T\in \mathcal{E}$ be an equilibrium on the horizon for $g^{\rm ext}$.
Also, let $\lambda \in {\rm Spec}(A^{\rm ext}_g)$ and ${\bf u}\in \ker((A_g^{\rm ext} - \lambda I)^N)\setminus \ker((A_g^{\rm ext} - \lambda I)^{N-1})$ for some $N\in \mathbb{Z}_{\geq 1}$, where ${\bf u}\in \mathbb{R}^{n+1}$ is linearly independent of ${\bf v}^{\rm ext}_{\ast, \alpha}$.
If
\begin{equation*}
\tilde \lambda:= r_{{\bf x}_\ast}^k \lambda,\quad {\bf U} := r_{{\bf x}_\ast}^{\Lambda_\alpha^{\rm ext}}{\bf u},
\end{equation*}
namely
\begin{equation*}
{\bf U} = (U_0, U_1,\ldots, U_n)^T,\quad U_i := r_{{\bf x}_\ast}^{\alpha_i}u_i,
\end{equation*}
then $\tilde \lambda \in {\rm Spec}(A^{\rm ext})$ and ${\bf U} \in \ker((A^{\rm ext} - \tilde \lambda I)^N)\setminus \ker((A^{\rm ext} - \tilde \lambda I)^{N-1})$.
Conversely, if $\tilde \lambda \in {\rm Spec}(A^{\rm ext})$ and ${\bf U}\in \ker((A^{\rm ext} - \lambda I)^N)\setminus \ker((A^{\rm ext} - \lambda I)^{N-1})$ for some $N\in \mathbb{Z}_{\geq 1}$,
then the pair $\{\lambda, {\bf u}\}$ defined by
\begin{equation*}
\lambda:= r_{{\bf Y}_0}^{-k} \tilde \lambda,\quad {\bf u} := r_{{\bf Y}_0}^{- \Lambda_\alpha^{\rm ext}}{\bf U}
\end{equation*}
satisfy $\lambda \in {\rm Spec}(A_g^{\rm ext})$ and ${\bf u} \in \ker((A_g^{\rm ext} - \lambda I)^N)\setminus \ker((A_g^{\rm ext} - \lambda I)^{N-1})$.
\end{prop}

The proof is essentially the same as the autonomous case (\cite{asym2}).
\begin{proof}
Assumption \ref{ass-f-1} implies that it is sufficient to consider the case that $f^{\rm ext}(t, {\bf y})$, equivalently $\tilde f^{\rm ext}(t, {\bf x})$, is quasi-homogeneous,
namely $f^{\rm ext}(t, {\bf y}) = f^{\rm ext}_{\alpha, k}(t, {\bf y})$ and $\tilde f^{\rm ext}(t, {\bf x}) = \tilde f^{\rm ext}_{\alpha, k}(t, {\bf x})$.
In this framework, the $t$-component of $f$ is identically $0$ because it does not contain the quasi-homogeneous component:
\begin{equation*}
f^{\rm ext}(t, {\bf y}) = (0, f_1(t, {\bf y}), \ldots, f_n(t, {\bf y}))^T = (0, f_{1;\alpha, k}(t, {\bf y}), \ldots, f_{n;\alpha, k}(t, {\bf y}))^T,
\end{equation*}
etc. 
See Remark \ref{rem-QH-0th}.
The above setting is assumed in the following arguments.
Recall that an equilibrium on the horizon $(t_\ast, {\bf x}_\ast)$ with $C_\ast > 0$ and the corresponding root ${\bf Y}_0$ of the balance law satisfy
\begin{align}
\label{identity-balance-equilibrium}
&{\bf x}_\ast = r_{{\bf Y}_0}^{-\Lambda_\alpha} {\bf Y}_0,\quad {\bf Y}_0 = r_{{\bf x}_\ast}^{\Lambda_\alpha}{\bf x}_\ast,\\
\notag
&r_{{\bf Y}_0} = p_\alpha({\bf Y}_0) = r_{{\bf x}_\ast} \equiv (kC_\ast)^{-1/k} > 0.
\end{align}
Similar to arguments in Lemma \ref{lem-identity-QHvf}, we have
\begin{equation}
\label{identity-QH-diff}
s^{\alpha_l}\frac{\partial f_i}{\partial x_l}( t, s^{\Lambda_\alpha}{\bf x} ) = s^{k + \alpha_i}\frac{\partial f_i}{\partial x_l}( t, {\bf x} ), \quad l=0,1,\ldots, n
\end{equation}
with the identification $x_0 = t$, while the left-hand side coincides with
\begin{equation*}
s^{\alpha_l}\frac{\partial f_i}{\partial x_l}( t, s^{\Lambda_\alpha}{\bf x} ) = \frac{\partial f_i}{\partial (s^{\alpha_l} x_l)}( t, s^{\Lambda_\alpha}{\bf x} )\frac{\partial (s^{\alpha_l} x_l)}{\partial x_l} \equiv \frac{\partial f_i}{\partial X_l}( {\bf X} )\frac{\partial (s^{\alpha_l} x_l)}{\partial x_l},
\end{equation*}
introducing an auxiliary variable ${\bf X} = (X_0, X_1, \ldots, X_n)^T$, $X_0 := t$, $X_i := s^{\alpha_i} x_i$ for $i=1,\ldots$, with some $s > 0$.
Remark that the above identity still holds with $l=0$ because $\alpha_l = 0$.
Let $D_{\bf X}$ be the derivative with respect to the vector variable ${\bf X}$.
Note that $D_{\bf X} \tilde f^{\rm ext} ({\bf X})|_{{\bf X} = (t, \bar {\bf x})} = D_{(t,{\bf x})} \tilde f^{\rm ext} (t, \bar {\bf x})$ when the variable ${\bf X}$ is set as $(t, {\bf x})$ and that $D_{\bf X} f({\bf X})|_{{\bf X} = (t, \bar {\bf Y})} = D_{\bf Y} f(t, \bar {\bf Y})$ when the variable ${\bf X}$ is set as $(t, {\bf Y})$.
Using the fact that $f^{\rm ext} (t, {\bf Y})$ and $\tilde f^{\rm ext} (t, {\bf x})$ have the identical form (as far as the quasi-homogeneous component is considered), we have
\begin{align*}
D_{(t,{\bf Y})} f^{\rm ext} (t_\ast, {\bf Y}_0) r_{{\bf x}_\ast}^{\Lambda_\alpha^{\rm ext}} = r_{{\bf x}_\ast}^{kI + \Lambda_\alpha^{\rm ext}} D_{(t,{\bf x})} f(t_\ast, {\bf x}_\ast)
\end{align*}
with $s = r_{{\bf x}_\ast}$ and $(t, {\bf x}) = (t_\ast, {\bf x}_\ast)$ in (\ref{identity-QH-diff}) and the identity (\ref{identity-balance-equilibrium}).
That is,
\begin{equation}
\label{identity-QH-diff-2-matrix}
D_{(t,{\bf Y})} f^{\rm ext} (t_\ast, {\bf Y}_0) = r_{{\bf x}_\ast}^{kI + \Lambda_{\alpha}^{\rm ext}}  D_{(t,{\bf x})} \tilde f^{\rm ext}(t_\ast, {\bf x}_\ast) r_{{\bf x}_\ast}^{-\Lambda_{\alpha}^{\rm ext}}.
\end{equation}
Then we have
\begin{align*}
A^{\rm ext} &= -\frac{1}{k} \Lambda_\alpha^{\rm ext} + D_{(t,{\bf Y})} f^{\rm ext}(t_\ast, {\bf Y}_0)\quad \text{(from (\ref{blow-up-power-determining-matrix}))}\\
	&= -r_{{\bf x}_\ast}^k C_\ast \Lambda_\alpha^{\rm ext} + D_{(t,{\bf Y})} f^{\rm ext}(t_\ast, {\bf Y}_0)\quad \text{(from (\ref{identity-balance-equilibrium}))}\\
	&= -r_{{\bf x}_\ast}^k C_\ast \Lambda_\alpha^{\rm ext} + r_{{\bf x}_\ast}^{kI + \Lambda_{\alpha}^{\rm ext}} D_{(t,{\bf x})} \tilde f^{\rm ext}_\ast r_{{\bf x}_\ast}^{-\Lambda_{\alpha}^{\rm ext}}\quad \text{(from (\ref{identity-QH-diff-2-matrix}))}\\
	&= r_{{\bf x}_\ast}^{kI + \Lambda_{\alpha}^{\rm ext}} \left( -C_\ast \Lambda_\alpha^{\rm ext} +  D_{(t,{\bf x})} \tilde f^{\rm ext}_\ast \right)r_{{\bf x}_\ast}^{-\Lambda_{\alpha}^{\rm ext}}\\
	&= r_{{\bf x}_\ast}^{kI + \Lambda_{\alpha}^{\rm ext}} A_g^{\rm ext} r_{{\bf x}_\ast}^{-\Lambda_{\alpha}^{\rm ext}} \quad \text{(from (\ref{Ag-QH}))}
\end{align*}
and hence
\begin{equation}
\label{conj-A-Ag}
A^{\rm ext} = r_{{\bf x}_\ast}^{kI+\Lambda_{\alpha}^{\rm ext}} A_g^{\rm ext} r_{{\bf x}_\ast}^{-\Lambda_{\alpha}^{\rm ext}} \quad \Leftrightarrow \quad 
A_g^{\rm ext} = r_{{\bf Y}_0}^{-(kI + \Lambda_{\alpha}^{\rm ext})}  A^{\rm ext} r_{{\bf Y}_0}^{\Lambda_{\alpha}^{\rm ext}},
\end{equation}
where we have used $r_{{\bf Y}_0} = r_{{\bf x}_\ast}$.
In particular, for any $\lambda\in \mathbb{C}$ and $N\in \mathbb{N}$ with the identity $\tilde \lambda =  r_{{\bf x}_\ast}^k \lambda$, we have
\begin{align}
\notag
&(A^{\rm ext} - \tilde \lambda I)^N = r_{{\bf x}_\ast}^{kN I + \Lambda_{\alpha}^{\rm ext}} (A_g^{\rm ext} - \lambda I)^N r_{{\bf x}_\ast}^{-\Lambda_{\alpha}^{\rm ext}}\\
\label{conj-A-Ag-2}
&\quad \Leftrightarrow \quad
(A_g^{\rm ext} - \lambda I)^N = r_{{\bf Y}_0}^{-(kN I +\Lambda_{\alpha}^{\rm ext})} (A^{\rm ext} - \tilde \lambda I)^N r_{{\bf Y}_0}^{\Lambda_{\alpha}^{\rm ext}}.
\end{align}
This identity directly yields our statements.
For example, let ${\bf u} = (u_0, u_1,\ldots, u_n)$ be an eigenvector of $A_g^{\rm ext}$ associated with an eigenvalue $\lambda$:  
$A_g^{\rm ext} {\bf u} = \lambda {\bf u}$.
Then (\ref{conj-A-Ag}) yields
\begin{align*}
 \lambda {\bf u} &= A_g^{\rm ext} {\bf u}
 	= r_{{\bf Y}_0}^{-(k I + \Lambda_{\alpha}^{\rm ext})} A^{\rm ext} r_{{\bf Y}_0}^{\Lambda_{\alpha}^{\rm ext}} {\bf u}
 	= r_{{\bf Y}_0}^{-(k I + \Lambda_{\alpha}^{\rm ext})} A^{\rm ext} {\bf U},
\end{align*}
and hence
\begin{equation*}
A^{\rm ext}{\bf U} = r_{{\bf Y}_0}^{-k}\lambda {\bf U} = \tilde \lambda {\bf U}.
\end{equation*}
Repeating the same argument conversely assuming the eigenstructure $A^{\rm ext}{\bf U} = \tilde \lambda {\bf u}$, we know that an eigenpair $(\lambda, {\bf u})$ of $A_g^{\rm ext}$ is constructed from a given eigenpair $(\tilde \lambda, {\bf U})$ of $A$ through (\ref{identity-balance-equilibrium}) and (\ref{conj-A-Ag}).
Correspondence of generalized eigenvectors follows from the similar arguments through (\ref{conj-A-Ag-2}).
\end{proof}

Finally we discuss the additional eigenpair in the nonautonomous setting, reflecting the (linearized) time-evolution of $g^{\rm ext}$ in the $t$-direction.
\begin{prop}
\label{prop-tangential-t}
Under Assumption \ref{ass-f-1}, the matrices $A^{\rm ext}$ and $Dg^{\rm ext}_\ast$ admits the eigenpair
\begin{equation*}
\left\{0, \begin{pmatrix}
 1 \\  -A^{-1}D_t f_{\alpha, k}(t_\ast, {\bf Y}_0)
\end{pmatrix}\right\}\quad \text{ and }\quad \left\{0, \begin{pmatrix}
1 \\ -((D_{\bf x}g)_\ast)^{-1}(D_t g)_\ast
\end{pmatrix}\right\},
\end{equation*}
respectively.
\end{prop}

\begin{proof}
First discuss the corresponding eigenvector of $A^{\rm ext}$.
If we pay attention to an eigenvector with {\em non-trivial $t$-component}, eigenvalue $0$ is appropriate\footnote{
It follows from Proposition \ref{prop-ev1} that this choice is the only suitable one.
Indeed, the remaining eigenpairs have nonzero eigenvalues and eigenvectors whose $t$-components are $0$.
}, which is compatible with Lemma \ref{lem-eigen-extend-1} with $a=0$.
In this case, the eigen-equation
\begin{equation*}
\begin{pmatrix}
0 & {\bf 0}_n^T \\
D_t f_{\alpha, k}(t_\ast, {\bf Y}_0) & A
\end{pmatrix} \begin{pmatrix}
1 \\ \tilde {\bf v}
\end{pmatrix} = {\bf 0}_{n+1}
\end{equation*}
yields $\tilde {\bf v} = -A^{-1}D_t f_{\alpha, k}(t_\ast, {\bf Y}_0)$.
\par
Next, because we assumed that $(D_{\bf x}g)_\ast$ is nonsingular, the Implicit Function Theorem provides a smooth $t$-parameter family $\{t, {\bf x}(t)\}$ of equilibria defined in a small neighborhood $I\subset \mathbb{R}$ of $t_\ast$. 
Thanks to Theorem \ref{thm-balance-1to1}, such a family is located on the horizon. 
Differentiating the identity $g(t, {\bf x}(t)) \equiv 0$ in $t$ at $t_\ast$, we know that the vector
\begin{equation*}
\tilde {\bf w} := -((D_{\bf x}g)_\ast)^{-1}(D_t g)_\ast
\end{equation*}
satisfies $(D_tg)_\ast + (D_{\bf x}g)_\ast \tilde {\bf w} = 0$ and is located on $T_{{\bf x}_\ast}\mathcal{E}_{\bf x}$.
Using Lemma \ref{lem-eigen-extend-1} for $A_g^{\rm ext}$ and Proposition \ref{prop-corr-Dg-Ag}, we have the statement for $Dg^{\rm ext}_\ast$.
\end{proof}

\subsection{Eigenstructure in \lq\lq transversal" direction; case 2: $k=1$}
\label{section-special-k=1}

Finally consider the transversal eigenpair of $Dg^{\rm ext}_\ast $ with $k=1$, in which case $\tilde f_0(t,{\bf x}) = \kappa_\alpha^{-1} = 1-p({\bf x})^{2c}$ and hence $Dg^{\rm ext}_\ast $ has the following form:
\begin{equation*}
Dg^{\rm ext}_\ast = A_g^{\rm ext} + A_g^{\rm res} + B_g^{\rm ext}.
\end{equation*}

Because remaining eigenvectors belong to the tangent space $T_{(t_\ast, {\bf x}_\ast)}\mathcal{E}$, the eigenvector in the present interest must be transversal to $T_{(t_\ast, {\bf x}_\ast)}\mathcal{E}$.
In particular, the \lq\lq transversal" eigenvector is written as
\begin{equation*}
\tilde {\bf v}^{\rm ext}_{\ast,\alpha}  = \begin{pmatrix}
u_0 \\ {\bf v}_{\ast, \alpha} + {\bf w}
\end{pmatrix},\quad u_0\in \mathbb{R},\quad {\bf w} \in T_{{\bf x}_\ast}\mathcal{E}_{\bf x}.
\end{equation*}
Letting $\lambda$ be the corresponding eigenvalue, we require
\begin{align*}
\lambda \begin{pmatrix}
u_0 \\ {\bf v}_{\ast, \alpha} + {\bf w}
\end{pmatrix} &= Dg^{\rm ext}_\ast \begin{pmatrix}
u_0 \\ {\bf v}_{\ast, \alpha} + {\bf w}
\end{pmatrix}\\
	&= \begin{pmatrix}
0 & -2c (D_{\bf x} p_\alpha)_\ast^T \\
(D_t g)_\ast & (D_{{\bf x}} g)_\ast
\end{pmatrix} \begin{pmatrix}
u_0 \\ {\bf v}_{\ast, \alpha} + {\bf w}
\end{pmatrix}\\
	&= \begin{pmatrix}
-2c \\ u_0(D_t g)_\ast + (D_{{\bf x}} g)_\ast ({\bf v}_{\ast, \alpha} + {\bf w})
\end{pmatrix}\\
	&= \begin{pmatrix}
-2c \\ u_0(D_t g)_\ast - C_\ast {\bf v}_{\ast, \alpha} + (D_{{\bf x}} g)_\ast{\bf w}
\end{pmatrix},
\end{align*}
where we have used $(D_{\bf x} p_\alpha)_\ast^T {\bf v}_{\ast, \alpha} = 1$ and $(D_{\bf x} p_\alpha)_\ast^T {\bf w} = 0$.
From the first component, we have the identity $\lambda u_0 = -2c$, namely $u_0 = -2c/\lambda$ (provided $\lambda \not = 0$).
Next, using the fact ${\bf w} \in T_{{\bf x}_\ast}\mathcal{E}_{\bf x}$, we further have 
\begin{equation*}
((D_{\bf x}g)_\ast - \lambda I_n)(I - P_\ast){\bf w} = (I - P_\ast) (A_g - \lambda I_n){\bf w}
\end{equation*}
proved in \cite{asym2}. 
See also Proposition \ref{prop-corr-Dg-Ag}.
We decompose the vector $(D_t g)_\ast$ as $(D_t g)_\ast= a_1 {\bf v}_{\ast, \alpha} + (I_n - P_\ast) (D_t g)_\ast$, in particular $a_1 = (D_{\bf x} p_\alpha)_\ast^T (D_t g)_\ast$.
Arranging the second identity, we have
\begin{equation}
\label{identity-transversal}
( \lambda + C_\ast) {\bf v}_{\ast, \alpha} + \frac{2c}{\lambda}\{ (D_{\bf x} p_\alpha)_\ast^T (D_t g)_\ast \} {\bf v}_{\ast, \alpha} 
	= (I_n - P_\ast) \left\{ -\frac{2c}{\lambda} (D_t g)_\ast  + (A_g - \lambda I_n) {\bf w} \right\}.
\end{equation}
Now we know the following result, which simplifies the expressions of roots.
\begin{lem}
\label{lem-orthogonal}
$(D_{\bf x} p_\alpha)_\ast^T (D_t g)_\ast = 0$. 
\end{lem}
\begin{proof}
Recall Remark \ref{rem-invariance} that the horizon $\mathcal{E}$ is $g^{\rm ext}$-invariant. 
In particular, $T_{(t_\ast, {\bf x}_\ast)}\mathcal{E}$ is $Dg^{\rm ext}_\ast$-invariant and its ${\bf x}$-component $T_{{\bf x}_\ast}\mathcal{E}_{\bf x}$ is $(D_{\bf x}g)_\ast$-invariant.
From identities in Proposition \ref{prop-tangential-t}, we have
\begin{equation*}
(D_t g)_\ast + (D_{\bf x} g)_\ast \tilde {\bf w}= {\bf 0}_n
\end{equation*}
for $\tilde {\bf w} \in T_{{\bf x}_\ast}\mathcal{E}_{\bf x}$ constructed in Proposition \ref{prop-tangential-t}. 
Taking the inner product with $(D_{\bf x}p_\alpha)_\ast$, we have
\begin{equation*}
(D_{\bf x}p_\alpha)_\ast^T (D_t g)_\ast + (D_{\bf x}p_\alpha)_\ast^T (D_{\bf x} g)_\ast \tilde {\bf w}= 0.
\end{equation*}
Because $T_{{\bf x}_\ast}\mathcal{E}_{\bf x}$ is $(D_{\bf x}g)_\ast$-invariant, the vector $(D_{\bf x} g)_\ast \tilde {\bf w}$ also belongs to $T_{{\bf x}_\ast}\mathcal{E}_{\bf x}$. 
In particular, 
\begin{equation*}
(D_{\bf x}p_\alpha)_\ast^T (D_{\bf x} g)_\ast \tilde {\bf w} = 0
\end{equation*}
and hence our claim holds.
\end{proof}
Because two vector spaces ${\rm span}\{{\bf v}_{\ast, \alpha} \}$ and $T_{{\bf x}_\ast}\mathcal{E}_{\bf x}$ are transversal in $\mathbb{R}^n$, the both sides in (\ref{identity-transversal}) must be ${\bf 0}_n$.
In particular, we have $\lambda ( \lambda + C_\ast)  = 0$ whose roots are $\lambda = 0, -C_\ast$.
We shall choose $\lambda = -C_\ast$ because $\lambda = 0$ induces the eigenvector which is exactly the same as the one in Proposition \ref{prop-tangential-t}.
Using this, (\ref{identity-transversal}) with $\lambda = -C_\ast$ is written as
\begin{align*}
{\bf 0}_n 
	&= \frac{2c}{C_\ast} (D_t g)_\ast + (A_g + C_\ast I_n) {\bf w}  - \tilde d{\bf v}_{\ast, \alpha}
\end{align*}
for some $\tilde d \in \mathbb{R}$, where we have used $P_\ast  (D_t g)_\ast = 0$ from Lemma \ref{lem-orthogonal}. 
In particular,
\begin{equation*}
{\bf w} = (A_g + C_\ast I_n)^{INV}\left\{ -\frac{2c}{C_\ast} (D_t g)_\ast + \tilde d{\bf v}_{\ast, \alpha}\right\},
\end{equation*}
where $(A_g + C_\ast I_n )^{INV}$ denotes the inverse\footnote{
If $- C_\ast\not\in {\rm Spec}(A_g )$, the matrix $A_g + C_\ast I_n$ is invertible and hence $(A_g + C_\ast I_n )^{INV}$ is just $(A_g + C_\ast I_n )^{-1}$ in the usual sense.
See e.g., Chapter 5.4 of \cite{K2013} for this description.
} of $A_g + C_\ast I_n$ restricted to the invariant subspace spanned by eigenvectors associated with ${\rm Spec}(A_g) \setminus \{-C_\ast\}$.
Because ${\bf w}$ originally belongs to $T_{{\bf x}_\ast} \mathcal{E}_{\bf x}$, we have
\begin{equation*}
(D_{\bf x}p_\alpha)_\ast^T (A_g + C_\ast I_n)^{INV}\left\{ -\frac{2c}{C_\ast} (D_t g)_\ast + \tilde d{\bf v}_{\ast, \alpha}\right\} = 0,
\end{equation*}
equivalently
\begin{equation*}
\tilde d = \frac{2c}{C_\ast} \frac{ (D_{\bf x}p_\alpha)_\ast^T (A_g + C_\ast I_n)^{INV} (D_t g)_\ast}{ (D_{\bf x}p_\alpha)_\ast^T (A_g + C_\ast I_n)^{INV} {\bf v}_{\ast, \alpha}}.
\end{equation*}
%
As a summary, we have the following result.

\begin{prop}[Transversal eigenpair, $k=1$]
\label{prop-ev-special-2}
Suppose that Assumption \ref{ass-f-1} holds and $k=1$.
Also suppose that $(t_\ast, {\bf x}_\ast) \in \mathcal{E}$ is an equilibrium on the horizon for the associated desingularized vector field $g^{\rm ext}$ in (\ref{desing-para-nonaut}).
Then the Jacobian matrix $Dg^{\rm ext}_\ast $ always possesses the eigenpair $\{-C_\ast, \tilde {\bf v}^{\rm ext}_{\ast,\alpha}\}$, where
\begin{equation}
\label{evec-transversal-desing}
\tilde {\bf v}^{\rm ext}_{\ast,\alpha} 
= \begin{pmatrix}
1 \\  \left\{ \frac{C_\ast}{2c} + d (A_g + C_\ast I_n )^{INV} \right\} {\bf v}_{\ast, \alpha} - (A_g + C_\ast I_n )^{INV} (D_t g)_\ast
\end{pmatrix},
\end{equation}
where the constant $d$ is given by 
\begin{equation}
\label{const_d}
d = \frac{ (D_{\bf x}p_\alpha)_\ast^T (A_g + C_\ast I_n)^{INV} (D_t g)_\ast}{ (D_{\bf x}p_\alpha)_\ast^T (A_g + C_\ast I_n)^{INV} {\bf v}_{\ast, \alpha}},
\end{equation}
whenever the denominator of $d$ does not vanish.
The expression (\ref{evec-transversal-desing}) itself still makes sense when $d=0$.
\end{prop}

\subsection{Correspondence of eigenstructures and simplified blow-up criterion}
\label{section-parameter-dep-Y}

Our results here determine the complete correspondence of eigenpairs among $A^{\rm ext}$ and $Dg_\ast^{\rm ext}$.
In particular, blow-up power eigenvalues are completely determined by ${\rm Spec}(Dg^{\rm ext}_\ast )$, and vice versa.
The complete correspondence of eigenstructures is obtained under a mild assumption of the corresponding matrices.

\begin{table}[ht]\em
\centering
{
\begin{tabular}{cccc}
\hline
  & $A^{\rm ext}$ & $Dg^{\rm ext}_\ast$\\
\hline\\[-2mm]
Transversal eigenvalue  & $1$ & $-C_\ast$\\ [2mm]
Transversal eigenvector ($k>1$) & ${\bf v}_{0,\alpha}^{\rm ext}$ (Proposition \ref{prop-ev1}) & ${\bf v}_{\ast,\alpha}^{\rm ext}$ (Proposition \ref{prop-ev-special-1}) \\  [2mm]
Transversal eigenvector ($k=1$) & ${\bf v}_{0,\alpha}^{\rm ext}$ (Proposition \ref{prop-ev1}) & $\tilde {\bf v}_{\ast,\alpha}^{\rm ext}$ (Proposition \ref{prop-ev-special-2}) \\  [2mm]
Tangential  eigenvalue (nonautonomous) & $0$ & $0$\\ [2mm]
Tangential  eigenvector (nonautonomous) & $\begin{pmatrix}
1 \\  -A^{-1}D_t f_{\alpha, k}(t_\ast, {\bf Y}_0)
\end{pmatrix}$ & $r_{{\bf Y}_0}^{-\Lambda_\alpha^{\rm ext}}\begin{pmatrix}
1 \\  -A^{-1}D_t f_{\alpha, k}(t_\ast, {\bf Y}_0)
\end{pmatrix}$ \\  [2mm]
Tangential  eigenvalue & $\tilde \lambda$ & $\lambda = r_{{\bf Y}_0}^{-k}\tilde \lambda$ \\  [2mm]
Tangential  (generalized) eigenvector & $\begin{pmatrix}0 \\ \tilde {\bf U} \end{pmatrix}$ & $(I-P_\ast) r_{{\bf Y}_0}^{-\Lambda_\alpha^{\rm ext}} \begin{pmatrix}0 \\ \tilde {\bf U} \end{pmatrix}$ \\  [2mm]
\hline 
\end{tabular}%
}
\caption{Correspondence of eigenstructures from $A^{\rm ext}$ to $Dg^{\rm ext}_\ast$}
\flushleft
The constant $r_{{\bf Y}_0}$ is $p_\alpha({\bf Y}_0)$. 
Once a nonzero root ${\bf Y}_0$ of the balance law and eigenpairs of $A$ are given, corresponding equilibrium on the horizon $(t_\ast, {\bf x}_\ast)$ and all eigenpairs of $Dg^{\rm ext}_\ast$ are constructed by the rule on the table.
\label{table-eigen0}
\end{table}%

\begin{table}[ht]\em
\centering
{
\begin{tabular}{cccc}
\hline
  & $Dg^{\rm ext}_\ast$  & $A^{\rm ext}$\\
\hline\\[-2mm]
Transversal eigenvalue  & $-C_\ast$ & $1$\\ [2mm]
Transversal eigenvector ($k>1$) & ${\bf v}_{\ast,\alpha}^{\rm ext}$ (Proposition \ref{prop-ev-special-1}) & ${\bf v}_{0,\alpha}^{\rm ext}$ (Proposition \ref{prop-ev1}) \\  [2mm]
Transversal eigenvector ($k=1$) & $\tilde {\bf v}_{\ast,\alpha}^{\rm ext}$ (Proposition \ref{prop-ev-special-2}) & ${\bf v}_{0,\alpha}^{\rm ext}$ (Proposition \ref{prop-ev1}) \\  [2mm]
Tangential  eigenvalue (nonautonomous) & $0$ & $0$\\ [2mm]
Tangential  eigenvector (nonautonomous) & $\begin{pmatrix}
1 \\  - ((D_{\bf x} g)_\ast)^{-1}(D_t g)_\ast
\end{pmatrix}$ & $r_{{\bf x}_\ast}^{\Lambda_\alpha^{\rm ext}}\begin{pmatrix}
1 \\  - ((D_{\bf x} g)_\ast)^{-1}(D_t g)_\ast
\end{pmatrix}$ \\  [2mm]
Tangential  eigenvalue & $\lambda$ & $\tilde \lambda = r_{{\bf x}_\ast}^k\lambda$ \\  [2mm]
Tangential  (generalized) eigenvector & $(I-P_\ast) \begin{pmatrix}0 \\ \tilde {\bf u} \end{pmatrix}$ & $r_{{\bf x}_\ast}^{\Lambda_\alpha^{\rm ext}} (A_g^{\rm ext} - kC_\ast I) \begin{pmatrix}0 \\ \tilde {\bf u} \end{pmatrix}$ \\  [2mm]
\hline 
\end{tabular}%
}
\caption{Correspondence of eigenstructures from $Dg^{\rm ext}_\ast $ to $A^{\rm ext}$}
\flushleft
The constant $r_{{\bf x}_\ast}$ is $(kC_\ast)^{-1/k}$, which is positive by Theorem \ref{thm-balance-1to1}.
Once an equilibrium on the horizon $(t_\ast, {\bf x}_\ast)$ and eigenpairs of $Dg^{\rm ext}_\ast$ are given, corresponding (nonzero) root of the balance law $(t_\ast, {\bf Y}_0)$ and all eigenpairs of $A^{\rm ext}$ are constructed by the rule on the table.
\label{table-eigen2}
\end{table}%

\begin{thm}
\label{thm-blow-up-estr}
Let $(t_\ast, {\bf x}_\ast)\in \mathcal{E}$ be an equilibrium on the horizon for $g$ which is mapped to a root $(t_\ast, {\bf Y}_0)$ with nonzero ${\bf Y}_0$ of the balance law (\ref{0-balance}) through (\ref{x-to-C}) and (\ref{C-to-x}), and suppose that Assumption \ref{ass-f-1} holds.
When all the eigenpairs of the blow-up power-determining matrix $A^{\rm ext}$ associated with $(t_\ast, {\bf Y}_0)$ are determined, then all the eigenpairs of $Dg^{\rm ext}_\ast$ are constructed through the correspondence listed in Table \ref{table-eigen0}.
Similarly, if all the eigenpairs of $Dg^{\rm ext}_\ast$ are determined, then all the eigenpairs of $A^{\rm ext}$ are constructed through the correspondence listed in Table \ref{table-eigen2}.
\par
Moreover, the Jordan structure associated with eigenvalues, namely the number of Jordan blocks and their size, are identical except $k C_\ast \in {\rm Spec}(A_g^{\rm ext})$ if exists, and $1\in {\rm Spec}(A^{\rm ext})$.
\end{thm}

\begin{proof}
Correspondences between {\em Transversal eigenvalue} and {\em Transversal eigenvector} in Tables follow from Propositions \ref{prop-ev1} and \ref{prop-ev-special-1}, and \ref{prop-ev-special-2}, while correspondences between {\em Tangential eigenvalue} and {\em Tangential  (generalized) eigenvector} in Tables follow from Propositions \ref{prop-corr-Dg-Ag}, \ref{prop-evec-Dg} and \ref{prop-correspondence-ev-Ag-A-2}.
\par
A special pair {\em Tangential eigenvalue (nonautonomous)} and {\em Tangential eigenvector (nonautonomous)} follows from Lemma \ref{lem-eigen-extend-1}, Propositions \ref{prop-corr-Dg-Ag}, \ref{prop-correspondence-ev-Ag-A-2} and \ref{prop-tangential-t}.
\par
If $kC_\ast \not \in {\rm Spec}(Dg^{\rm ext}_\ast )$ (in particular, $1\in {\rm Spec}(A^{\rm ext})$ is simple from Propositions \ref{prop-ev1} and  \ref{prop-correspondence-ev-Ag-A-2}), the number and size of Jordan blocks are identical by Proposition \ref{prop-correspondence-ev-Ag-A-2}.
\end{proof}

Similar to autonomous cases, we have a simple criterion of the existence and characterization of type-I blow-up solutions through the complete correspondence of eigenstructures of associated systems.

\begin{thm}[Criterion of existence of blow-ups]
\label{thm-existence-blow-up}
Let $(t_\ast, {\bf Y}_{0,\ast})$ be a root of the balance law (\ref{0-balance}) with ${\bf Y}_{0,\ast}\not = {\bf 0}_n$.
Assume that ${\rm Spec}(A)\cap i\mathbb{R} = \emptyset$, where $A$ is defined as (\ref{Aext-sub}) for the corresponding blow-up power determining matrix $A^{\rm ext}$ associated with ${\bf Y}_{0,\ast}$.
Then (\ref{ODE-original}) admits a blow-up solution ${\bf y}(t)$ blowing up as $t \to t_\ast - 0$.
Moreover, if $f$ is $C^4$ and ${\rm Spec}(A^{\rm ext})$ satisfies the non-resonance condition, the corresponding blow-up solution admits the asymptotic behavior $y_i(t) \sim Y_{0,i}\theta(t)^{-\alpha_i/k}$ as $t\to t_{\max} < \infty$, provided $Y_{0,i} \not = 0$.
\end{thm}

\begin{proof}
Eigenvalues ${\rm Spec}(A^{\rm ext})$ of $A^{\rm ext}$ consist of $1$ and remaining $n$ eigenvalues, all of which except the simple eigenvalue $0$ have nonzero real parts by the original form of $A^{\rm ext}$ in (\ref{Aext-sub}) and our assumption.
Therefore the Implicit Function Theorem ensures the existence of a smooth $t$-parameter family $\{(t, {\bf Y}_0(t)) \in t\in I_\ast\}$ of roots of the balance law with an open interval $I_\ast \subset \mathbb{R}$ including $t_\ast$ such that ${\bf Y}_0(t_\ast) = {\bf Y}_{0,\ast}$.
Taking $I_\ast$ smaller if necessary, we may conclude that the matrix $A^{\rm ext} = A^{\rm ext}(t)$ with $t\in I_\ast$ has the following structure such that $A^{\rm ext}(t_\ast)$ originally possesses:
\begin{itemize}
\item $0$ is a simple eigenvalue.
\item The corresponding submatrix $A(t)$ is hyperbolic.
\end{itemize}
Then, through Theorems \ref{thm-balance-1to1} and \ref{thm-blow-up-estr}, roots of the balance law $\{(t, {\bf Y}_0(t)) \in I_\ast\}$ corresponds to the $t$-parameter family of equilibria on the horizon $M_\ast = \{ {\bf p}_\ast(t) = (t, {\bf x}_\ast(t)) \mid t\in I_\ast\} \subset \mathcal{E}$ for $g^{\rm ext}$, and the associated Jacobian matrix $Dg^{\rm ext}(t, {\bf x}_\ast(t))$ has the following structure for all $t\in I_\ast$:
\begin{itemize}
\item $0$ is a simple eigenvalue and the associated eigenvector is tangent to the horizon $\mathcal{E}$.
\item The corresponding submatrix $D_{\bf x}g(t, {\bf x}_\ast(t))$ is hyperbolic.
\end{itemize}
Therefore the family $M_\ast$ constructs a $1$-dimensional NHIM
Moreover $W^s_{\rm loc}(M_\ast; g^{\rm ext})\cap \mathcal{D} \not = \emptyset$ because $-C_\ast < 0$, and the associated eigenvector ${\bf v}^{\rm ext}_{\ast, \alpha}$, determines the distribution of $W^s_{\rm loc}(M_\ast; g^{\rm ext})$ transversal to $\mathcal{E}$.
Then Theorem \ref{thm:blowup} provides the corresponding blow-up solution.
\par
Under the non-resonance condition, the asymptotic behavior $y_i(t) = O(\theta(t)^{-\alpha_i/k})$ is also provided, as long as $x_{\ast, i}\not = 0$.
Therefore, the bijection
\begin{equation*}
\frac{x_i(t)}{(1-p_\alpha({\bf x}(t))^{2c})^{\alpha_i}} = \theta(t)^{-\alpha_i/k}Y_{0,i},\quad i=1,\cdots, n
\end{equation*}
provides the concrete form of the blow-up solution ${\bf y}(t)$ whenever $Y_{0,i} \not = 0$.
%
\end{proof}

We therefore conclude that, like the autonomous case \cite{asym2}, {\em asymptotic expansions of blow-up solutions themselves provide a criterion of the existence of blow-up solutions}.
On the other hand, blow-up power eigenvalues do {\em not} extract exact dynamical properties around the corresponding blow-up solutions, as shown below.
The proof is the consequence of the correspondence of eigenstructures; Theorem \ref{thm-blow-up-estr}, as well as manifolds constructed in Theorem \ref{thm-existence-blow-up}.

\begin{thm}[Stability gap, cf. \cite{asym2}]
\label{thm-stability}
Let $f$ be an asymptotically quasi-homogeneous vector field of type $\alpha = (0,\alpha_1, \ldots, \alpha_n)$ and order $k+1$ satisfying Assumption \ref{ass-f-1}.
Let ${\bf p}_\ast = (t_\ast, {\bf x}_\ast)$ be an equilibrium on the horizon for the desingularized vector field $g^{\rm ext}$ associated with $f$ such that $W^s_{\rm loc}({\bf p}_\ast; g^{\rm ext})\cap \mathcal{D}\not = \emptyset$, and ${\bf Y}_0$ be the corresponding root of the balance law which is not identically zero\footnote{
Theorem \ref{thm-existence-blow-up} and the construction of an inflowing invariant manifold $M_I$ admitting normally hyperbolic property as well as arguments in \cite{Mat2025_NHIM} indicate that $W^s_{\rm loc}({\bf p}_\ast; g^{\rm ext})$ is well-defined via the local invariant foliation $\mathcal{F}^s$ of $W^s_{\rm loc}(M_I; g^{\rm ext})$.
}.
If
\begin{align*}
m &:= \dim W^s_{\rm loc}({\bf p}_\ast; g^{\rm ext}),\quad m_{A} := \sharp \{\lambda\in {\rm Spec}(A) \mid {\rm Re}\lambda < 0\},
\end{align*}
then we have $m = m_{A}+1$.
Here the matrix $A$ is given in (\ref{Aext-sub}).
\end{thm}


\section{Examples}
\label{section-examples}

Examples demonstrating simple criteria for characterizing (type-I) blow-up solutions for nonautonomous systems are collected here.

\subsection{The first Painlev\'{e} equation}

First we consider the first Painlev\'{e} equation 
\begin{equation}
\label{Pain1-original}
u'' = 6u^2 + t,\quad ' = \frac{d}{dt}
\end{equation}
from the viewpoint of blow-up description.
It is well-known that (\ref{Pain1-original}) possesses the following solution: for fixed $t_0\in \mathbb{R}$, 
\begin{equation}
\label{blow-up-Pain1}
u(t) \sim (t_{\max} - t)^{-2}\quad \text{ as }\quad t \to t_{\max}-0
\end{equation}
for some $t_{\max} = t_{\max}(t_0) > t_0$ and sufficiently large initial point $u_0 = u(t_0)$.
Although this asymptotic behavior can be derived via substitution of a formal (Frobenius-type) power series solution into the equation (e.g., \cite{KJH1997}), we shall derive the existence of such a solution through dynamics at infinity.
\par
Rewrite (\ref{Pain1-original}) as the system of the first order ODEs:
\begin{equation}
\label{Pain1}
\begin{cases}
\chi' = 1, & \\
u' = v, & \\
v' = 6u^2 + \chi, &
\end{cases}\quad ' = \frac{d}{dt}.
\end{equation}
We immediately have the following property.
\begin{lem}
The system (\ref{Pain1}) is asymptotically quasi-homogeneous of type $(0,2,3)$ and order $k+1 = 2$.
\end{lem}
A blow-up description of this system is discussed in \cite{Mat2025_NHIM}, which required lengthy calculations.
We revisit this system here and unravel the correspondence of two descriptions of blow-ups we have derived.
First, the balance law is
\begin{equation}
\label{Pain1-balance}
2U = V,\quad 3V = 6U^2
\end{equation}
for $(\chi, U, V)$ whose solutions, if exist, determine coefficients of (type-I) blow-ups.
Solutions are
\begin{equation*}
(\chi, U, V) = (\chi, 0,0),\quad (\chi, 1, 2),\quad \chi\in \mathbb{R}.
\end{equation*}
Our interest here is the nontrivial root $(\chi, 1, 2)$, which corresponds to the family
\begin{equation*}
M = \left\{ \left. (\chi, {\bf x}_\ast(\chi)) \equiv \left(\chi,\, \frac{1}{17^{1/6}},\,  \frac{2}{17^{1/4}}\right) \right| \chi\in \mathbb{R} \right\}
\end{equation*}
as an invariant manifold on the horizon $\mathcal{E} \equiv \{(\chi, {\bf x}) \mid p_\alpha(\chi,{\bf x})=1\}$ for the desingularized vector field $g$ associated with (\ref{Pain1}).
Obviously, the vector $(1,0,0)^T$ is the tangent vector of $M$ at any points.
The blow-up power determining matrix at $(\chi, U, V) = (\chi, 1, 2)$ is
\begin{equation*}
\begin{pmatrix}
0 & 0 & 0 \\
0 & -2 & 1 \\
0 & 12U & -3
\end{pmatrix}_{(\chi, U, V) = (\chi, 1,2)} = \begin{pmatrix}
0 & 0 & 0 \\
0 & -2 & 1 \\
0 & 12 & -3
\end{pmatrix}
\end{equation*}
whose eigenvalues are $0, 1, -6$.
Because the eigenvalue $0$ is simple, and admits the eigenvector $(1,0,0)^T$ corresponding to the tangent vector of $M$, and hence, for any compact interval $I\subset \mathbb{R}$, the tube $M_I$ admits normally hyperbolic structure for $g$.
We then know the following, as the consequence of Theorem \ref{thm-existence-blow-up}.
\begin{thm}
The system (\ref{Pain1}) admits a type-I blow-up solution with the asymptotic behavior
\begin{equation*}
u(t)\sim (t_{\max} - t)^{-2},\quad v(t)\sim 2(t_{\max} - t)^{-3}\quad \text{ as }\quad t\to t_{\max}-0.
\end{equation*}
\end{thm}
This result is a re-interpretation of blow-up descriptions by means of the correspondence between the balance law and dynamics at infinity.
\begin{rem}
For equilibria on $M$, we directly have $r_{\bf {Y}_0} = 17^{-1/12}$.
On the other hand, it is calculated in \cite{Mat2025_NHIM} that
\begin{equation*}
G(\chi_\ast, {\bf x}_\ast) = \frac{1}{2}x_{\ast, 1}^5x_{\ast, 2} + 2x_{\ast, 1}^2 x_{\ast, 2}^3,
\end{equation*}
which is $17^{-1/12} \equiv C_\ast$ at any points on $M$.
It is also confirmed in \cite{Mat2025_NHIM} that the eigenvalues of the Jacobian matrix $Dg^{\rm ext}$ at $(\chi, {\bf x}_\ast)$
\end{rem}

\subsection{A system associated with self-similarity}

The next example we shall consider is the following system:
\begin{equation}
\label{ss-original}
( u^{m-1}u' )' + \beta \chi u' + \alpha u = 0,\quad ' = \frac{d}{d\chi},\quad \chi \in \mathbb{R},
\end{equation}
where $\alpha, \beta \in \mathbb{R}$ are parameters. 
The system (\ref{ss-original}) originates from the diffusion equation (e.g., \cite{FV2003})
\begin{equation*}
U_t = (U^{m-1} U_x)_x.
\end{equation*}
The parameter $m$ controls the strength of nonlinear diffusion.
Paying our attention to {\em self-similar solutions} of the form\footnote{
The ansatz (\ref{backward-ss}) represents the {\em backward} self-similarity.
Although there are another types of self-similarity; {\em forward} and {\em exponential-type} ones, the governing equation becomes the same one, (\ref{ss-original}).
} 
\begin{equation}
\label{backward-ss}
U(t,x) = (T-t)^\alpha u(x(T-t)^{\beta}),\quad t < T
\end{equation}
for some $T>0$, the system is reduced to (\ref{ss-original}).
We concentrate on the very fast diffusion case, $m < 0$, and assume that $\beta < 0$.
In what follows we investigate blow-up solutions following our proposed machinery.
First rewrite (\ref{ss-original}) as the first order extended autonomous system:
\begin{equation}
\label{ss}
\begin{cases}
\chi' = 1, & \\
u' = u^{1-m} v, & \\
v' = -\beta \chi u^{1-m}v - \alpha u, & \\
\end{cases}\quad ' = \frac{d}{dt}.
\end{equation}
Direct calculations yield the following observation.

\begin{lem}
The system (\ref{ss}) is asymptotically quasi-homogeneous of type $(0,1,1)$ and order $k+1 = 2-m$.
\end{lem}
A blow-up description of this system is discussed in \cite{Mat2025_NHIM} through a different embedding.
We revisit this system here through both methodologies introduced before and unravel their correspondence.
First, the balance law is
\begin{equation}
\label{ss-balance}
\frac{1}{1-m}U = U^{1-m}V,\quad \frac{1}{1-m}V = -\beta \chi U^{1-m}V,\quad \chi \in \mathbb{R}.
\end{equation}
The nontrivial solution must satisfy
\begin{equation*}
\frac{1}{1-m} = -\beta \chi U^{1-m}\quad \Leftrightarrow \quad U = (-(1-m)\beta \chi)^{-1/(1-m)},\quad \text{provided}\quad \chi \geq 0
\end{equation*}
because $\beta < 0$ and $m < 0$ are assumed.
As a result, we have the following family of roots of (\ref{ss-balance}):
\begin{align}
\notag
(\chi, U, V) &= (\chi, U_\ast(\chi), V_\ast(\chi))\\
\label{ss-balance-root}
&\equiv (\chi,  (-(1-m)\beta \chi)^{-1/(1-m)},  (-\beta \chi)^{-m/(1-m)}(1-m)^{-1/(1-m)} ),\quad \chi \geq 0.
\end{align}

\begin{lem}
Under the assumptions $\beta < 0$ and $m < 0$. the family of roots (\ref{ss-balance-root}) corresponds to an invariant manifold
\begin{equation*}
M = \left\{ \left. (\chi, {\bf x}_\ast(\chi)) \equiv \left(\chi,\, \frac{1}{\sqrt{( 1 + (-\beta \chi)^2 )}},\, \frac{ -\beta \chi }{\sqrt{( 1 + (-\beta \chi)^2 )} }\right) \right| \chi\geq 0 \right\}.
\end{equation*}
on the horizon $\mathcal{E} = \{ x_1^2 + x_2^2 = 1\}$ for the desingularized vector field $g^{\rm ext}$ associated with (\ref{ss}).
\par
The tangent vector along $M_{\mathbb{R}_{>0}}$, namely the restriction of $M$ over the open subinterval $\{\chi > 0\}$, at each point is
\begin{equation}
\label{ss-tangent}
\left(1,\, \frac{-\beta^2 \chi}{( 1 + (-\beta \chi)^2 )^{3/2}},\, \frac{-\beta}{( 1 + (-\beta \chi)^2 )^{3/2}} \right)^T.
\end{equation}
\end{lem}

\begin{proof}
The first statement follows from the correspondence of roots; Theorem \ref{thm-balance-1to1}, through the constant $r_{{\bf Y}_0}$ obtained as follows:
\begin{align}
\notag
r_{{\bf Y}_0}^2 &= (-(1-m)\beta \chi)^{-2/(1-m)} +  (-\beta \chi)^{-2m/(1-m)}(1-m)^{-2/(1-m)}  \\
\notag
	&= (1-m)^{-2/(1-m)}(-\beta \chi)^{-2/(1-m)}\left( 1 + (-\beta \chi)^{-2m/(1-m) + 2/(1-m)} \right)\\
\label{ss-rY0}
	&= (1-m)^{-2/(1-m)}(-\beta \chi)^{-2/(1-m)}\left( 1 + (-\beta \chi)^2 \right).
\end{align}
Regarding the form of $M$ as the curve parameterized by $\chi$, the tangent vector is obtained by differentiating the form by $\chi$: 
\begin{align*}
\frac{d}{d\chi} \frac{1}{\sqrt{( 1 + (-\beta \chi)^2 )}} 
	&= \frac{-\beta^2 \chi}{( 1 + (-\beta \chi)^2 )^{3/2}},\quad
\frac{d}{d\chi} \frac{\chi}{\sqrt{( 1 + (-\beta \chi)^2 )}} 
	=   \frac{ 1}{( 1 + (-\beta \chi)^2 )^{3/2}}.
\end{align*}
\end{proof}
Now several objects associated with the desingularized vector field $g^{\rm ext}$ and $M$ are derived.
The constant $C_\ast$ at $(\chi, {\bf x}_\ast(\chi))$ is
\begin{align*}
C_\ast &= C_\ast(\chi) \equiv G(\chi, {\bf x}_\ast(\chi)) \\
	&= x_{\ast, 1} (x_{\ast, 1}^{1-m} x_{\ast, 2}) + x_{\ast, 2} (-\beta \chi x_{\ast, 1}^{1-m} x_{\ast, 2})\\
	&= (x_{\ast, 1} - \beta \chi x_{\ast, 2}) (x_{\ast, 1}^{1-m} x_{\ast, 2}) \\
	&= \left\{  \frac{1}{\sqrt{( 1 + (-\beta \chi)^2 )}} +  \frac{(-\beta \chi)^2 }{\sqrt{( 1 + (-\beta \chi)^2 )}} \right\} \left\{ \frac{ 1 }{ ( 1 + (-\beta \chi)^2 )^{(1-m)/2} }  \frac{ -\beta \chi }{\sqrt{( 1 + (-\beta \chi)^2 )} } \right\}\\
	&= ( 1 + (-\beta \chi)^2 )^{-\frac{1-m}{2}} (-\beta \chi)
\end{align*}
and the constant $r_{{\bf x}_\ast} = r_{{\bf x}_\ast}(\chi)$ is
\begin{align*}
r_{{\bf x}_\ast}(\chi) &\equiv \left((1-m)C_\ast(\chi)\right)^{-1/(1-m)}\\
	&=  \left\{(1-m)( 1 + (-\beta \chi)^2 )^{-\frac{1-m}{2}} (-\beta \chi) \right\}^{-1/(1-m)}\\
	&= \{(1-m)(-\beta \chi)\}^{-1/(1-m)}  \sqrt{1 + (-\beta \chi)^2}\\
	&= r_{{\bf Y}_0},
\end{align*} 
where $r_{{\bf Y}_0}$ is the constant obtained in (\ref{ss-rY0}), which is compatible with arguments in Section \ref{section-correspondence}.
\par
The projection $P_\ast^{\rm ext} \equiv {\bf v}_{\ast, \alpha}^{\rm ext} (D p_\alpha)_\ast^T$ and $I - P_\ast^{\rm ext}$ are
\begin{align*}
P_\ast^{\rm ext} &= \begin{pmatrix}
0 \\  \frac{1}{\sqrt{( 1 + (-\beta \chi)^2 )}} \\  \frac{-\beta \chi}{\sqrt{( 1 + (-\beta \chi)^2 )}}
\end{pmatrix}\begin{pmatrix}
0 & \frac{1}{\sqrt{( 1 + (-\beta \chi)^2 )}} &  \frac{-\beta \chi}{\sqrt{( 1 + (-\beta \chi)^2 )}}
\end{pmatrix}
	= \frac{1}{1 + (-\beta \chi)^2 }\begin{pmatrix}
0 & 0 & 0 \\
0 & 1 & -\beta \chi \\
0 & -\beta \chi & (-\beta \chi)^2 \\
\end{pmatrix},\\
I - P_\ast^{\rm ext} &= \frac{1}{1 + (-\beta \chi)^2 }\begin{pmatrix}
1+(-\beta \chi)^2 & 0 & 0 \\
0 & (-\beta \chi)^2 & \beta \chi \\
0 & \beta \chi & 1 \\
\end{pmatrix}.
\end{align*}
\par
\bigskip
The blow-up power determining matrix at the root is
\begin{align*}
A^{\rm ext} &=\begin{pmatrix}
0 & 0 & 0\\
0 & -\frac{1}{1-m} + (1-m)U^{-m}V & U^{1-m}\\
-\beta U^{1-m}V & -(1-m)\beta \chi U^{-m}V & -\frac{1}{1-m} -\beta \chi U^{1-m}
\end{pmatrix}_{(\chi, U, V) = (\chi, U_\ast(\chi), V_\ast(\chi) )}\\
&= 
\begin{pmatrix}
0 & 0 & 0\\
0 & -\frac{1}{1-m} + 1 & (-(1-m)\beta \chi)^{-1} \\
((1-m) \chi)^{-1}  (-\beta \chi)^{-m/(1-m)}(1-m)^{-1/(1-m)}  & -\beta \chi & 0
\end{pmatrix}
\end{align*}
and eigenvalues are
\begin{equation*}
0, \quad 1,\quad -\frac{1}{1-m}.
\end{equation*}
The eigenvector of $A^{\rm ext}$ associated with the eigenvalue $0$ is constructed through the unique solution of the linear system
\begin{equation*}
\begin{pmatrix}
-\frac{1}{1-m} + 1 & \{(1-m)(-\beta \chi)\}^{-1} \\
-\beta \chi & 0
\end{pmatrix}\begin{pmatrix}
\tilde v_1 \\ \tilde v_2
\end{pmatrix} = -\begin{pmatrix}
0 \\ ((1-m) \chi)^{-1}  (-\beta \chi)^{-m/(1-m)}(1-m)^{-1/(1-m)}
\end{pmatrix},
\end{equation*}
according to Lemma \ref{lem-eigen-extend-1}.
The solution is
\begin{align*}
\tilde v_1 &= -((1-m) \chi)^{-1}   \{(1-m)(-\beta \chi)\}^{-1/(1-m)},\\
\tilde v_2 &= \frac{m}{1-m} (-\beta) \{(1-m)(-\beta \chi)\}^{-1/(1-m)}
\end{align*}
and the eigenvector is $(1, \tilde v_1, \tilde v_2)^T$.
In particular,
\begin{align*}
(I-P_\ast^{\rm ext}) r_{{\bf x}_\ast}^{-1} \begin{pmatrix}
1\\ \tilde v_1\\ \tilde v_2
\end{pmatrix}
& = \frac{1}{ 1 + (-\beta \chi)^2 }\begin{pmatrix}
1+(-\beta \chi)^2 & 0 & 0 \\
0 & (-\beta \chi)^2 & \beta \chi \\
0 & \beta \chi & 1 \\
\end{pmatrix}\\
& \quad \cdot
\begin{pmatrix}
1 & 0 & 0\\
0 & \frac{1}{ \{(1-m)(-\beta \chi)\}^{-1/(1-m)}  \sqrt{1 + (-\beta \chi)^2} } & 0\\
0 & 0 & \frac{1}{ \{(1-m)(-\beta \chi)\}^{-1/(1-m)}  \sqrt{1 + (-\beta \chi)^2} } \\
\end{pmatrix}\\
&\quad \cdot
\begin{pmatrix}
1\\ 
-((1-m) \chi)^{-1}   \{(1-m)(-\beta \chi)\}^{-1/(1-m)} \\ 
\frac{m}{1-m} (-\beta)  \{(1-m)(-\beta \chi)\}^{-1/(1-m)}
\end{pmatrix}\\
&= \frac{1}{ 1 + (-\beta \chi)^2 }\begin{pmatrix}
1+(-\beta \chi)^2 & 0 & 0 \\
0 & (-\beta \chi)^2 & \beta \chi \\
0 & \beta \chi & 1 \\
\end{pmatrix}
\begin{pmatrix}
1\\ 
\frac{ -\chi^{-1} }{ (1-m) \sqrt{1 + (-\beta \chi)^2} }  \\ 
\frac{m(-\beta) }{(1-m) \sqrt{1 + (-\beta \chi)^2}} 
\end{pmatrix}\\
&= \begin{pmatrix}
1& 
\frac{ -\beta^2 \chi }{ (1 + (-\beta \chi)^2)^{3/2}} &  
\frac{ -\beta }{(1 + (-\beta \chi)^2)^{3/2} } 
\end{pmatrix}^T,
\end{align*}
which is exactly the tangent vector of $M_{\mathbb{R}_{>0}}$ given in (\ref{ss-tangent}).
Note that our present requirement $m<0$ guarantees that Assumption \ref{ass-f-1} holds true. 
Indeed, the quasi-homogeneous part $f^{\rm ext}_{(0,1,1); 1-m}$ and the residual part $f^{\rm ext}_{\rm res}$ of (\ref{ss}) is
\begin{equation*}
f^{\rm ext}_{(0,1,1); 1-m} = \begin{pmatrix}
0  \\
u^{1-m} v  \\
-\beta \chi u^{1-m}v \\
\end{pmatrix},\quad f^{\rm ext}_{\rm res} = \begin{pmatrix}
1 \\
0 \\
- \alpha u \\
\end{pmatrix},
\end{equation*}
respectively. 
The term $f^{\rm ext}_{0;{\rm res}} = 0$ has the order $0$ less than $k + \alpha_0 - 1 = -m > 0$, while  $f^{\rm ext}_{2;{\rm res}} = - \alpha u$ the order $1$ less than $k + \alpha_2 - 1 = 1-m > 1$ (cf. Remark \ref{rem-order-typical}).
Therefore the above observation is compatible with Theorem \ref{thm-blow-up-estr}.
Through the restriction of $M$ to $M_I \equiv M\cap \{\chi \in I\}$ with a compact interval $I\subset \mathbb{R}_{>0}$, we have the following result.
See \cite{Mat2025_NHIM}, Section 6.3.
\begin{thm}
The system (\ref{ss}) with $m, \beta < 0$ admits a solution $(u(\chi), v(\chi))$ blowing up at a finite time $\chi = \chi_{\max} > 0$ with the following asymptotic behavior
\begin{align}
\notag
u(\chi) &\sim (-(1-m)\beta \chi_{\max})^{-1/(1-m)}(\chi_{\max} - \chi)^{-1/(1-m)},\\ 
\label{ss-blow-up}
v(\chi) &\sim (-\beta \chi_{\max})^{-m/(1-m)}(1-m)^{-1/(1-m)}(\chi_{\max} - \chi)^{-1/(1-m)}
\end{align}
as $\chi \to \chi_{\max}$.
\end{thm}
Several remarks about invariant manifolds on the horizon are mentioned in \cite{Mat2025_NHIM}.

\subsection{A Hamiltonian system with time-variant driver}
\label{section-ex-Hamiltonian}
The last example addresses the blow-up behavior of a nonautonomous system
\begin{equation}
\label{ODE_WWL}
u_j'' + u_j^{2n+1} + \frac{i}{u_j}u_1^i \cdots u_m^i a(t) = 0,\quad a\in C^1(\mathbb{R}),\quad j=1,\ldots, m,
\end{equation}
admitting a Hamiltonian $H$ given by
\begin{align*}
H(t, {\bf u}, {\bf v}) &= \sum_{j=1}^m h_j(u_j, v_j) + a(t)\prod_{j=1}^n u_j^i,\quad 
h_j (u_j, v_j) = \frac{1}{2}v_j^2  + \frac{1}{2n+2}u_j^{2n+2},
\end{align*}
where\footnote{
In preceding studies such as \cite{WWL2012}, the function $a(t)$ was assumed to be time-periodic.
On the other hand, it turns out through the following arguments that it is {\em not} required to describe (type-I, stationary, nonautonomous) blow-up solutions of (\ref{ODE_WWL}).
} ${\bf u}=(u_1,\ldots, u_n)^T$, ${\bf v}=(v_1,\ldots, v_n)^T$ with ${\bf v}' = {\bf u}$.
In particular, (\ref{ODE_WWL}) has the form
\begin{equation*}
{\bf u}' = \frac{\partial H}{\partial {\bf v}},\quad {\bf v}' = -\frac{\partial H}{\partial {\bf u}}.
\end{equation*}
%
Several preceding works (e.g. \cite{WWL2012}) reported that (\ref{ODE_WWL}) admitted blow-up solutions {\em provided} that $(2n+2)/m < i$.
For example, if $n=2$ and $m=2$, then $i > 3$ is required.
The present issue is to investigate a nature of blow-ups in the case $(2n+2)/m = i$, which is excluded from arguments in \cite{WWL2012}.

\subsubsection{The case $k>1$}

First we study the case
\begin{equation*}
n=2,\quad m=2\quad \text{ and }\quad i=3.
\end{equation*}
The system (\ref{ODE_WWL}) as the first order system is written as
\begin{align}
\notag
t'&= 1,\\
\notag
u_1' &= v_1,\\
\label{WWL_system}
v_1' &= -u_1^5 - 3 u_1^2u_2^3 a(t),\\
\notag
u_2' &= v_2,\\
\notag
v_2' &= -u_2^5 - 3 u_1^3 u_2^2 a(t).
\end{align}
First we easily have the following observation.

\begin{lem}
Aligning the state variable $(t, u_1, v_1, u_2, v_2)$, the system (\ref{WWL_system}) is asymptotically  quasi-homogeneous of type $\alpha = (0,1,3,1,3)$ and order $k+1 = 3$.
\end{lem}

Next, the balance law is considered to calculate the roots:
\begin{equation}
\label{balance_WWL_system}
\frac{1}{2}U_1 = V_1,\quad \frac{3}{2}V_1 = -U_1^5 - 3 U_1^2U_2^3 a(t),\quad \frac{1}{2}U_2 = V_2,\quad \frac{3}{2}V_2 = -U_2^5 - 3 U_1^3U_2^2 a(t).
\end{equation}
where the time variable $t$ is regarded as a parameter.
Assuming that $U_1, U_2 \not = 0$, the above equations are reduced to the following:
\begin{equation*}
\frac{3}{4} = -U_1^4 - 3 U_1U_2^3 a(t),\quad \frac{3}{4} = -U_2^4 - 3 U_1^3U_2 a(t),
\end{equation*}
in particular
\begin{equation*}
(U_1^2 + U_2^2)(U_1^2 - U_2^2) = 3U_1U_2(U_1^2 - U_2^2)a(t),
\end{equation*}
which yield $U_1^2 = U_2^2$.
Substituting this identity into the above equations, we have
\begin{equation*}
\frac{3}{4} = -U_1^4 \mp 3 U_1^4 a\quad \Rightarrow \quad U_1^4 = \frac{3}{4(-1 \mp 3a(t))},
\end{equation*}
where $\mp$ in the denominator of $U_1^4$ has the following correspondence: $-1 - 3a(t)$ is chosen when $U_1 = U_2$, while $-1 + 3a(t)$ is chosen when $U_1 = -U_2$.
In the above identity, {\em $a(t) < -1/3$ is required when $U_1 = U_2$, while $a(t) > 1/3$ is required when $U_1 = -U_2$}.
As a summary, we have the following observation.
\begin{prop}
Roots of the balance law (\ref{balance_WWL_system}) are the following:
\begin{align}
\label{roots_balance_WWL_system_1}
(U_1, V_1, U_2, V_2) = &\left(\sqrt[4]{\frac{3}{4(-1 - 3a(t))}}, \frac{1}{2}\sqrt[4]{\frac{3}{4(-1 - 3a(t))}}, \sqrt[4]{\frac{3}{4(-1 - 3a(t))}}, \frac{1}{2}\sqrt[4]{\frac{3}{4(-1 - 3a(t))}}\right)\\
\label{roots_balance_WWL_system_2}
	\text{ and }\quad  &\left(\sqrt[4]{\frac{3}{4(-1 + 3a(t))}}, \frac{1}{2}\sqrt[4]{\frac{3}{4(-1 + 3a(t))}}, -\sqrt[4]{\frac{3}{4(-1 + 3a(t))}}, -\frac{1}{2}\sqrt[4]{\frac{3}{4(-1 + 3a(t))}}\right)
\end{align}
provided that $a(t) < -1/3$ when $U_1 = U_2$, namely in (\ref{roots_balance_WWL_system_1}), while $a(t) > 1/3$ when $U_1 = -U_2$, namely in (\ref{roots_balance_WWL_system_2}).
For each point $(U_1, V_1, U_2, V_2)$ given in (\ref{roots_balance_WWL_system_1}) and (\ref{roots_balance_WWL_system_2}), $(-U_1, -V_1, -U_2, -V_2)$ is also a roof of (\ref{balance_WWL_system}).
\end{prop}
The corresponding blow-up power-determining matrix is
\begin{equation}
\label{blow-up_matrix_WWL_system}
-\frac{1}{2}\begin{pmatrix}
0 & 0 & 0 & 0 & 0\\
0 & 1 & 0 & 0 & 0\\
0  & 0 & 3 & 0 & 0\\
0 & 0 & 0 & 1 & 0\\
0 & 0 & 0 & 0 & 3\\
\end{pmatrix} + \begin{pmatrix}
0 & 0 & 0 & 0 & 0\\
0 & 0 & 1 & 0 & 0\\
- 3 U_1^2U_2^3 a'  & -5U_1^4 - 6 U_1U_2^3 a & 0 & - 9U_1^2U_2^2 a & 0\\
0 & 0 & 0 & 0 & 1\\
- 3 U_1^3U_2^2 a' &  - 9U_1^2U_2^2 a & 0 & -5U_2^4 - 6 U_1^3U_2 a & 0\\
\end{pmatrix}.
\end{equation}
Obviously, this matrix admits an eigenvalue $0$.
The remaining eigenvalues are determined by the lower-right $4\times 4$ submatrix whose associated characteristic equation is
\begin{equation*}
\det \begin{pmatrix}
\lambda +\frac{1}{2} & -1 & 0 & 0\\
5A_{\pm} \pm 6A_{\pm} a & \lambda + \frac{3}{2} & 9A_{\pm} a & 0\\
0 & 0 & \lambda + \frac{1}{2} & -1\\
9A_{\pm} a & 0 & 5A_{\pm} \pm 6A_{\pm} a & \lambda + \frac{3}{2}\\
\end{pmatrix} = 0,\quad A_{\pm} = 3/\{4(-1\mp 3a)\}.
\end{equation*}
When $U_1 = U_2$, this is reduced to
\begin{equation*}
\left\{ \left(\lambda + \frac{1}{2}\right) \left( \lambda + \frac{3}{2}\right) + 5A_{+} + 15A_{+} a\right\}\left\{ \left(\lambda + \frac{1}{2}\right) \left( \lambda + \frac{3}{2}\right) + 5A_{+} - 3A_{+} a\right\}= 0,
\end{equation*}
while reduced to
\begin{equation*}
\left\{ \left(\lambda + \frac{1}{2}\right) \left( \lambda + \frac{3}{2}\right) + 5A_{-} + 3A_{-} a  \right\}\left\{ \left(\lambda + \frac{1}{2}\right) \left( \lambda + \frac{3}{2}\right) + 5A_{-} - 15A_{-} a \right\}=0
\end{equation*}
when $U_1 = -U_2$.
Direct calculations yield the following property.
\begin{prop}
We obtain the blow-up power eigenvalues associated with the balance law (\ref{balance_WWL_system}) as follows.
\begin{enumerate}
\item When $U_1 = U_2$, in particular provided that $a(t) < -1/3$, then the corresponding eigenvalues are
\begin{equation*}
\lambda = 0,\quad 1,\quad -3,\quad -1 \pm \sqrt{\frac{ 3a(t)-8 }{2(-1-3a(t))}}.
\end{equation*}
The latter two values are complex conjugate with negative real part.
\item When $U_1 = -U_2$, in particular provided that $a(t) >1/3$, then the corresponding eigenvalues are
\begin{equation*}
\lambda = 0,\quad 1,\quad -3,\quad -1 \pm \sqrt{\frac{ 3a(t)+8 }{2(1 - 3a(t))}}.
\end{equation*}
The latter two values are complex conjugate with negative real part.
\end{enumerate}
\end{prop}
Presence of the eigenvalue $1$ is consistent with Proposition \ref{prop-ev1}.
The Implicit Function Theorem then indicates that there is a compact neighborhood $I_1$ of $t = t_{\max} < \infty$ with $a(t_{\max}) < -1/3$ such that the balance law (\ref{balance_WWL_system}) admits a smooth $t$-family $\mathcal{U}_{I_1}$ of solutions $(U_1(t), V_1(t), U_2(t), V_2(t))$ over $I_1$ satisfying $U_1(t) \equiv U_2(t)$.
The similar conclusion holds for some compact neighborhood $I_2$ of $t = t_{\max} < \infty$ with $a(t_{\max}) > 1/3$ such that the balance law (\ref{balance_WWL_system}) admits a smooth $t$-family $\mathcal{U}_{I_2}$of solutions $(U_1(t), V_1(t), U_2(t), V_2(t))$ over $I_2$ satisfying $U_1(t) \equiv -U_2(t)$.
From the correspondence of equilibria and eigenstructure obtained in Theorems \ref{thm-balance-1to1} and \ref{thm-blow-up-estr}, respectively, the families $\{\mathcal{U}_{I_i}\}_{i=1}^2$ induce those of invariant manifolds for the desingularized vector field $g^{\rm ext}$ associated with (\ref{WWL_system}) and $T_{{\rm para}; \alpha}$ admitting the normally {\em attracting} structure.
It then follows from Theorem \ref{thm-existence-blow-up} the existence of type-I blow-up solutions of (\ref{WWL_system}).
\begin{thm}
\label{thm-blowup-Hamilton-k2}
The system (\ref{WWL_system}) admits blow-up solutions at $t = t_{\max} < \infty$ with
\begin{equation*}
u_1(t)\sim \sqrt[4]{\frac{3}{4(-1 - 3a(t_{\max}))}}(t_{\max} - t)^{-1/2} \quad \text{ and }\quad u_2(t)\sim \sqrt[4]{\frac{3}{4(-1 - 3a(t_{\max}))}}(t_{\max} - t)^{-1/2}
\end{equation*} 
as $t\to t_{\max} - 0$, provided that the initial point $(u_1(t_0), u_2(t_0))$ with $0 < t_{\max} - t_0 \ll 1$ is sufficiently large.
The above blow-ups are attained when $a(t_{\max}) < -1/3$.
\par
Similarly, the system (\ref{WWL_system}) also admits blow-up solutions at $t = t_{\max} < \infty$ with
\begin{equation*}
u_1(t)\sim \sqrt[4]{\frac{3}{4(-1 + 3a(t_{\max}))}}(t_{\max} - t)^{-1/2} \quad \text{ and }\quad u_2(t)\sim -\sqrt[4]{\frac{3}{4(-1 + 3a(t_{\max}))}}(t_{\max} - t)^{-1/2}
\end{equation*} 
as $t\to t_{\max} - 0$, provided that the initial point $(u_1(t_0), u_2(t_0))$ with $0 < t_{\max} - t_0 \ll 1$ is sufficiently large.
The above blow-ups are attained when $a(t_{\max}) > 1/3$.
\end{thm}
As mentioned in the beginning of this subsection, the above result is {\em not} observed in preceding works \cite{WWL2012}, and extracts an influence of nonautonomous term $a(t)$ on blow-up behavior.
Sample computation results with $a(t) = \sin t$ are collected in Table \ref{table-WWL} in Appendix \ref{section-numerics} with the initial point
\begin{equation*}
{\bf x}_0 \equiv (x_{0, 1}, x_{0, 2}, x_{0, 3}, x_{0, 4}) = (0.7, 0.1, 0.1, 0.1)
\end{equation*}
and $t_0$ listed in the table, which support the present arguments.

\begin{rem}
When the constraints for $a(t_{\max})$ provided in Theorem \ref{thm-blowup-Hamilton-k2} are violated, nothing could be understood about blow-up behavior from our correspondence, expect a few observations.
First, if $a(t_{\max}) = -1/3$ when $u_1(t)$ and $u_2(t)$ have identical sign near $t = t_{\max}$, we will have $r_{{\bf Y}_0} = \infty$, correspondingly $C_\ast = 0$ through the identity $r_{{\bf Y}_0} = r_{{\bf x}_\ast}$. 
This will imply that the corresponding equilibrium ${\bf x}_\ast$ on the horizon will be {\em nonhyperbolic} from Propositions \ref{prop-ev-special-1} and \ref{prop-ev-special-2}.
According to the preceding study \cite{Mat2019}, nonhyperbolic equilibria or non-(normally) hyperbolic invariant sets on the horizon (for desingularized vector fields) can induce blow-up behavior different from type-I blow-up rates (referred to as {\em type-II} blow-ups in terminologies for PDEs, or slower rates), including possibilities of {\rm grow-up}; divergence with $t_{\max} = +\infty$.
Second, there {\em could be} cases that solutions blow up at $t = t_{\max}$ whose image under embeddings approach non-equilibrium invariant sets on the horizon (e.g., \cite{Mat2018, Mat2025_NHIM}), if our constraints on $t_{\max}$ stated in Theorem \ref{thm-blowup-Hamilton-k2} are violated.
The balance law (in the present form) provides no information of blow-ups associated with such situations, and direct investigations of desingularized vector fields would be necessary to unravel the genuine nature.
\par
As far as we have numerically investigated the desingularized vector fields associated with (\ref{WWL_system}), no behavior of solutions approaching to non-equilibrium invariant sets on the horizon has been observed.
This would imply that blow-ups induced by invariant sets on the horizon in the exceptional case of $t_{\max}$ would be unstable under perturbation of initial points.
The similar observations will be provided for (\ref{WWL_system_k1}) discussed below.
\end{rem}

\subsubsection{The case $k=1$}
We move to another case
\begin{align}
\notag
t'&= 1,\\
\notag
u_1' &= v_1,\\
\label{WWL_system_k1}
v_1' &= -u_1^3 - 2 u_1 u_2^2 a(t),\\
\notag
u_2' &= v_2,\\
\notag
v_2' &= -u_2^3 - 2 u_1^2 u_2 a(t),
\end{align}
corresponding to
\begin{equation*}
n=1,\quad m=2\quad \text{ and }\quad i=2.
\end{equation*}
\begin{lem}
Aligning the state variable $(t, u_1, v_1, u_2, v_2)$, the system (\ref{WWL_system_k1}) is asymptotically quasi-homogeneous of type $\alpha = (0,1,2,1,2)$ and order $k+1 = 2$, in particular $k=1$.
\end{lem}
The balance law for this case is
\begin{equation}
\label{balance_WWL_system_k1}
U_1 = V_1,\quad 2V_1 = -U_1^3 - 2 U_1 U_2^2 a(t),\quad U_2 = V_2,\quad 2V_2 = -U_2^3 - 2 U_1^2 U_2 a(t),
\end{equation}
regarding $t$ as a parameter.
Assuming that $U_1, U_2 \not = 0$, the above equations are reduced to the following:
\begin{equation*}
2 = -U_1^2 - 2 U_2^2 a(t),\quad 2 = -U_2^2 - 2 U_1^2 a(t),
\end{equation*}
in particular
\begin{equation*}
U_1^2 - U_2^2 = 2(U_1^2 - U_2^2)a(t),
\end{equation*}
which yield $U_1^2 = U_2^2$ for {\em arbitrary} choice of $a(t)$.
Substituting this identity into the above equations, we have
\begin{equation*}
2 = -U_1^2 (1+2 a(t)) \quad \Rightarrow \quad U_1^2 = \frac{-2}{1 + 2a(t)},
\end{equation*}
as far as $a(t) < -1/2$ holds for any case.
Summarizing the above calculations, we have the following observation.
\begin{prop}
Roots of the balance law (\ref{balance_WWL_system_k1}) are the following:
\begin{align}
\label{roots_balance_WWL_system_k1_1}
(U_1, V_1, U_2, V_2) = &\left( \pm \sqrt{\frac{-2}{1 + 2a(t)}}, \pm \sqrt{\frac{-2}{1 + 2a(t)}}, \pm \sqrt{\frac{-2}{1 + 2a(t)}}, \pm \sqrt{\frac{-2}{1 + 2a(t)}} \right)\\
\label{roots_balance_WWL_system_k1_2}
	\text{ and }\quad  &\left( \pm \sqrt{\frac{-2}{1 + 2a(t)}}, \pm \sqrt{\frac{-2}{1 + 2a(t)}}, \mp \sqrt{\frac{-2}{1 + 2a(t)}}, \mp \sqrt{\frac{-2}{1 + 2a(t)}} \right)
\end{align}
provided that $a(t) < -1/2$.
\end{prop}

The corresponding blow-up power-determining matrix is
\begin{equation}
\label{blow-up_matrix_WWL_system_k1}
-\begin{pmatrix}
0 & 0 & 0 & 0 & 0\\
0 & 1 & 0 & 0 & 0\\
0  & 0 & 2 & 0 & 0\\
0 & 0 & 0 & 1 & 0\\
0 & 0 & 0 & 0 & 2\\
\end{pmatrix} + \begin{pmatrix}
0 & 0 & 0 & 0 & 0\\
0 & 0 & 1 & 0 & 0\\
- 2 U_1 U_2^2 a'  & -3U_1^2 - 2 U_2^2 a & 0 & - 4U_1 U_2 a & 0\\
0 & 0 & 0 & 0 & 1\\
- 2 U_1^2 U_2 a' &  - 4U_1 U_2 a & 0 & -3U_2^2 - 2 U_1^2 a & 0\\
\end{pmatrix}.
\end{equation}
Obviously, this matrix admits an eigenvalue $0$.
The remaining eigenvalues are determined by the lower-right $4\times 4$ submatrix whose associated characteristic equation is
\begin{equation*}
\det \begin{pmatrix}
\lambda +1 & -1 & 0 & 0\\
 3B + 2 B a & \lambda +2 & \pm 4B a & 0\\
0 & 0 & \lambda +1 & -1\\
 \pm 4B a & 0 & 3B +  2 B a & \lambda +2\\
\end{pmatrix} = 0,\quad B = \frac{-2}{1 + 2a},
\end{equation*}
where \lq\lq $\pm$" in the determinant corresponds to the sign of $U_1U_2$.
This is reduced to
\begin{equation*}
\left\{ \left(\lambda + 1\right) \left( \lambda + 2\right) + (3+6a)B \right\}\left\{ \left(\lambda + 1\right) \left( \lambda + 2\right) + (3-2a)B  \right\}= 0.
\end{equation*}
This reduction is valid for both cases $U_1 = \pm U_2$.
Substitution of $B$ yields
\begin{equation*}
\left\{ \left(\lambda + 1\right) \left( \lambda + 2\right) -6  \right\}\left\{ \left(\lambda + 1\right) \left( \lambda + 2\right) - \frac{2(3-2a)}{1 + 2a}  \right\}= 0.
\end{equation*}
Direct calculations yield the following property.
\begin{prop}
\label{prop-ev-WWL_k1}
The blow-up power eigenvalues associated with the balance law (\ref{balance_WWL_system_k1}) with $a(t) < -1/2$ are as follows:
\begin{equation*}
\lambda = 0,\quad 1,\quad -4,\quad \frac{1}{2}\left\{ -3 \pm \sqrt{9 - \frac{ 16(-1+2a(t)) }{ 1+2a(t) }} \right\}.
\end{equation*}
The latter two values are complex conjugate with negative real part.
\end{prop}
Presence of the eigenvalue $1$ is consistent with Proposition \ref{prop-ev1}.
The Implicit Function Theorem then indicates that there is a compact neighborhood $I_{\pm}$ of $t = t_{\max} < \infty$ with $a(t_{\max}) < -1/2$ such that the balance law (\ref{balance_WWL_system_k1}) admits a smooth $t$-family $\mathcal{U}_{\pm}$ of solutions $(U_1(t), V_1(t), \pm U_2(t), \pm V_2(t))$ over $I_{\pm}$ satisfying $U_1(t) \equiv \pm U_2(t)$.
From the correspondence of equilibria and eigenstructure obtained in Theorems \ref{thm-balance-1to1} and \ref{thm-blow-up-estr}, respectively, the families $\{\mathcal{U}_{I_\pm}\}$ induce those of invariant manifolds for the desingularized vector field $g^{\rm ext}$ associated with (\ref{WWL_system_k1}) and $T_{{\rm para}; \alpha}$ admitting the normally {\em attracting} structure.
It then follows from Theorem \ref{thm-existence-blow-up} the existence of type-I blow-up solutions of (\ref{WWL_system_k1}).
\begin{thm}
\label{thm-blowup-Hamilton-k1}
The system (\ref{WWL_system_k1}) admits blow-up solutions at $t = t_{\max} < \infty$ with
\begin{equation*}
u_1(t)\sim \pm \sqrt{\frac{-2}{1 + 2a(t_{\max})}} (t_{\max} - t)^{-1} \quad \text{ and }\quad u_2(t)\sim \mp \sqrt{ \frac{-2}{1 + 2a(t_{\max})}} (t_{\max} - t)^{-1}
\end{equation*} 
as $t\to t_{\max} - 0$, provided that the initial point $(u_1(t_0), u_2(t_0))$ with $0 < t_{\max} - t_0 \ll 1$ has sufficiently large modulus.
These blow-ups are attained when $a(t_{\max}) < -1/2$.
\end{thm}

Sample computation results with $a(t) = \sin t$ as well as numerical arguments supporting the reliability of our theoretical results are collected in Appendix \ref{section-numerics}.

\section*{Concluding Remarks}

The present paper has addressed the extraction of simple characteristics determining the existence of blow-up solutions for nonautonomous systems of ODEs.
The key idea is, as achieved in autonomous case \cite{asym2}, the connection between objects providing multi-order asymptotic expansions of blow-ups (cf. \cite{asym1}) and \lq\lq dynamics at infinity" by means of local dynamical nature around equilibria on the horizon for the desingularized vector fields.
Nonautonomous terms, depending on the time variable $t$ as a parameter, require several modifications of the above characteristics from autonomous cases, while they do not violate the essential idea towards our achievement.
In particular, the leading coefficients of blow-ups (\cite{asym1}), namely $t$-families of the roots of balance laws, and associated eigenstructures naturally construct NHIMs on the horizon for the desingularized vector fields, indicating that {\em (sometimes simple) analytic information in blow-ups can naturally correspond to geometric nature in dynamics at infinity}, which yields a simple criterion of (non)existence of blow-ups with specific nature.
The essence presented in this paper will be extended to more complex blow-ups such as {\em periodic blow-ups} (\cite{Mat2018}); the blow-ups with infinitely many oscillations, and {\em normally hyperbolic blow-ups} (\cite{Mat2025_NHIM}); the blow-ups shadowing \lq\lq {\em NHIMs at infinity}", which will be one of the next directions.
Another direction will be description of asymptotic behavior of blow-ups induced by {\em non-normally hyperbolic} invariant manifolds on the horizon for $g$ (or $g^{\rm ext}$), as observed in e.g. \cite{K2014}.

\section*{Acknowledgements}
KM was partially supported by World Premier International Research Center Initiative (WPI), Ministry of Education, Culture, Sports, Science and Technology (MEXT), Japan, and JSPS Grant-in-Aid for Scientists (B) (No. JP23K20813).

\bibliographystyle{plain}
\bibliography{blow_up_asymptotic_nonaut_R1}

\begin{thebibliography}{10}

\bibitem{A2002}
B.~Andrews.
\newblock Singularities in crystalline curvature flows.
\newblock {\em Asian J. Math.}, 6(1):101--122, 2002.

\bibitem{asym1}
T.~Asai, H.~Kodani, K.~Matsue, H.~Ochiai, and T.~Sasaki.
\newblock Multi-order asymptotic expansion of blow-up solutions for autonomous
  {ODE}s: {I}. {M}ethod and {J}ustification.
\newblock {\em Nonlinearity}, 38(4):045003, 2025.

\bibitem{BFGK2011}
E.~Berchio, A.~Ferrero, F.~Gazzola, and P.~Karageorgis.
\newblock Qualitative behavior of global solutions to some nonlinear fourth
  order differential equations.
\newblock {\em Journal of Differential Equations}, 251(10):2696--2727, 2011.

\bibitem{BK1988}
M.~Berger and R.V. Kohn.
\newblock A rescaling algorithm for the numerical calculation of blowing-up
  solutions.
\newblock {\em Communications on pure and applied mathematics}, 41(6):841--863,
  1988.

\bibitem{CHO2007}
C.-H. Cho, S.~Hamada, and H.~Okamoto.
\newblock On the finite difference approximation for a parabolic blow-up
  problem.
\newblock {\em Japan Journal of Industrial and Applied Mathematics},
  24(2):131--160, 2007.

\bibitem{D1985}
J.W. Dold.
\newblock Analysis of the early stage of thermal runaway.
\newblock {\em The Quarterly Journal of Mechanics and Applied Mathematics},
  38(3):361--387, 1985.

\bibitem{D1993}
F.~Dumortier.
\newblock Techniques in the theory of local bifurcations: Blow-up, normal
  forms, nilpotent bifurcations, singular perturbations.
\newblock In {\em Bifurcations and Periodic Orbits of Vector Fields}, pages
  19--73. Springer, 1993.

\bibitem{DNZ2022}
G.K. Duong, N.~Nouaili, and H.~Zaag.
\newblock Refined asymptotics for the blow-up solution of the complex
  {G}inzburg-{L}andau equation in the subcritical case.
\newblock {\em Annales de l'{I}nstitut {H}enri {P}oincar{\'e} {C}},
  39(1):41--85, 2022.

\bibitem{FV2003}
R.~Ferreira and J.L. V{\'a}zquez.
\newblock Study of self-similarity for the fast-diffusion equation.
\newblock {\em Advances in Differential Equations}, 8(9):1125--1152, 2003.

\bibitem{FM2002}
M.~Fila and H.~Matano.
\newblock Blow-up in nonlinear heat equations from the dynamical systems point
  of view.
\newblock {\em Handbook of dynamical systems}, 2:723--758, 2002.

\bibitem{GV2002}
V.A. Galaktionov and J.-L. V{\'a}zquez.
\newblock The problem of blow-up in nonlinear parabolic equations.
\newblock {\em Discrete \& Continuous Dynamical Systems-A}, 8(2):399, 2002.

\bibitem{GP2011}
F.~Gazzola and R.~Pavani.
\newblock Blow up oscillating solutions to some nonlinear fourth order
  differential equations.
\newblock {\em Nonlinear Analysis: Theory, Methods \& Applications},
  74(17):6696--6711, 2011.

\bibitem{GP2013}
F.~Gazzola and R.~Pavani.
\newblock Wide oscillation finite time blow up for solutions to nonlinear
  fourth order differential equations.
\newblock {\em Archive for Rational Mechanics and Analysis}, 207(2):717--752,
  2013.

\bibitem{H2016}
T.-H. Hsu.
\newblock Viscous singular shock profiles for a system of conservation laws
  modeling two-phase flow.
\newblock {\em Journal of {D}ifferential {E}quations}, 261(4):2300--2333, 2016.

\bibitem{IY2003}
T.~Ishiwata and S.~Yazaki.
\newblock On the blow-up rate for fast blow-up solutions arising in an
  anisotropic crystalline motion.
\newblock {\em Journal of Computational and Applied Mathematics},
  159(1):55--64, 2003.

\bibitem{asym2}
H.~Kodani, K.~Matsue, H.~Ochiai, and A.~Takayasu.
\newblock Multi-order asymptotic expansion of blow-up solutions for autonomous
  {ODE}s: {II}. {D}ynamical {C}orrespondence.
\newblock {\em Nonlinearity}, 38(4):045004, 2025.

\bibitem{KJH1997}
M.D. Kruskal, N.~Joshi, and R.~Halburd.
\newblock Analytic and asymptotic methods for nonlinear singularity analysis: a
  review and extensions of tests for the {P}ainlev{\'e} property.
\newblock {\em Integrability of nonlinear systems}, pages 171--205, 1997.

\bibitem{K2014}
C.~Kuehn.
\newblock Normal hyperbolicity and unbounded critical manifolds.
\newblock {\em Nonlinearity}, 27(6):1351, 2014.

\bibitem{K2013}
Y.A. Kuznetsov.
\newblock {\em Elements of applied bifurcation theory (3rd Edition)}, volume
  112.
\newblock Springer Science \& Business Media, 2013.

\bibitem{LMT2023}
J.-P. Lessard, K.~Matsue, and A.~Takayasu.
\newblock Saddle-{T}ype {B}low-{U}p {S}olutions with {C}omputer-{A}ssisted
  {P}roofs: {V}alidation and {E}xtraction of {G}lobal {N}ature.
\newblock {\em Journal of {N}onlinear {S}cience}, 33(3):46, 2023.

\bibitem{MZ2008}
N.~Masmoudi and H.~Zaag.
\newblock Blow-up profile for the complex {G}inzburg-{L}andau equation.
\newblock {\em Journal of {F}unctional {A}nalysis}, 255(7):1613--1666, 2008.

\bibitem{Mat2018}
K.~Matsue.
\newblock On blow-up solutions of differential equations with
  {P}oincar\'{e}-type compactifications.
\newblock {\em SIAM Journal on Applied Dynamical Systems}, 17(3):2249--2288,
  2018.

\bibitem{Mat2019}
K.~Matsue.
\newblock Geometric treatments and a common mechanism in finite-time
  singularities for autonomous {ODE}s.
\newblock {\em Journal of Differential Equations}, 267(12):7313--7368, 2019.

\bibitem{Mat2025_NHIM}
K.~Matsue.
\newblock Blow-up behavior for {ODE}s with normally hyperbolic nature in
  dynamics at infinity.
\newblock {\em SIAM Journal on Applied Dynamical Systems}, 24(1):415--456,
  2025.

\bibitem{MT2020_1}
K.~Matsue and A.~Takayasu.
\newblock Numerical validation of blow-up solutions with quasi-homogeneous
  compactifications.
\newblock {\em Numerische Mathematik}, 145:605--654, 2020.

\bibitem{MT2020_2}
K.~Matsue and A.~Takayasu.
\newblock Rigorous numerics of blow-up solutions for {ODE}s with exponential
  nonlinearity.
\newblock {\em Journal of Computational and Applied Mathematics}, 374:112607,
  2020.

\bibitem{M2016}
N.~Mizoguchi.
\newblock Type {II} blowup in a doubly parabolic {K}eller-{S}egel system in two
  dimensions.
\newblock {\em Journal of Functional Analysis}, 271(11):3323--3347, 2016.

\bibitem{PS2001}
P.~Plech{\'a}{\v{c}} and V.~{\v{S}}ver{\'a}k.
\newblock On self-similar singular solutions of the complex {G}inzburg-{L}andau
  equation.
\newblock {\em Communications on {P}ure and {A}pplied {M}athematics},
  54(10):1215--1242, 2001.

\bibitem{S2004}
S.~Schecter.
\newblock Existence of {D}afermos profiles for singular shocks.
\newblock {\em Journal of Differential Equations}, 205(1):185--210, 2004.

\bibitem{SS1999}
C.~Sulem and P.-L. Sulem.
\newblock {\em The {N}onlinear {S}chr{\"o}dinger {E}quation: {S}elf-{F}ocusing
  and {W}ave {C}ollapse}, volume 139.
\newblock Springer, 1999.

\bibitem{TMSTMO2017}
A.~Takayasu, K.~Matsue, T.~Sasaki, K.~Tanaka, M.~Mizuguchi, and S.~Oishi.
\newblock Numerical validation of blow-up solutions for ordinary differential
  equations.
\newblock {\em Journal of Computational and Applied Mathematics}, 314:10--29,
  2017.

\bibitem{WWL2012}
Z.~Wang, Y.~Wang, and H.~Lu.
\newblock The coexistence of quasi-periodic and blow-up solutions in a class of
  {H}amiltonian systems.
\newblock {\em Journal of {M}athematical {A}nalysis and {A}pplications},
  388(2):888--898, 2012.

\bibitem{W2013_NHIM}
S.~Wiggins.
\newblock {\em Normally hyperbolic invariant manifolds in dynamical systems},
  volume 105.
\newblock Springer Science \& Business Media, 2013.

\bibitem{W2013_KS}
M.~Winkler.
\newblock Finite-time blow-up in the higher-dimensional parabolic-parabolic
  {K}eller-{S}egel system.
\newblock {\em Journal de Math{\'e}matiques Pures et Appliqu{\'e}es},
  100(5):748--767, 2013.

\end{thebibliography}

\appendix
\section{Algebraic arrangements in eigenstructure of matrices}
\label{section-algebra}

In this appendix, a technique for characterizing eigenstructures of matrices associated with nonautonomous systems as the extended autonomous systems is discussed.
In typical applications, nonautonomous systems have the forms satisfying Assumption \ref{ass-f-1}.
In particular, we take the structure of (linearized) matrices derived from nonautonomous systems into account.
\par
For $n\geq 1$, let $A \in M_{n+1}(\mathbb{R})$ be a matrix.
We consider two particular forms of $A$ and complete description of eigenstructures by means of those of {\em reduced} matrices, denoted by $\tilde A$.

\begin{ass} 
\label{ass-matrix-1}
The matrix $A$ has the following structure:
\begin{equation}
A = \begin{pmatrix}
a & {\bf 0}_n^T \\
\tilde {\bf u} & \tilde A
\end{pmatrix},
\end{equation}
where $a\in \mathbb{R}$, $\tilde {\bf u} = (u_1, \ldots, u_n)^T\in \mathbb{R}^n$ and $\tilde A\in M_n(\mathbb{R})$. 
Moreover, $a$ is a simple eigenvalue of $A$. 
In particular, $a\not \in {\rm Spec}(\tilde A)$.
\end{ass}
In the above situations, all eigenstructures of $A$ are constructed by those of $\tilde A$.

\begin{lem}
\label{lem-eigen-extend-1}
Suppose that Assumption \ref{ass-matrix-1} is satisfied.
Then $(a, (1,v_1,\ldots, v_n)^T) \in \mathbb{R}\times \mathbb{R}^{n+1}$ is an eigenpair of $A$, where $(v_1,\ldots, v_n)^T \equiv \tilde {\bf v}$ is the unique solution of the linear system
\begin{equation*}
(\tilde A - aI_n)\tilde {\bf v} = -\tilde {\bf u}.
\end{equation*}
In particular, if $\tilde {\bf u} = 0$, then $\tilde {\bf v} = 0$.
\par
Next, assume that $\lambda \not = a$.
Then the pair $(\lambda, \tilde {\bf v})\in \mathbb{C}\times \mathbb{C}^n$ with $\tilde {\bf v} = (v_1, \ldots, v_n)^T\in \mathbb{C}^n$ is an eigenpair of $\tilde A$
if and only if the pair $(\lambda, {\bf v}) \in \mathbb{C}\times \mathbb{C}^{n+1}$ with
${\bf v} = (0, \tilde {\bf v})^T$ is an eigenpair of $A$.
\end{lem}

\begin{proof}
The first assertion requires to construct a linear system.
Now
\begin{align*}
\begin{pmatrix}
a & {\bf 0}_n^T \\
\tilde {\bf u} & \tilde A
\end{pmatrix}\begin{pmatrix}
1 \\ v_1 \\ \vdots \\ v_n
\end{pmatrix} &= \begin{pmatrix}
a \\ u_1 + \sum_{j=1}^n \tilde a_{1j}v_j \\ \vdots \\ u_n + \sum_{j=1}^n \tilde a_{nj}v_j
\end{pmatrix}.
\end{align*}
Our requirement here is the vector is the eigenvector associated with the \lq\lq eigenvalue" $a$, namely
\begin{equation*}
\begin{pmatrix}
a \\ u_1 + \sum_{j=1}^n \tilde a_{1j}v_j \\ \vdots \\ u_n + \sum_{j=1}^n \tilde a_{nj}v_j
\end{pmatrix} = a\begin{pmatrix}
1 \\ v_1 \\ \vdots \\ v_n
\end{pmatrix},
\end{equation*}
in particular our problem is reduced to the unique solvability of the following linear system:
\begin{equation}
\label{linear-eigen-extend-1}
(\tilde A - aI_n)\tilde {\bf v} = -\tilde {\bf u},\quad {\bf v} = (v_1, \ldots, v_n)^T
\end{equation}
Now the assumption $a\not \in {\rm Spec}(\tilde A)$ implies that the coefficient matrix $\tilde A - aI_n$ is nonsingular, and hence the system (\ref{linear-eigen-extend-1}) is uniquely solvable to obtain $\tilde {\bf v}$.
As a result, the first statement holds.
\par 
Next, assume that $(\lambda, {\bf v})$ with $\lambda \not = a$ and ${\bf v} = (v_0, \tilde {\bf v}) \in \mathbb{C}^{1+n}$ be an eigenpair of $A$. 
Then 
\begin{equation*}
\lambda \begin{pmatrix}
v_0 \\ \tilde {\bf v}
\end{pmatrix} \equiv \lambda {\bf v} = A{\bf v} \equiv \begin{pmatrix}
a & {\bf 0}_n^T \\
\tilde {\bf u} & \tilde A
\end{pmatrix}\begin{pmatrix}
v_0 \\ \tilde {\bf v}
\end{pmatrix} = \begin{pmatrix}
a v_0 \\ v_0 \tilde {\bf u} + \tilde A\tilde {\bf v}
\end{pmatrix},
\end{equation*}
which indicates that $v_0 = 0$ because $\lambda \not = a$.
Consequently, $\tilde A\tilde {\bf v} =  \lambda \tilde {\bf v}$.
The converse follows from direct calculations.
\end{proof}

In our applications in the main part, $a=0$ is always assumed.
Because ${\rm Spec}(\tilde A) \cap i\mathbb{R} = \emptyset$ is always assumed for the corresponding matrix $\tilde A$, assumptions involving eigenvalues are always valid.
Assumption \ref{ass-matrix-1} is therefore sufficient to construct all eigenpairs of $A$ assuming the knowledge for $\tilde A$.

\section{Numerical investigation of characteristics in Example \ref{section-ex-Hamiltonian}}
\label{section-numerics}

Here several numerical observations showing the reliability of results in Section \ref{section-correspondence} for (\ref{WWL_system_k1});
\begin{align*}
t'&= 1,\\
u_1' &= v_1,\\
v_1' &= -u_1^3 - 2 u_1 u_2^2 \sin(t),\\
u_2' &= v_2,\\
v_2' &= -u_2^3 - 2 u_1^2 u_2 \sin(t),
\end{align*}
are collected.
In particular, characteristics newly observed in the nonautonomous setting compared with autonomous ones, \cite{asym2}, are investigated.
\par
\bigskip
The vector field (\ref{WWL_system_k1}) is order $2$, namely $k=1$, which induces different characteristics of \lq\lq transversal" eigenpairs among $A$ and $Dg_\ast^{\rm ext}$ from the remaining cases, $k>1$.
As an example, we consider $U_1 = U_2$.
\par
The desingularized vector field $g^{\rm ext}(t, {\bf x})$ associated with (\ref{WWL_system_k1}) under the embedding $(t, u_1, v_1, u_2, v_2)\mapsto (t, {\bf x})$ with 
\begin{align*}
&u_1 = \frac{x_1}{1 - p_\alpha^4},\quad  v_1 = \frac{x_2}{(1 - p_\alpha^4)^2},\quad  u_2 = \frac{x_3}{1 - p_\alpha^4},\quad  v_2 = \frac{x_4}{(1 - p_\alpha^4)^2}, \\
&p_\alpha^4 \equiv p_\alpha({\bf x})^4 = x_1^4 + x_2^2 + x_3^4 + x_4^2
\end{align*}
consists of
\begin{align*}
g^{\rm ext}(t, {\bf x}) &= \begin{pmatrix}
\frac{1}{4}(1+3p_\alpha^4({\bf x})) (1-p_\alpha^4({\bf x})) \\ 
g(t,{\bf x})
\end{pmatrix},\\
g(t,{\bf x}) &= \frac{1}{4}(1+3p_\alpha({\bf x})^4) \tilde f(t,{\bf x}) - G(t,{\bf x}) (x_1, 2x_2, x_3, 2x_4)^T,\\
\tilde f_1(t,{\bf x}) &= x_2, \quad \tilde f_2(t,{\bf x}) = -x_1^3 - 2 x_1 x_3^2 a(t),\quad 
\tilde f_3(t,{\bf x}) = x_4, \quad \tilde f_4(t,{\bf x}) = -x_3^3 - 2 x_1^2 x_3 a(t),\\
G(t,{\bf x}) &= x_1^{3}\tilde f_1(t, {\bf x}) + \frac{1}{2}x_2 \tilde f_2(t, {\bf x})  + x_3^{3}\tilde f_3(t, {\bf x}) + \frac{1}{2}x_4 \tilde f_4(t, {\bf x}).
\end{align*}
We shall investigate typical trajectories approaching to the horizon as $\tau \to +\infty$ by numerical simulations through, e.g., the standard 4th order Runge-Kutta scheme with the initial point
\begin{equation*}
{\bf x}_0 \equiv (x_{0, 1}, x_{0, 2}, x_{0, 3}, x_{0, 4}) = (0.7, 0.1, 0.1, 0.1),
\end{equation*}
while the initial time $t_0$ is chosen as various values.
The trajectory approaches to an equilibrium, which depends on $t_0$.
On the other hand, the quantity $r_{{\bf Y}_0}$ associated with 
\begin{equation*}
(t_{\max}, {\bf Y}_0) = \left( t_{\max},  \sqrt{\frac{-2}{1 + 2a(t_{\max})}}, \sqrt{\frac{-2}{1 + 2a(t_{\max})}}, \sqrt{\frac{-2}{1 + 2a(t_{\max})}},  \sqrt{\frac{-2}{1 + 2a(t_{\max})}} \right)
\end{equation*}
is 
\begin{align*}
r_{{\bf Y}_0} &= p_\alpha({\bf Y}_0) \\
	&= \left\{  \left( \frac{-2}{1 + 2a(t_{\max})} \right)^2 + \frac{-2}{1 + 2a(t_{\max})} + \left( \frac{-2}{1 + 2a(t_{\max})} \right)^2 +  \frac{-2}{1 + 2a(t_{\max})}  \right\}^{1/4} \\
	&= \left[ 2\left( \frac{-2}{1 + 2a(t_{\max})} \right) \left\{ \frac{-2}{1 + 2a(t_{\max})} + 1\right\}  \right]^{1/4}.
\end{align*}
From the correspondence among roots of the balance law and equilibria on the horizon (Theorem \ref{thm-balance-1to1}), the corresponding equilibrium will be ${\bf x}_\ast = (x_{\ast, 1}, x_{\ast, 1}, x_{\ast, 1}, x_{\ast, 1})^T$ with
\begin{align*}
x_{\ast, 1} &= x_{\ast, 3} = \left[ \frac{1}{2}\left( \frac{-2}{1 + 2a(t_{\max})} \right) \left\{ \frac{-2}{1 + 2a(t_{\max})} + 1\right\}^{-1} \right]^{1/4},\\
x_{\ast, 2} &= x_{\ast, 4} = \left[ 2 \left\{ \frac{-2}{1 + 2a(t_{\max})} + 1\right\}  \right]^{-1/2}
\end{align*}
and 
\begin{align*}
C_\ast &= G(t_{\max}, {\bf x}_\ast)\\
	&= x_{\ast, 1}^{3}x_{\ast, 2} - \frac{1}{2}x_{\ast, 1}^3 x_{\ast, 2} \{1 + 2  a(t_{\max})\} + x_{\ast, 1}^{3}x_{\ast, 2} - \frac{1}{2} x_{\ast, 1}^3 x_{\ast, 2} \{1 + 2 a(t_{\max})\}\\
	&=x_{\ast, 1}^{3}x_{\ast, 2} \{ 1 - 2 a(t_{\max})\}\\
	&= \frac{1}{\sqrt{2}} \left( \frac{1}{| 1 + 2a(t_{\max}) |} \right)^{3/4} \left\{ \frac{1 - 2a(t_{\max})}{|1 + 2a(t_{\max})|} \right\}^{-5/4}\{ 1 - 2 a(t_{\max})\} \\
	&=  \sqrt{\frac{| 1 + 2a(t_{\max}) |}{2 \sqrt{1 - 2a(t_{\max})}}}. 
\end{align*}
A demonstrating result with $t_0 = 0.02$ is shown below.
The numerically computed trajectory is shown in Figure \ref{fig-WWL_k1}, which approaches to a point $(t_{\max}, {\bf x}_\ast)$ with
\begin{align*}
(t_{\max}, {\bf x}_\ast) &\approx (23.1796559, 0.769262666, 0.387058299, 0.769262666, 0.387058299).
\end{align*}
The corresponding value $a(t_{\max}) = \sin(t_{\max})$ is $-0.927813165 \cdots$, and 
(approximate) eigenvalues of the associated Jacobian matrices are
\begin{equation*}
0,\quad -2.01261970,\quad -0.503154925,\quad -0.754732388 \pm 1.67633491 i.
\end{equation*}
On the other hand, the constant $C_\ast$ is
\begin{equation*}
C_\ast \approx \sqrt{\frac{2\cdot 0.927813165 - 1}{2 \sqrt{1 + 2 \cdot 0.927813165 }}} = \sqrt{\frac{0.85562633}{2 \cdot 1.68985985514}} \approx 
\sqrt{0.25316487855} \approx 0.503154925,
\end{equation*}
where we see that $-C_\ast$ is indeed an eigenvalue.
Moreover, we have 
\begin{align*}
\frac{-2.01261970}{0.503154925}\approx 4,\quad \frac{-0.754732388 \pm 1.67633491 i}{0.503154925} \approx -\frac{3}{2} \pm 3.33164762324 i.
\end{align*}
Noting the fact that
\begin{align*}
\frac{1}{2}\sqrt{ \left| 9 - \frac{ 16(-1+2\sin(t_{\max})) }{ 1+2\sin(t_{\max}) } \right| } &\approx \frac{1}{2}\sqrt{ \left| 9 - \frac{ 16(-2.85562633) }{ -0.85562633} \right| }
	\approx \frac{1}{2}\sqrt{44.3995035894} \approx 3.33164762503
\end{align*}
from Proposition \ref{prop-ev-WWL_k1}, the above calculations support the validity of the eigenvalue correspondence stated in Theorem \ref{thm-blow-up-estr}.
\par
The remaining nontrivial object\footnote{
Correspondence of the tangential eigenvectors in concrete examples are already confirmed in \cite{asym2}.
} is the transversal eigenvector $\tilde {\bf v}^{\rm ext}_{\ast,\alpha}$ given in (\ref{evec-transversal-desing}), associated with the eigenvalue $-C_\ast$.
Numerically computed eigenvector is 
\begin{equation}
\label{evec_computed}
(1, 0.029769405, 0.25496239, 0.029769405, 0.25496239)^T.
\end{equation}

Now
\begin{align*}
{\bf v}_{\ast, \alpha} &= (0.769262666, 0.7741166, 0.769262666, 0.7741166)^T,\\
(D_tg)_\ast &\approx (0.10112685, -0.2378724, 0.10112685, 0.2378724)^T
\end{align*}
and the matrix $A_g + C_\ast I_4$ is invertible because $-C_\ast$ is {\em not} an eigenvalue\footnote{
It follows from the proof of Lemma \ref{lem-ass-f-1} or corresponding argument shown in Theorem 3.11 of \cite{asym2} that $A_g$ admits an eigenvalue $+C_\ast$, not $-C_\ast$ in general.
} of $A_g$.
Note that the vector (\ref{evec_computed}) is different from ${\bf v}_{\ast, \alpha}^{\rm ext}$.
From these quantities, we obtain the constant $d$ in (\ref{const_d}) as follows:
\begin{equation*}
d \equiv \frac{ (D_{\bf x}p_\alpha)_\ast^T (A_g + C_\ast I_n)^{INV} (D_t g)_\ast}{ (D_{\bf x}p_\alpha)_\ast^T (A_g + C_\ast I_n)^{INV} {\bf v}_{\ast, \alpha}} = \frac{0.14587053}{1.9874594} \approx 0.07339548
\end{equation*}
and
\begin{align*}
&\left\{ \frac{C_\ast}{2c} + d (A_g + C_\ast I_n )^{INV} \right\} {\bf v}_{\ast, \alpha} - (A_g + C_\ast I_n )^{INV} (D_t g)_\ast\\
&\quad = \begin{pmatrix}
0.029769405,\ 0.25496239,\ 0.029769405,\ 0.25496239 
\end{pmatrix}^T,
\end{align*}
implying the identity in (\ref{evec-transversal-desing}).
Other sample computation results are collected in Table \ref{table-WWL_k1} supporting the present arguments.

\begin{figure}[h!]\em
\centering
\includegraphics[width=8cm]{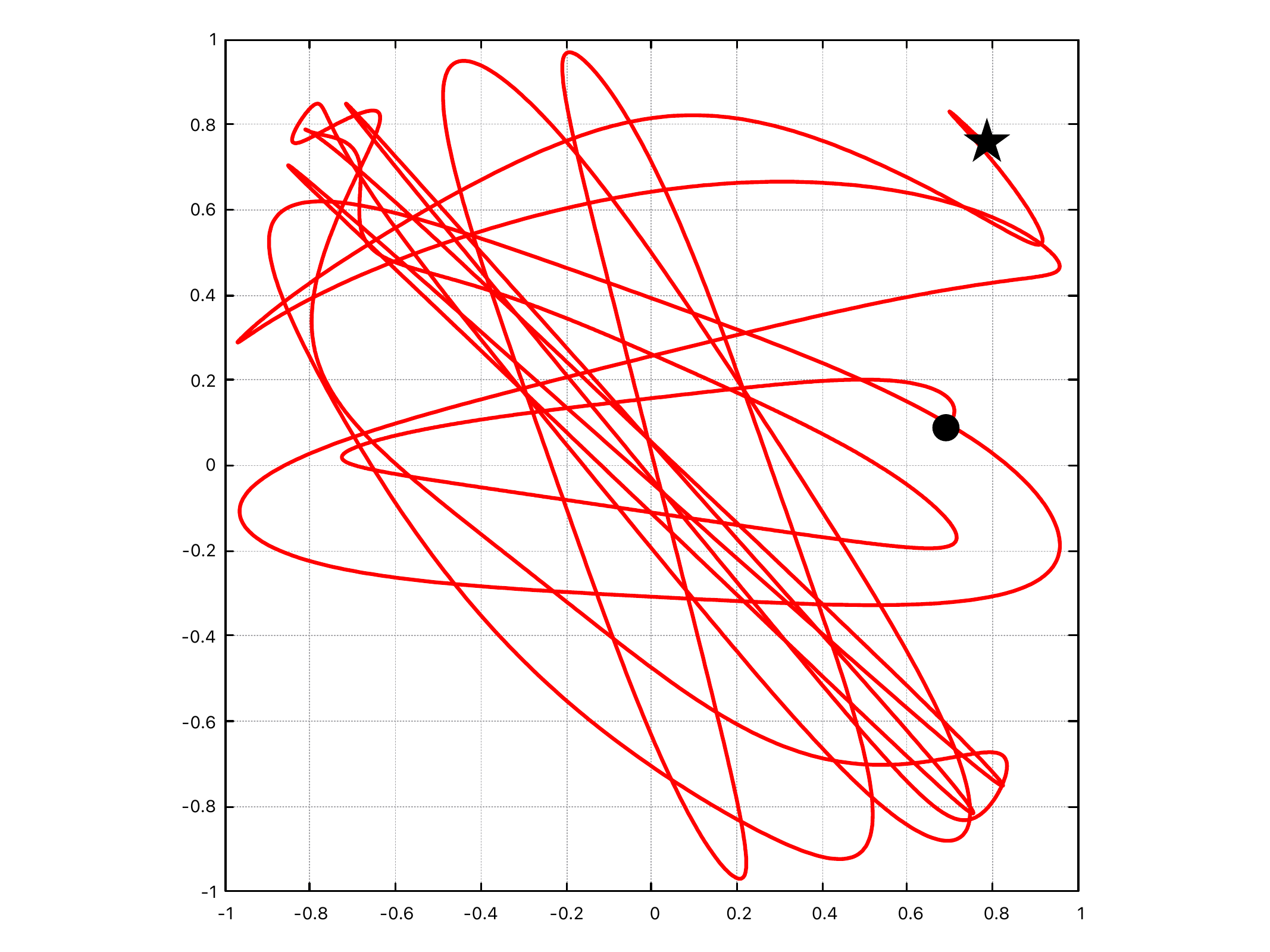}
\caption{Trajectory for (\ref{WWL_system_k1})}
\label{fig-WWL_k1}
\par
Projection of the trajectory onto the $(x_1, x_3)$-plane is drawn. 
Black dot shows the initial point ${\bf x}_0$, while black star shows the converged point ${\bf x}_\ast$.
\end{figure}

\begin{table}[ht]\em
\centering
{
\begin{tabular}{c|c|c|c}
\hline
$t_0$ & $t_{\max}$ & $x_1x_3$ & $\sin (t_{\max})$ \\
\hline
$0.0$ & $32.8825988$ & $-$ & $0.994583983$ \\  [1mm]
$0.02$ & $27.2058393$ & $-$ & $0.876476743$ \\  [1mm]
$0.04$ & $26.8604357$ & $-$ & $0.987716715$ \\  [1mm]
$0.06$ & $32.3622666$ & $-$ & $0.811281139$ \\  [1mm]
$0.08$ & $58.0551668$ & $-$ & $0.997933643$ \\  [1mm]
$0.10$ & $64.9226436$ & $-$ & $0.867822074$ \\  [1mm]
$0.12$ & $167.810591$ & $+$ & $-0.965193131$ \\  [1mm]
$0.14$ & $64.9566839$ & $-$ & $0.850408826$ \\  [1mm]
$0.16$ & $101.590941$ & $-$ & $0.872343878$ \\  [1mm]
$0.18$ & $27.2550679$ & $-$ & $0.851723643$ \\  [1mm]
$0.20$ & $27.3562087$ & $-$ & $0.794464435$ \\  [1mm]
\hline 
\end{tabular}%
}
\caption{$t_0$, $t_{\max}$, $x_1x_3$ and $\sin (t_{\max})$ for (\ref{WWL_system})}
\flushleft
\label{table-WWL}
\end{table}%

\begin{table}[ht]\em
\centering
{
\begin{tabular}{c|c|c|c}
\hline
$t_0$ & $t_{\max}$ & $x_1x_3$ & $\sin (t_{\max})$ \\
\hline
$0.0$ & $16.8997936$ & $-$ & $-0.929047680$ \\  [1mm]
$0.02$ & $23.1796559$ & $+$ & $-0.927813165$ \\  [1mm]
$0.04$ & $17.0537921$ & $-$ & $-0.97480137$ \\  [1mm]
$0.06$ & $17.0394695$ & $-$ & $-0.971506471$ \\  [1mm]
$0.08$ & $16.8828818$ & $-$ & $-0.922658414$ \\  [1mm]
$0.10$ & $16.6462128$ & $-$ & $-0.806524463$ \\  [1mm]
$0.12$ & $16.7658949$ & $-$ & $-0.871342465$ \\  [1mm]
$0.14$ & $16.7562709$ & $-$ & $-0.866579900$ \\  [1mm]
$0.16$ & $16.9183831$ & $-$ & $-0.935764124$ \\  [1mm]
$0.18$ & $17.0743196$ & $-$ & $-0.979174827$ \\  [1mm]
$0.20$ & $17.3945135$ & $-$ & $-0.993307996$ \\  [1mm]
\hline 
\end{tabular}%
}
\caption{$t_0$, $t_{\max}$, $x_1x_3$ and $\sin (t_{\max})$ for (\ref{WWL_system_k1})}
\flushleft
\label{table-WWL_k1}
\end{table}%


\end{document}